\documentclass[11pt,a4paper]{article}
\usepackage{latexsym,amssymb,amsmath,amsthm,amsfonts,enumerate,verbatim,xspace,
exscale}

\input xy
\xyoption{all} \CompileMatrices \UseComputerModernTips

\usepackage{graphicx,tabularx}
\usepackage{amsmath,amsbsy,amsfonts,amssymb}
\usepackage{xcolor}
\usepackage{amsthm}
\usepackage{mathtools}
\usepackage{lineno}
\title{First Integrals and symmetries of nonholonomic systems}

\author{P. Balseiro  \footnotemark}
\author{{\sc{Paula Balseiro}\thanks{
         Universidade Federal Fluminense, Instituto de Matem\'atica, Rua Mario Santos Braga S/N, 24020-140,
        Niteroi, Rio de Janeiro, Brazil. \newline{\texttt{E-mail: pbalseiro@id.uff.br}}}
        }  \ \
        {\sc{Nicola Sansonetto}\thanks{
         Universit\`a degli Studi di Verona, Dipartimento di Informatica, strada le Grazie 15, 37134 Verona, Italy.
\newline{\texttt{E-mail: nicola.sansonetto@univr.it}}}} 
}

\theoremstyle{plain}
\newtheorem{theorem}{Theorem}[section]
\newtheorem{lemma}[theorem]{Lemma}
\newtheorem{proposition}[theorem]{Proposition}
\newtheorem{corollary}[theorem]{Corollary}
\newtheorem*{theorem*}{Theorem}
\newtheorem{remarkth}[theorem]{Remark}
\newtheorem{remarkths}[theorem]{Remarks}

\theoremstyle{definition}
\newtheorem{definition}[theorem]{Definition}
\newtheorem{example}[theorem]{Example}
\newtheorem*{ConditionsA}{Conditions ${\mathcal A}$}
\newtheorem*{ConditionA4}{Condition $({\mathcal A}4)$}
\newenvironment{remark}{\begin{remarkth}\upshape}{\hfill$\diamond$\end{remarkth}}
\newenvironment{remarks}{\begin{remarkths}\upshape}{\hfill$\diamond$\end{remarkths}}

\newcommand{\g}{\mathfrak{g}}

\newcommand{\note}[1]{
\begin{minipage}[c]{0.86\textwidth} \tiny {\bf Note:} #1
\end{minipage}}

\def\W{\mathcal{W}}
\def\M{\mathcal{M}}

\def\V{\mathcal{V}}
\def\S{\mathcal{S}}
\def\C{\mathcal{C}}

\def\R{\mathbb{R}}

\def\red{{\mbox{\tiny{red}}}}
\def\nh{{\mbox{\tiny{nh}}}}

\def\B{{\mbox{\tiny{$B$}}}}
\def\subW{{\mbox{\tiny{$\W$}}}}

\def\subS{{\mbox{\tiny{$\S$}}}}
\def\subM{{\mbox{\tiny{$\M$}}}}
\def\subQ{{\mbox{\tiny{$Q$}}}}

\newcommand{\cA}{\mathcal{A}}

\newcommand{\cJ}{\mathcal{J}}

\newcommand{\cM}{\mathcal{M}}

\newcommand{\cO}{\mathcal{O}}
\newcommand{\cP}{\mathcal{P}}

\newcommand{\cU}{\mathcal{U}}

\newcommand{\bR}[1]{\mathbb{R}^{#1}}

\newcommand{\bI}[1]{\mathbb{I}^{#1}}

\newcommand{\bJ}{\mathbb{J}}

\newcommand{\bT}[1]{\mathbb{T}^{#1}}

\def\vecOm{\boldsymbol{\Omega}}

\def\vecL{\boldsymbol{\lambda}}

\def\vecgamma{\boldsymbol{\gamma}}
\def\vecalpha{\boldsymbol{\alpha}}
\def\vecbeta{\boldsymbol{\beta}}

\begin{document}
\maketitle

\begin{abstract}

 In nonholonomic mechanics, the presence of constraints in 
   the velocities breaks the well-under\-stood link between symmetries and first 
   integrals of holonomic systems, expressed in Noether's Theorem.
   However there is a known special class of first integrals of nonholonomic systems generated by vector
   fields tangent to the group orbits, called {\it horizontal gauge momenta}, that suggest that some version of this link should still hold.
   In this paper we give sufficient conditions for the existence of horizontal gauge momenta; our analysis leads to a constructive 
   method and a precise estimate of their number, with fundamental consequences to the  integrability of some nonholonomic systems as well as their hamiltonization.  
   We apply our results to three paradigmatic examples:
   the snakeboard, a solid of revolution rolling without sliding on a plane and 
   a heavy homogeneous ball that rolls without sliding inside a convex surface of revolution.    
   In particular, for the snakeboard we show  the existence of a new 
   horizontal gauge momentum that reveals new aspects of its integrability. 
   
\end{abstract}

\tableofcontents

\section{Introduction}

%%%%%%%%%%%%%%%%%%%%%%%%%%%%%%%%%%%%%%%%%%%%%%%%%%%%%%%%%%%%%%

% Intro has changed.  Original version in the file 2408202 in the folder ``Old versions''.

%%%%%%%%%%%%%%%%%%%%%%%%%%%%%%%%%%%%%%%%%%%%%%%%%%%%%%%%%%%%%%%
\subsection{Symmetries and first integrals}

The existence of first integrals plays a fundamental role in the study of dynamical systems and 
it influences many aspects of their behavior, in particular their integrability.
It is well-known that in holonomic systems with symmetries 
(described by a suitable action of a Lie group), Noether Theorem ensures that the components of the momentum map are first integrals of the dynamics. 
When we impose constraints in the velocities, we obtain the so-called {\it nonholonomic systems} \cite{Routh,NF,Bloch,CDS}: mechanical systems on a manifold $Q$ where the permitted velocities define a nonintegrable constant-rank distribution $D\subset TQ$ on $Q$.  
One way to see the non lagrangian/hamiltonian character 
of these systems is that the presence of symmetries does not necessarily
lead to first integrals  (see \cite{NF,CKSB,marle95,BKMM,CdLdDM,Sniatycki98,marle2001,
Bloch,cortes,zenkov2003,FRS2007,CDS,FS2010}); in particular, the components of the momentum map 
need not be conserved by the dynamics. %\cite{CKSB,marle95,BKMM,CdLdDM,marle2001,FRS2007,FS2010}, 
On the other hand, it has been observed that there are many first integrals linear in 
the momenta that are generated by vector fields that are not infinitesimal generators of the symmetry
action, but are still tangent to the group orbits
\cite{BGM96,zenkov2003,FGS2008,FGS2012,BS2016}. 
% In nonholonomic  mechanics the situation is different: there are symmetries and first integrals, but it is not clear if there exists a systematic way to link them, since Noether Theorem does not hold anymore because of the non-variational character of the nonholonomic dynamics \cite{LM94}. 

The research of a possible link between the presence of symmetries  and the existence 
of first integrals in nonholonomic systems --if any exists-- has been an active field of research in the last thirty years \cite{BS93,CKSB,marle95,BGM96,BKMM,CdLdDM,Sniatycki98,marle2001,zenkov2003,FRS2007,FS2010}, 
and it dates back at least to the fifties with the work of Agostinelli \cite{agostinelli1956} 
and fifteen years later with the works of Iliev \cite{iliev1,iliev2}.
More recently new tools and techniques, with a strong relation with the symmetries of the system, have been introduced in order to understand the dynamical and geometrical aspects of nonholonomic systems, such as nonholonomic momentum map, momentum equations, and gauge momenta. 
In the present paper, we investigate the existence of first integrals of the nonholonomic dynamics coming from the presence of symmetries using these tools and the so-called {\it gauge method}, introduced in \cite{BGM96} 
and further developed in \cite{FGS2008,FGS2009, FGS2012}.

\subsection{Main results of the paper}
Given a nonholonomic system with a symmetry described by the (free and proper) 
action of a Lie group $G$,  we consider functions of  type $J_\xi = \langle J,\xi\rangle$, where 
$J$ is the canonical momentum map and $\xi$ is a section of the bundle $Q\times \g \to Q$, with the property that 
the infinitesimal generator of each $\xi(q)$, $q\in Q$, is tangent to the constraint distribution.
Theorem \ref{T:Main} gives conditions on nonholonomic systems ensuring that the presence 
of symmetries induces the existence of first integrals of type $J_\xi$, called {\it horizontal gauge momenta} 
(while the section $\xi$ is called a {\it horizontal gauge symmetry})
\cite{BGM96}. Denoting by $k$ the rank of the distribution $S$ given by the intersection of the 
constraint distribution $D$ with the tangent space to the $G$-orbits, we characterize the 
nonholonomic systems that admit exactly $k$ horizontal gauge momenta that are functionally 
independent and $G$-invariant. Precisely, we write an explicit 
{\it system of linear ordinary differential equations} whose solutions give rise to the 
$k$ horizontal gauge momenta.    

These results are based on an intrinsic {\it momentum equation} that characterizes the horizontal gauge momenta.   
We also show that this intrinsic momentum equation can be regarded as a parallel transport equation,
that is,  we prove that a horizontal gauge symmetry $\xi$ is a parallel section along the nonholonomic 
dynamics on $Q$, with respect to an affine connection defined on (a subbundle of)  $Q\times\g\to Q$. 
This affine connection arises by adding to the Levi-Civita connection a bilinear form that 
carries the information related to the system of differential equations determining the horizontal gauge momenta.

%Our main result, Theorem~\ref{T:Main}, \textcolor{blue}{gives two main contributions}: 
%the first one shows that certain nonholonomic systems admit exactly $k$ $G$-invariant
%{\it horizontal gauge momenta} (i.e., first integrals linear in the momenta induced by the 
%symmetries), where $k$ is the rank of the distribution given by the intersection of the 
%nonholonomic constraints with the tangent space to the $G$-orbits (recall that for a 
%hamiltonian system, Noether theorem guarantees $n$ first integrals where $n$ is the 
%rank of the tangent space to the $G$-orbits). The second aspect is constructive: it says how to 
%explicitly construct such first integrals. 

The fact that we know the exact number of horizontal gauge momenta and have a systematic 
way of constructing them has fundamental consequences on the geometry and dynamics 
of nonholonomic systems, see e.g. \cite{hermans,zenkov1995,BMK2002,FGS2005,FG2007,CDS,balseiro2017,GNMontaldi}.
% 
% 
% Having the information about the existence of horizontal gauge momenta 
% and an explicit method to find them gives
% a detailed picture of the dynamics of the reduced and complete system, as well as its hamiltonization as observed in 
%
Under the hypotheses of Theorem~\ref{T:Main} we first show that 
the reduced dynamics is integrable by quadratures and, if some compactness issues are satisfied, 
it is indeed periodic (Theorem~\ref{T:reduced-integrability}). 
From a more geometric point of view, if the reduced dynamics is periodic,  we have that 
the reduced space inherits the structure of an $S^1$-principal bundle outside the equilibria. 
Second, we prove (Theorem~\ref{T:BalYapu19}) the {\it hamiltonization} of these nonholonomic systems 
(see also \cite{GNMontaldi,BalYapu19}); precisely  
the existence of $k = \textup{rank}(S)$ horizontal gauge momenta and 
the fact that $\textup{dim}(Q/G)=1$ guarantee the existence of a {\it Poisson} bracket on 
the reduced space $\M/G$ that describes the reduced dynamics.  This bracket is 
constructed using a {\it dynamical gauge transformation by a 2-form} that we also 
show to be related to the {\it momentum equation}. 
Third, when the reduced dynamics is periodic, we can obtain information on the
complete dynamics (Theorem~\ref{T:reconstruction}). In particular, if the symmetry group $G$ 
is compact, the reconstructed dynamics is quasi-periodic on tori of dimension at most 
$r+1$, where $r$ is the rank of the Lie group $G$, and the phase space inherits the structure 
of a torus bundle. If the symmetry group is not compact, the situation is less simple, but 
still understood: the complete dynamics is either quasi-periodic or
diffeomorphic to $\bR{}$, and whether one or the other case is more frequent or generic depends 
on the symmetry group (see Section~\ref{Sec:c-integrability}, 
Appendix~\ref{app:rec} and \cite{AM1997,FPZ2020}).

{\small\begin{table}[h]\begin{center} 
\begin{tabular}{c|c|c|c}
  \textcolor{blue}{\bf System} & \textcolor{blue}{\bf Symmetry} & \textcolor{blue}{\bf rank$(S)$} &
  \textcolor{blue}{\bf{\scriptsize{ \begin{tabular}{c} $\sharp$ horizontal \\ gauge momenta \end{tabular} }}}  \\
  \hline
  &&&\\
  {\bf Nonholonomic oscillator} & $\bT{2}$ & 1 & \textcolor{red}{\bf 1} \\
  \hline
%  &&&\\
%  Vertical Disk & $\bR{2}\times\bT{2}$ & 2 & \textcolor{red}{\bf 2} & \textcolor{blue}{\bf 2} \\
  &&&\\
  Vertical Disk & $\textrm{SE}(2)\times S^1$ & 2 & \textcolor{red}{\bf 2} \\
  \hline
  &&&\\
  Tippe--top & $\textrm{SE}(2)\times S^1$ & 2 & \textcolor{red}{\bf 2}  \\
  \hline
  &&&\\
  Falling disk & $\textrm{SE}(2)\times S^1$ & 2 & \textcolor{red}{\bf 2}  \\
  \hline
  &&&\\
  {\bf Snakeboard} & $\textrm{SE}(2)\times S^1$ & 2 & \textcolor{red}{\bf 2} \\  
  \hline
  &&&\\
  {\bf Body of revolution} & $\textrm{SE}(2)\times S^1$ & 2 & \textcolor{red}{\bf 2}  \\
  \hline
  &&&\\
  Ball in a cylinder & $\textrm{SO}(3)\times S^1$ & 2 & \textcolor{red}{\bf 2} \\
  \hline
  &&&\\
  {\bf Ball in a cup/cap} & $\textrm{SO}(3)\times S^1$ & 2 & \textcolor{red}{\bf 2}     \\
  \hline  
  &&&\\
  Ball in a surface of revolution & $\textrm{SO}(3)\times S^1$ & 2 & \textcolor{red}{\bf 2}     \\  
%  & & & $J_1,\; J_2$ casi  \\
\end{tabular}
\end{center}
\caption{Nonholonomic systems and related horizontal gauge momenta with respect to the symmetry.}
\label{tab:table}
\end{table}
}

Table~\ref{tab:table}  
shows how  many classical examples of nonholonomic systems fit into the scheme of 
Theorem~\ref{T:Main} and also puts in evidence the relation between the $\textup{rank}(S)$ 
and the number of horizontal gauge symmetries as stated in the theorem.  
%We finally observe that many of the classical examples of nonholonomic systems satisfy the required  hypotheses \textcolor{red}{see Table~\ref{Tab:table}, where in the first column we list the nonholonomic systems under consideration, in the second one the symmetry group, then the dimension of the space $S$ and eventually the corresponding number of horizontal gauge momenta}.
We study in detail four of these examples: the nonholonomic oscillator,  the snakeboard, a solid of revolution rolling on a plane 
and a heavy homogeneous ball rolling on a surface of revolution. 
In particular, the last two examples are paradigmatic of a large class of nonholonomic systems with symmetry.
In the case of the snakeboard, we find two horizontal gauge momenta, one of which, 
as far as we know, has not appeared in the literature before. 
We use this fact to prove the integrability by quadrature of the reduced system and its hamiltonization.   
Then we investigate what happens 
in certain examples when the hypotheses of Theorem \ref{T:Main} are not satisfied.  In these cases, using the intrinsic momentum equation, it is still possible to find horizontal gauge momenta (in some cases, less than $k$ of them).

%Moreover, we give also a geometric interpretation of the system of ordinary differential equations by showing that the momentum equation is a parallel transport equation. That is, we prove that a horizontal gauge symmetry $\zeta$ can be found as a parallel transport of an element $\zeta_0\in (\g_S)_{q_0}$,  along the nonholonomic dynamics on $Q$, with respect to an affine connection defined on the bundle $\g_S\to Q$. This affine connection arises as a modification of the Levi-Civita connection plus a bilinear form that carries the information related to the system of differential equations determining the horizontal gauge momenta. 

\subsection{Outline of the paper}
The paper is organized as follows: in Section 2 we recall the basic aspects and notations of nonholonomic 
systems and horizontal gauge momenta.   In Section \ref{S:MainResult} we present an intrinsic formulation of the
momentum equation and the main result of the paper, Theorem \ref{T:Main}. The results of this Section are illustrated with the example of the nonholonomic oscillator. 
The fundamental consequences of Theorem \ref{T:Main}, integrability and hamiltonization, are studied in Section 
\ref{Sec:Consequences}.  
Finally, in Section~\ref{S:Examples} we first apply our 
techniques and results to three paradigmatic examples outlined in bold in Table~\ref{tab:table}. 
Moreover, we also study different cases where the hypotheses of Theorem \ref{T:Main} 
are not satisfied. The paper is complemented by two appendices: App. \ref{A:Hamiltonization}  recalls basic defintions regarding almost Poisson brackets and gauge tranformations, and App. \ref{app:rec} presents basic facts about reconstruction theory.
Throughout the work, we assume that all objects (functions, manifolds, distributions, etc) 
are smooth. 
Moreover, unless stated otherwise, we consider Lie group actions that are free and proper or 
we confine our analysis in the submanifold where the action is free and proper. 
Finally, whenever possible, summation over repeated indices is understood.

{\bf Acknowledgement:}  P.B. would like to thank University of Padova and Prof. F. Fass\`o for the kind hospitality during her visit and to  CNPq (Brazil) for financial support. N.S. thanks IMPA and Prof. H. Bursztyn, PUC-Rio and Prof. A. Mandini for the kind hospitality during all his visits in Rio de Janeiro. P.B. and N.S. also thank
F. Fass\`o and A. Giacobbe for many interesting and useful discussions on finite dimensional 
non-Hamiltonian integrable systems and Alejandro Cabrera and Jair Koiller for their insightful comments.

\section{Initial Setting: Nonholonomic systems and horizontal gauge momenta}\label{sec:HGM}

\subsection{Nonholonomic systems with symmetries}\label{Ss:InitialSetting}
A nonholonomic system is a mechanical system on a configuration manifold $Q$ with (linear) 
constraints in the velocities.  The permitted velocities are represented by a nonintegrable constant-rank
distribution $D$ on $Q$.  
A nonholonomic system, denoted by the pair $(L,D)$, is given by 
a manifold $Q$, a lagrangian function $L:TQ\to \R$ of mechanical type, i.e., $L = \kappa - U$ 
for $\kappa$ and $U$ the kinetic and potential energy respectively, and a nonintegrable 
distribution $D$ on $Q$.  We now write the equations of motion of such systems following \cite{BS93}. 

Since the lagrangian $L$ is of mechanical type, the Legendre transformation $Leg:TQ \to T^*Q$ 
defines the submanifold $\M := Leg(D)$ of $T^*Q$.  Moreover, since $Leg$ is linear on the fibers, 
$\tau_\subM:= \tau|_\M : \M\to Q$ is also a subbundle of $\tau: T^*Q \to Q$, where $\tau$ 
denotes canonical projection.  Then,  if $\Omega_Q$ denotes the canonical 2-form on $T^*Q$ and $H$ 
the hamiltonian function induced by the lagrangian $L$, we denote by 
$\Omega_\subM := \iota^*\Omega_Q$ and $H_\subM := \iota^* H$ the 2-form and the 
hamiltonian on $\M$, where $\iota : \M\to T^*Q$ is the natural inclusion.
We define the (noningrable) distribution $\C$ on $\M$ given, at each $m\in\M$, by 
\begin{equation}\label{Def:C}
\C_m := \{v_m \in T_m\M \ : T\tau_\subM(v_m) \in D_q \mbox{ for } q = \tau_\subM(m)\}. 
\end{equation}

The nonholonomic dynamics is then given by the integral curves of the vector field $X_\nh$ on $\M$, 
taking values in $\C$ (i.e., $X_\nh(m) \in \C_m$) such that 
\begin{equation}\label{Eq:NHDyn}
{\bf i}_{X_\nh} \Omega_\subM |_\C = dH_\subM |_\C,
\end{equation}
where $\Omega_\subM |_\C$ and $dH_\subM|_\C$ are the point-wise restriction of the forms to $\C$. 
It is worth noticing that the 2-section $\Omega_\subM |_\C$ is nondegenerate and thus we have 
a well defined vector field $X_\nh$ satisfying \eqref{Eq:NHDyn}, called the {\it nonholonomic vector field}.
%of $H_\cM$.  %Moreover, the 2-section $\Omega_\subM|_\C$ and the nonintegrable distribution $\C$ define the {\it nonholonomic bracket} $\{\cdot, \cdot\}_\nh$ or equivalently the {\it nonholonomic bivector field} $\pi_\nh$ on $\M$.  Thus, the nonholonomic system defined by $L$ and $D$ is also determined by the triple $(\M, \pi_\nh , \Ham_\subM)$. 

On the hamiltonian side we will denote a nonholonomic system by the triple 
$(\M, \Omega_\subM|_\C, H_\subM)$.   %or $(\M, \{\cdot, \cdot\}_\nh, H_\subM)$.

\medskip

\noindent {\bf Symmetries of a nonholonomic system.}   We say that an action of a Lie group $G$ 
on $Q$ defines a {\it symmetry} of the nonholonomic system $(L,D)$  if it is 
free and proper and its tangent lift leaves $L$ and $D$ invariant. 

%\begin{remark} %\marginpar{\textcolor{blue}{\scriptsize{Rmk.2.1  I added a comment in the intro. Do you think we should erase the remark? I would leave it...}}} 
%\textcolor{red}{ Observe that in the definition of $G$-symmetry, we are assuming that the action is free and proper. 
%\textcolor{red}{
% This is because the theory studied in this article can be only applied to free and proper actions and the examples will be analyzed in the submanifold where the action satisfies these properties. } \marginpar{\textcolor{red}{In the intro, when we write the organisation of the paper we can state explicitly this assumption, without repeating it anymore}}
%\end{remark}

%In this article we will mainly consider analysis on symmetries given by a Lie group acting {\it freely} and properly on $Q$.
%We will then call {\it nonholonomic system $(L,D)$  or $(\M, \{\cdot, \cdot\}_\nh, H_\subM)$ with a $G$-symmetry}  every nonholonomic system $(L,D)$ or  $(\M, \{\cdot, \cdot\}_\nh, H_\subM)$, in which a Lie group $G$ acts freely and properly leaving both the lagrangian or the hamiltonian and the constraint manifold invariant.
Let $\g$ be the Lie algebra associated to the Lie group $G$. At each $q\in Q$, 
we denote by $V_q \subset T_qQ$ the tangent space to the $G$-orbit at $q$, 
that is $V_q := \textup{span} \{ \eta_Q(q) : \eta \in \g   \}$, 
where $\eta_Q (q)$ denotes the infinitesimal generator of $\eta$ at $q$. 
 
The lift of the $G$-action to the cotangent bundle $T^*Q$ leaves also the submanifold 
$\M \subset T^*Q$ invariant, hence there is a well defined $G$-action on $\M$ denoted by 
$\Psi:G\times \M \to \M$.  
The hamiltonian function $H_\subM$ and the 2-section $\Omega_\subM|_\C$ are 
$G$-invariant and we say that $(\M, \Omega_\subM|_\C, H_\subM)$ is a 
{\it nonholonomic system with a $G$-symmetry}. We denote by $\mathcal{V}_m \subset T_m\M$ the tangent space to the $G$-orbit at $m\in\M$ (i.e., $\mathcal{V}_m =\{ \eta_\subM(m) : \eta\in\g\}$).

\begin{definition}[\cite{BKMM}]\label{Def:DimAssum} A nonholonomic system $(\M, \Omega_\subM|_\C, H_\subM)$
with a $G$-symmetry verifies the {\it dimension assumption} if, for each $q\in Q$,  
\begin{equation}\label{Eq:dim-assumption}
T_qQ = D_q + V_q. 
\end{equation}
\end{definition}
Equivalently, the dimension assumption can be stated as $T_m\M = \C_m + \mathcal{V}_m$ for each $m\in\M$.

At each $q\in Q$, we define the distribution $S$ over $Q$ whose fibers are 
$S_q := D_q\cap V_q$ and the distribution $\g_S$ over $Q$ with fibers 
\begin{equation}\label{Def:gS}
(\g_S)_q = \{\xi^q\in \g \ : \ \xi_Q(q) \in S_q\},
\end{equation}
where $\xi_Q(q):= (\xi^q)_Q(q)$. Due to the dimension assumption \eqref{Eq:dim-assumption}, 
$\g_S\to Q$ is a vector subbundle of $Q\times \g\to Q$ and, if the action is free then $\textup{rank}(S) = \textup{rank}(\g_S)$ (see \cite{balseiro2017}).  
During this article, we denote by $\Gamma(\g_S)$ the {\it sections} of the bundle $\g_S\to Q$.

\medskip

\noindent {\bf Reduction by symmetries.} If  $(\M, \Omega_\subM|_\C, H_\subM)$ is a nonholonomic 
system with a $G$-symmetry, the nonholonomic vector field $X_\nh$ is $G$-invariant, i.e., 
$T\Psi_g(X_\nh(m)) = X_\nh(\Psi_g(m))$ with $\Psi_g :\M\to\M$ the $G$-action on $\M$ and $g\in G$, 
and hence it can be reduced to the quotient space $\M/G$. More precisely, denoting by $\rho:\M\to \M/G$ 
the orbit projection, the reduced dynamics on $\M/G$ is described by the integral curves of 
the vector field 
\begin{equation} \label{Eq:RedDyn}
X_\red := T\rho(X_\nh).
\end{equation}

%Our next goal is to obtain a system of ordinary differential equations so that if $(f_1,...,f_k)$ is a solution of such a system then ${\mathcal J}=f_iJ_i$ is a horizontal gauge momentum for the nonholonomic system. 
%For that purpose,  we write a {\it momentum equation} to characterize the horizontal gauge momenta.  Such a momentum equation is based on a decomposition of the tangent bundle.

\medskip

\noindent {\bf Splitting of the tangent bundle.} The dimension assumption ensures the existence of a  {\it vertical complement $W$ of the constraint distribution} $D$ (see \cite{balseiro2014}), that is, $W$ is a distribution on $Q$ so that
\begin{equation}\label{Eq:TQ=D+W}
TQ = D \oplus W \qquad \mbox{where} \qquad W \subset V.
\end{equation}

A vertical complement $W$ also induces  a splitting of the vertical space $V=S\oplus W$.  
Moreover, there is a one to one correspondence between the choice of an $Ad$-invariant 
subbundle $\g_W \to Q$ of $\g\times Q\to Q$ such that, at each $q\in Q$, 
\begin{equation}\label{Eq:g=gs+gw}
(\g\times Q)_q = (\g_S)_q \oplus (\g_W)_q,
\end{equation}
%with $(\g_W)_q  = \{\eta^q \in \g \ : \ \eta_Q(q) \in W_q\}$, for each $q\in Q$,  
and the choice of a $G$-invariant vertical complement of the constraints $W$. 
\begin{remark}
  If the $G$-action is free, the existence of a $G$-invariant vertical complement $W$ 
  is guaranteed by choosing $W = S^\perp \cap V$, where $S^\perp$ denotes the orthogonal 
  complement of $S$ with respect to the ($G$-invariant) kinetic energy metric (however $W$ 
  does not have to be chosen in this way). 
  In the case of non-free actions, as anticipated in the Introduction, 
  we restrict our study to the submanifold $\widetilde Q$ of 
  $Q$ where the action is free (see Examples \ref{Ex:Solids} and 
  \ref{Ex:BallSurface})\footnote{If the action is not free, it can be proven that for compact Lie groups $G$ (or the product of a compact Lie group and a vector space), the dimension assumption guarantees that it is always possible to choose a $G$-invariant vertical complement $W$, \cite{balseiro2017}.}.
\end{remark}

%\begin{comment}
%\begin{remark}[vertical symmetries] We say that a system with symmetries satisfies the vertical symmetry condition if the subbundle $\g_W$ is induced by a Lie subalgebra $\mathfrak{w}$ of $\g$, that is $\g_W \simeq \mathfrak{w}\times Q$ (\cite{Jac, Bal2017}). In this case, the bundle $\g_S \to Q$ is automatically trivial. In fact, considering a basis $\{\eta_1,...,\eta_l\}$ of $\mathfrak{w}$ and completing this basis so that 
%$\{ \tilde\eta_1,...,\tilde\eta_k,\eta_1,...,\eta_l\}$ is a basis of $\g$, then we have a (global) basis of $\g_S$ given by 
%\[
%  \g_S = \textup{span}\{ \xi_1:=P_{\g_S}(\tilde\eta_1),..., \xi_k := P_{\g_S}(\tilde\eta_k)\}, 
%\]   
% where $P_{\g_S} : \g \to \g_S$ is the projection associated to the decomposition \eqref{Eq:g=gs+gw}.  \marginpar[right]{Where??}
%\end{remark}
%\end{comment}

Next, we pull back the decomposition \eqref{Eq:TQ=D+W} to $\M$.  % At each $m\in\M$, we denote by $\V_m \subset T_m\M$ the vertical 
%space associated to the $G$-orbit on $\M$, i.e., $\V_m = \textup{span}\{\eta_\subM(m) \ : \ \eta\in\g\}$ 
%where $\eta_\subM(m)$ denotes the infinitesimal generator of $\eta$ on $m$. 
From \eqref{Eq:TQ=D+W} and  \eqref{Def:C} we obtain the corresponding decomposition on $T\M$, 
\begin{equation} \label{Eq:TM=C+W}
T\M = \C \oplus \W \qquad \mbox{with} \qquad \W\subset \V,
\end{equation}
where, at each $m \in \M$,  $\W_m = \{ (\xi^q)_\subM(m) \ : \ \xi^q \in (\g_\subW)_q \mbox{ for } q = \tau_\subM(m)\}$.  We define the distribution $\S = \C \cap \V$ or equivalently, for each $m\in \M$, 
\[
  \S_m = \{(\xi^q)_\subM(m) \ : \  \xi^q \in (\g_S)_q \ \mbox{for } q = \tau_\subM(m)\}.
\]
%Let us denote by $P_\C : T\M \to \C$ and $P_\subW : T\M \to \W$ the projection to the first and second factor associated to the splitting $T\M = \C\oplus \W$ respectively. 

%\begin{lemma}\cite{BalYapu19} \label{L:gsBundle}
%If the dimension assumption is satisfied then $\g_S\to Q$ is a subbundle of the trivial bundle $Q\times \g \to Q$. 
%\end{lemma}

\subsection{Horizontal gauge momenta}\label{Ss:HGM}

Consider a nonholonomic system $(\cM,\Omega_\subM|_\C,H_\subM)$ with a $G$-symmetry and recall that $\Theta_\subM$ is the Liouville 1-form restricted to $\M$ (i.e., $\Theta_\subM := \iota^*\Theta_Q$).
It is well known that for a nonholonomic system, an element $\eta$ of the Lie algebra does not necessarily induce
a first integral of the type $\mathcal{J}_\eta$ (see \cite{FS2010} for a discussion of this fact). 
%Even more, a section $\zeta\in\Gamma(\g_S)$ does not necessarily induce a first integral $\mathcal{J}_\zeta$.

\begin{definition}[\cite{BGM96,FGS2008}] \label{Def:HGM} A function ${\mathcal J} \in C^\infty(\M)$ is a {\it horizontal gauge momentum} if there exists $\zeta \in \Gamma(\g_S)$ such that ${\mathcal J}= {\mathcal J}_\zeta := {\bf i}_{\zeta_\M} \Theta_\subM$ and also ${\mathcal J}$ is a first integral of the nonholonomic dynamics $X_\nh$, i.e., $X_\nh(\mathcal{J}) = 0$. In this case, the section $\zeta \in \Gamma(\g_S)$ is called {\it horizontal gauge symmetry}.   
\end{definition}
 
We are interested in looking for horizontal gauge momenta of a given nonholonomic system with symmetries satisfying the 
dimension assumption. Looking for a horizontal gauge momentum $\mathcal{J}$ is equivalent 
to look for the corresponding horizontal gauge symmetry.

\begin{remark} 
The original definition of {\it horizontal gauge momentum} introduced in \cite{BGM96} (and later in \cite{FGS2008,FGS2012}) was not exactly as in Definition~\ref{Def:HGM} but given in local coordinates. % given in local coordinates as a function ${\mathcal J}$ on $\M$ of the form ${\mathcal J} =\langle f_iY_i, p\rangle$, where $\{Y_i\}_{i=1,...,k}$ form a basis of $S$ and $f_i$ are functions on $Q$. %Recall that the $Y_i$ are combination of infinitesimal generators associated to the $G$-action, that is, $Y_i = g_i^n (\nu_n)_Q$ where $g_i^n \in C^\infty(Q)$ and $\nu_n \in \g \ \forall n$.  
\end{remark}

The {\it nonholonomic momentum map} (\cite{BKMM}) $J^\nh:\M\to \g_S^*$ is the bundle map over the identity, given, for each $m\in \M$ and $\xi \in \g_S |_m$, by 
\begin{equation}\label{Eq:NHMomMap}
\langle J^\nh, \xi\rangle (m) =  {\bf i}_{\xi_\M} \Theta_\subM (m).
\end{equation}
Hence, a horizontal gauge momentum can also be seen as a function of the type $\langle J^\nh, \zeta\rangle \in C^\infty(\M)$ that is a first integral of $X_\nh$.  

\begin{proposition} 
A nonholonomic system $(\M, \Omega_\subM|_\C, H_\subM)$ with a $G$-symmetry 
satisfying the dimension assumption admits, at most, $k=\textup{rank}(S)$ (functionally 
independent) horizontal gauge momenta.   
\end{proposition}
\begin{proof}
Consider $\xi_1, \xi_2\in \Gamma(\g_S)$. It is easy to see that if $J_1={\bf i}_{(\xi_1)_\subM}\Theta_\subM$ and $J_2={\bf i}_{(\xi_2)_\subM}\Theta_\subM$ are functionally independent functions then $\xi_1, \xi_2$ are linearly independent.
\end{proof}

Observe that the existence of a horizontal gauge momentum, implies the existence of a global section on $\g_S\to Q$. Hence, in order to prove that a nonholonomic system admits exactly $k$ horizontal gauge symmetries, we have to assume the triviality of the bundle $\g_S \to Q$, that is,  $\g_S\to Q$ admits a global basis of sections that we denote by 
\begin{equation}\label{Eq:Basis_gS}
\mathfrak{B}_{\g_S} = \{ \xi_1,..., \xi_k\}.
\end{equation}
 The basis $\mathfrak{B}_{\g_S}$ induces functions $J_1, ..., J_k$ on $\M$ (linear on the fibers) defined by 
\begin{equation}\label{Eq:J_i}
J_i := \langle J^\nh,\xi_i\rangle = {\bf i}_{(\xi_i)_\M} \Theta_\subM \qquad \mbox{for } i=1,...,k. 
\end{equation}
If ${\mathcal J} \in C^\infty(\M)$ is a horizontal gauge momentum with $\zeta$ its associated horizontal gauge symmetry, then $\mathcal{J}$ and $\zeta$ can be written, with respect to the basis \eqref{Eq:Basis_gS}, as
\begin{equation}\label{Eq:DefJi}
{\mathcal J} = f_i J_i \qquad \mbox{and} \qquad \zeta = f_i \xi_i, \quad \mbox{for} \ f_i\in C^\infty(Q).
\end{equation}
We call the functions $f_i$, $i=1,...,k$ the {\it coordinate functions of} $\mathcal{J}$ 
with respect to the basis $\mathfrak{B}_{\g_S}=\{\xi_1,...,\xi_k\}$.

From now, if not otherwise stated, we assume the following conditions on the symmetry given by the action of the Lie group $G$.%\footnote{However from time to time, we will remark what happens if one or more of them are not satisfied (specially Condition $\cA3$).} 

\medskip

\begin{ConditionsA} We say that a nonholonomic system with a $G$-symmetry 
satisfies {\it Conditions} $\cA$ if 
\begin{enumerate}
  \item[$(\cA1)$] the dimension assumption \eqref{Eq:dim-assumption} is fulfilled; 
  \item[$(\cA2)$] the bundle $\g_S\longrightarrow Q$ is trivial; 
  \item[$(\cA3)$] the action of $G$ on $Q$ is proper and free.
  \end{enumerate}
\end{ConditionsA}

A section $\xi$ of the bundle $Q\times \g \to Q$ is $G$-{\it invariant} if $[\xi,\eta]=0$ for all $\eta \in \g$.
As a consequence of Conditions $\cA$ we obtain the following Lemma.

\begin{lemma}\label{L:invariance} Consider a nonholonomic system with a $G$-symmetry 
satisfying Conditions $\cA$, then  
 \begin{enumerate}
 \item[$(i)$] there exists a global basis $\mathfrak{B}_{\g_S}$ of $\Gamma(\g_S)$ given by $G$-invariant sections.
  \item[$(ii)$] Let $\xi\in \Gamma(\g_S)$. The function $J_\xi = {\bf i}_{\xi_\subM}\Theta_\subM$ is $G$-invariant if and only if $\xi\in \Gamma(\g_S)$ is $G$-invariant.
  \item[$(iii)$]  Let $\rho_{\subQ} :Q\to Q/G$ be the orbit projection associated to the $G$-action on $Q$. 
 If $X\in \mathfrak{X}(Q)$ is $\rho_\subQ$-projectable, then $[X, \xi_Q] \in \Gamma(V)$, for $\xi \in \Gamma(Q\times \g\to Q)$.
 \end{enumerate}

\end{lemma}
\begin{proof}
 Items $(ii)$ and $(iii)$ were already proven in \cite[Lemma~3.8]{BalYapu19}. 
 To prove item $(i)$ observe that items ($\cA2$) and ($\cA3$) imply that $S$ admits a global basis of $G$-invariant sections $\{Y_1,...,Y_k\}$, i.e., $[Y_i, \nu_Q]=0$ for all $\nu\in\g$.  %recall that $S$ is $G$-invariant and thus the bundle $\g_S\to Q$ is $Ad$-invariant.  Let us consider a global $Ad$-invariant basis of $\g_S\to Q$ denoted by $\{\xi_1,...,\xi_k\}$.  Then $T \psi_g\left( (\xi_i)_Q \right) = (Ad_g \xi_i)_Q = (\xi_i)_Q$, which means that $Y_i$ are $G$-invariant vector fields, i.e., $[Y_i,\nu_Q]=0$ for all $\nu\in \g$. Therefore $([\xi_i, \nu])_Q(q) = 0$ for all $q\in Q$ which means that
 Since the action is free, we conclude that, for $(\xi_i)_\subQ=Y_i$  we have that  $[\xi_i,\nu]\in \Gamma(Q\times \g)$ is the zero section and thus $\xi_i$ are $G$-invariant.  
 \begin{comment}
 \textcolor{blue}{
 \note{I managed to write this proof without assuming the freeness of the the $G$-action (as we discussed that day at impa).  However, I have to use some condition on the regularity of this action: something like given $\eta\in \Gamma(Q\times \g)$, if $\eta_Q(q)=0$ for all $q\in Q$ then $\eta\equiv0$ (that is, it is the zero-section). Does it make sense?}
 }
 \end{comment}
\end{proof}

Under Conditions $\cA$, we guarantee the existence of a global $G$-invariant basis $\mathfrak{B}_{\g_S}$ of sections of $\g_S\to Q$ with associated $G$-invariant functions $J_i$ (defined as in \eqref{Eq:DefJi}). Hence, $\mathcal{J}$ is a $G$-invariant horizontal gauge momentum if and only if the corresponding coordinate functions $f_i$ in \eqref{Eq:DefJi} are $G$-invariant as well.

\section{A momentum equation} \label{S:MainResult}

\subsection{An intrinsic momentum equation}\label{Ss:MomEq}

In order to achieve our goal of giving a precise estimate of the number of (functionally independent) horizontal gauge momenta of a nonholonomic system, we write a {\it momentum equation}. % that characterizes the coordinate functions $f_i$, $i=1,...,k$. 
%thus determining \textcolor{red}{the} horizontal gauge symmetries of the nonholonomic dynamics $X_\nh$.\marginpar{\textcolor{red}{the?}}  
Let $(\cM,\Omega_\subM|_\C,H_\subM)$ be a nonholonomic system with a 
$G$-symmetry satisfying Conditions $\cA$.  
First, we consider a decomposition  (or a principal connection)
\begin{equation}\label{Eq:TQ=H+V}
TQ= H \oplus V \qquad \mbox{so that} \qquad H\subset D.
\end{equation}
We denote by $A: T\M \to \g$ the connection 1-form  such that $\textup{Ker}A = H$. 
Since the vertical space $V$ is also decomposed as $V = S \oplus W$, 
the connection  $A$ can be written as $A = A_S + A_W$, where, for each 
$X\in TQ$, $A_W : TQ\to \g$ is given by 
$$
A_W(X) = \eta \qquad \mbox{if and only if} \qquad \eta_\subQ = P_W(X),
$$
and  $A_S : TQ\to \g$ is given by 
\begin{equation}\label{Eq:DefAs}
A_S(X) = \xi \qquad \mbox{if and only if} \qquad \xi_\subQ = P_S(X),
\end{equation}
where $P_W :TQ \to W$ and $P_S :TQ \to S$ are the corresponding projections associated to decomposition 
\begin{equation}\label{Eq:TQ=H+S+W} 
TQ = H \oplus S \oplus W. 
\end{equation}

Second, we see that each map $A_S$ and $A_W$ defines a corresponding 2-form on 
$Q$ in the following way (see \cite{balseiro2014}): on the one hand, the {\it $W$-curvature}
on $Q$ is a $\g$-valued 2-form defined, for each $X, Y\in TQ$, as
\[
 K_W(X,Y) = d^DA_W (X,Y) = dA_W (P_D(X), P_D(Y)) = -A_W ([P_D(X), P_D(Y)]),
\]
with $P_D: TQ = D\oplus W \to D$ the projection to the first factor.  On the other hand, 
after the choice of a global 
%$G$-invariant 
basis $\mathfrak{B}_{\g_S} = \{\xi_1,...,\xi_k\}$ 
of $\g_S\to Q$,  the $\g$-valued 1-form $A_S$ on $\M$  can be written as 
\[
A_S = {Y}^i \otimes \xi_i,
\]
where ${Y}^i$ are 1-forms on $Q$ such that  ${Y}^i|_{H} = {Y}^i|_{W}=0$ and 
${Y}^i((\xi_j)_\subQ)=\delta_{ij}$  for all $i=1,...,k$ (recall that the sum over 
repeated indexes is understood). Then the corresponding $\g$-valued 2-form 
is given, for each $X,Y\in TQ$, by 
\[
  (d^DY^i)\otimes \xi^i (X,Y) =  dY^i(P_D(X), P_D(Y)) \otimes \xi^i.
\]

Recalling that $\tau_\subM:\M\to Q$ is the canonical projection, we define the $\g$-valued 2-forms $\bar{\sigma}_{\g_S}$ and $\sigma_{\g_S}$ on $Q$ and $\M$ respectively, by 
\begin{equation}\label{Def:sigma}
\begin{split}
\bar{\sigma}_{\g_S} & : = K_W + d^D Y^i\otimes \xi_i,\\ 
\sigma_{\g_S} & := \tau_\subM^* \bar{\sigma}_{\g_S}\,. 
\end{split}
\end{equation}
Equivalently, $\sigma_{\g_S}$ is given by $\sigma_{\g_S} = \mathcal{K}_\subW + d^\C {\mathcal Y}^i\otimes \xi_i$, where $\mathcal{K}_\subW = \tau_\subM^*K_W$, $\mathcal{Y}^i = \tau_\subM^*Y^i$ and $d^\C \mathcal{Y}^i ({\mathcal X},\mathcal{Y}) = d \mathcal{Y}^i (P_\C(\mathcal{X}), P_\C(\mathcal{Y}))$ for $\mathcal{X}, \mathcal{Y} \in T\M$, and $P_\C:T\M\to \C$ the projection associated to decomposition \eqref{Eq:TM=C+W}.

\begin{definition}\label{Def:Jsigma}
Consider a nonholonomic system $(\cM,\Omega_\subM|_\C,H_\subM)$ with a $G$-symmetry
satisfying Conditions $\cA$ and denote by $\mathfrak{B}_{\g_S} = \{\xi_1,...,\xi_k\}$ 
a global 
%$G$-invariant 
basis of $\Gamma(\g_S)$. The 2-form $\langle J, \sigma_{\g_S} \rangle$ on $\M$ is defined by
\begin{equation*}
 \begin{split}
\langle J, \sigma_{\g_S} \rangle := & \ \langle J, \mathcal{K}_\subW \rangle + \langle J ,  d^\C \mathcal{Y}^i \otimes \xi^i\rangle, \\
:= & \ \langle J, \mathcal{K}_\subW \rangle + J_i \, d^\C \mathcal{Y}^i,
\end{split}
\end{equation*}
where $J:\M\to \g^*$ is the canonical momentum map restricted to $\M$ and 
$\langle \cdot, \cdot\rangle$ denotes the pairing between $\g^*$ and $\g$.
\end{definition}

The 2-form $\langle J, \sigma_{\g_S} \rangle$ already appeared in \cite{BalYapu19} for a specific choice of the basis $\mathfrak{B}_{\g_S}$ (see Sec.~\ref{Ss:Hamiltonization}).

\begin{lemma}\label{L:sigma}
 Assume that Conditions $\cA$ are satisfied, then
 \begin{enumerate}
  \item[$(i)$] The $\g$-valued 2-forms $\bar{\sigma}_{\g_S}$ and  $\sigma_{\g_S}$ depend on the chosen basis $\mathfrak{B}_{\g_S}$. 
 % \item[$(ii)$] $\bar{\sigma}_{\g_S}$ and $\sigma_{\g_S}$ are $ad$-equivariant as long as the basis $\mathfrak{B}_{\g_S}$ is chosen $G$-invariant. 
  \item[$(ii)$] If the basis $\mathfrak{B}_{\g_S}$ is $G$-invariant, then the 2-form  $\langle J, \sigma_{\g_S} \rangle$ is $G$-invariant as well.   
 \end{enumerate}
\end{lemma}

\begin{proof}
 It is straightforward to see that the $\g$-valued 2-forms $\bar{\sigma}_{\g_S}$ and  $\sigma_{\g_S}$ depend directly on the chosen basis $\mathfrak{B}_{\g_S}$. Item $(ii)$ is proven in \cite[Lemma 3.8]{BalYapu19}. 
 
\end{proof}

\begin{proposition}\label{Prop:MomEq1}\textup{(Momentum equation)} Let us consider a nonholonomic 
system $(\cM,\Omega_\subM|_\C,H_\subM)$ with a $G$-symmetry satisfying Conditions $\cA$, and let 
$\mathfrak{B}_{\g_S} = \{\xi_1,...\xi_k\}$ be a (global) basis of $\Gamma(\g_S)$ with associated momenta $J_1,...,J_k$ as in \eqref{Eq:J_i}.  
The function ${\mathcal J}=f_iJ_i$, for $f_i\in C^\infty(Q)$, is a horizontal 
gauge momentum if and only if the coordinate functions $f_i$ satisfy the momentum equation
\begin{equation}\label{Eq:MomEq}
f_i \langle J, \sigma_{\g_S}\rangle ({\mathcal Y}_i, X_{\emph\nh}) + J_i X_{\emph\nh}(f_i) = 0,
\end{equation}
where ${\mathcal Y}_i := (\xi_i)_\subM$. 
\end{proposition}

\begin{proof}
First, from Lemma \ref{L:invariance} observe that if $\mathcal{X}$ is a vector field on $\M$ that is $T\rho$-projectable, then $[\mathcal{Y}_i, \mathcal{X}]\in \Gamma(\V)$ for $i=1,...,k$.  Thus, using \eqref{Def:sigma}, 
\begin{equation*}
 \begin{split}
\sigma_{\g_S} (\mathcal{Y}_i, \mathcal{X}) & = [d^\C \tau_\subM^* A_W  +   d^\C \tau_\subM^* Y^j \otimes \xi_j ] ( \mathcal{Y}_i,  \mathcal{X})  = - \tau_\subM^* A_W ([\mathcal{Y}_i,  \mathcal{X}]) - \tau_\subM^*Y^j([\mathcal{Y}_i,  \mathcal{X}]) \otimes \xi_j \\
& = -\tau_\subM^* A([\mathcal{Y}_i,  \mathcal{X}]).
 \end{split}
\end{equation*}
Second, by the definition of the canonical momentum map $J:\M\to \g^*$, we get that 
$$ 
\langle J, \sigma_{\g_S} \rangle (\mathcal{Y}_i, \mathcal{X}) = - \langle J, \tau_\subM^* A ([\mathcal{Y}_i, \mathcal{X}]) \rangle = - {\bf i}_{[\mathcal{Y}_i, \mathcal{X}]} \Theta_\subM.
$$  
Then, recalling that  $\Omega_\subM = -d\Theta_\subM$ and using that $\Theta_\subM(\mathcal{X})$ is an invariant function,  we observe that
\begin{equation*}%\label{Proof:MomEq1}
 \begin{split}
(\Omega_\subM + \langle J, \sigma_{\g_S}\rangle )(\mathcal{Y}_i, \mathcal{X}) & = -{\mathcal Y}_i( \Theta_\subM(\mathcal{X})) + \mathcal{X}(J_i) + \Theta_\subM([\mathcal{Y}_i, \mathcal{X}]) - {\bf i}_{[\mathcal{Y}_i, \mathcal{X}]} \Theta_\subM  = dJ_i (\mathcal{X}).
 \end{split}
\end{equation*}
Now, $\mathcal{J}=f_iJ_i$ is a first integral of $X_\nh$ if and only if $0 = d\mathcal{J}(X_\nh) = f_idJ_i(X_\nh) + J_i X_\nh(f_i)$ which is equivalent, for $X_\nh = \mathcal{X}$, to
$$
0 = f_i (\Omega_\subM + \langle J, \sigma_{\g_S}\rangle) (\mathcal{Y}_i, X_\nh) + J_i X_\nh(f_i) = -f_i dH_\subM(\mathcal{Y}_i) + f_i\langle J, \sigma_{\g_S}\rangle (\mathcal{Y}_i, X_\nh) + J_i X_\nh(f_i).
$$
Using the $G$-invariance of the hamiltonian function $H_\subM$ we get \eqref{Eq:MomEq}.
 
 \end{proof}

\begin{remark} \label{R:MomEq2}
From the proof of Proposition~\ref{Prop:MomEq1}, we observe that the {\it momentum equation} 
can be equivalently written  as $0 =  f_i  \Theta_\subM([(\xi_i)_\subM, X_{\nh}]) - J_iX_{\nh}(f_i)$.
\end{remark}

In the light of Proposition~\ref{Prop:MomEq1} (or more precisely Remark~\ref{R:MomEq2}), 
we recover the well-known 
result that horizontal symmetries generate first integrals \cite{BS93,BKMM}. 
Recall that a {\it horizontal symmetry} is an element $\eta\in\g$ such that $\eta_Q\in \Gamma(D)$
(see e.g. \cite{Bloch}).
\begin{corollary}[Horizontal symmetries] \label{C:HorSym1}
Let $(\M, \Omega_\subM|_\C, H_\subM)$ be a nonholonomic system with a $G$-symmetry 
satisfying Conditions $\cA$.  If the bundle $\g_S\to Q$ admits a horizontal symmetry $\eta$, then the function $\langle J, \eta\rangle$ is a horizontal gauge momentum for the nonholonomic system. Hence if there is global basis of horizontal symmetries of $\g_S$, then  the nonholonomic 
system admits $k = \textup{rank}\, (\g_S)$ horizontal gauge momenta.
\end{corollary}
\begin{proof}
If $\eta_1$ is a horizontal symmetry, then let $\mathfrak{B}_{\g_S} = \{\eta_1,\xi_2,...,\xi_k\}$ a basis of $\Gamma(\g_S)$. A section $\zeta =f_1\eta_1 + f_i \xi_i$ is a horizontal gauge symmetry if 
$J_1 X_\nh(f_1) + f_i  \Theta_\subM([ X_{\nh}, (\xi_i)_\subM]) + J_iX_{\nh}(f_i) = 0$, since $[X_\nh , \eta_1]=0$.  Then we see that $f_1 = 1$ and $f_i=0$ for $i=2,...,k$ is a solution of the momentum equation and hence $\eta_1$ is a horizontal gauge symmetry. 
As a consequence, if the bundle $\g_S\to Q$ admits a basis of horizontal symmetries, then the nonholonomic admits $k$ horizontal gauge momenta. 
\end{proof}

 %Observe that the {\it momentum equation} in Proposition~\ref{Prop:MomEq1} is quadratic in the fibers. 
 A set of solutions $(f_1,...,f_k)$ of the momentum equation \eqref{Eq:MomEq}  may depend on $\M$ and not only on 
 $Q$.  Based on the fact that the equation \eqref{Eq:MomEq} is quadratic in the fibers,  we show next that it is equivalent to a 
 system of partial differential equations for the functions $f_i$ on the manifold $Q$.  

 \subsection{The ``strong invariance'' condition on the kinetic energy}
 
We now introduce and study an invariance property, called strong invariance, 
that involves the kinetic energy, the constraints and the G-symmetry.
This condition is crucial to state our main result in Theorem \ref{T:Main}.  
 
\begin{definition} 
Consider a Riemannian metric $\kappa$ on a manifold $Q$ and a distribution 
$S\subset TQ$ on $Q$. %, i.e., for all $q\in Q$, $S_q \subset  V_q$
The metric $\kappa$ is called {\it strong invariant on $S$} (or $S$-{\it strong invariant}) if for all $G$-invariant sections
$Y_1, Y_2, Y_3 \in \Gamma(S)$, holds that
$$
\kappa(Y_1,[Y_2, Y_3]) = - \kappa(Y_3,[Y_2, Y_1]).
$$
\end{definition}

First we observe that, for a Riemannian metric $\kappa$, being $G$-invariant is weaker than being strong invariant on the whole tangent bundle as the following example shows:

\begin{example}
{\bf The case $Q =G$ with a strong invariant metric on $TG$.}   Consider a Lie group $G$ acting on itself with the left action and let $\kappa_G$ be a Riemannian metric on it. In this case, the metric being $G$-invariant is equivalent to being left invariant, while being strong invariant on $TG$ is equivalent to being bi-invariant.  
 In fact,  if the metric is strong invariant on $TG$ then $\kappa_G([Y_i,Y_j],Y_l) = - \kappa_G(Y_j,[Y_i,Y_l])$ for all $Y_i\in \mathfrak{X}(G)$ such that $[Y_i, \eta^R]=0$ for all $\eta\in \g$ and $\eta^R$ the corresponding right-invariant vector field on $G$ (we are using that the infinitesimal generator associated to the left action is the corresponding right invariant vector field on $G$). Then, the inner product $\langle \cdot, \cdot \rangle$ on $\g$ defined by 
 $$
 \langle \eta_1, \eta_2\rangle = \kappa_G(\eta_1^L, \eta_2^L)(e), \qquad \mbox{for } \eta_i\in \g,
 $$
 is $ad$-invariant and hence the metric $\kappa_G$ turns out to be bi-invariant on $G$. 
\end{example}

%\textcolor{red}{I want to check if the strong invariance has something to do with the so-called $LR$ systems (dynamical systems on a compact Lie group $G$ with a left-invariant metric and right-invariant nonholonomic constraints ).  What do you think?}

\begin{example}\label{Ex:StrongV}{\bf A nonholonomic system with a strong invariant kinetic energy on the vertical distribution $V$}.  Consider a nonholonomic system 
 $(\cM,\Omega_\subM|_\C,H_\subM)$ with a $G$-symmetry.  If the kinetic energy metric $\kappa$ is strong invariant on $V$ then it induces 
 a bi-invariant metric on the Lie group $G$.
 This case only may occur when the group of symmetries $G$ is compact or a product 
 of a compact Lie group with a vector space. In order to prove this, we first observe that 
 \begin{lemma}\label{L:Ex:StrongV} The kinetic energy metric satisfies
  $\kappa([Y_i,Y_j],Y_l) = - \kappa(Y_j,[Y_i,Y_l])$ for all $Y_i\in \Gamma(V)$ $G$-invariant if and only if $\kappa([(\eta_a)_Q,(\eta_b)_Q],(\eta_c)_Q ) = - \kappa((\eta_b)_Q,[(\eta_a)_Q,(\eta_c)_Q])$ for all $\eta_i\in\g$. 
 \end{lemma}
 \begin{proof}
 The vertical distribution $V$ admits a basis of $G$-invariant sections $\{Y_1,...,Y_n\}$. For $\eta\in \g$, 
 there are functions $g^j\in C^\infty(Q)$, $j=1,...,n$ so that $\eta_Q = g^jY_j$  and hence 
 $0 = [Y_i, \eta_Q] = g^j[Y_i,Y_j] + Y_i(g^j)Y_j$. Then we obtain that
 \begin{equation*}
  \begin{split}
  \kappa([(\eta_a)_Q,(\eta_b)_Q],(\eta_c)_Q ) &= g_a^i g_b^jg_c^l \kappa([Y_i,Y_j],Y_l) + g_a^ig_c^l \kappa(Y_i(g_b^j)Y_j,Y_l) - g_b^jg_c^l\kappa(Y_j(g_a^i)Y_i,Y_l) \\
  & = - g_a^i g_b^jg_c^l \kappa([Y_i,Y_j],Y_l).
  \end{split}
 \end{equation*}
Conversely, we write $Y_i=g_i^a\eta_a$ and we repeat the computation. 
 \end{proof}
As a direct consequence of Lemma~\ref{L:Ex:StrongV}, \ if the kinetic energy is strong invariant on $V$, \, then 
 $\kappa([(\eta_a)_Q,(\eta_b)_Q],(\eta_c)_Q ) = - \kappa((\eta_b)_Q,[(\eta_a)_Q,(\eta_c)_Q])$ 
 for all $\eta_i\in\g$. Hence, for each $q\in Q$, there is an $ad$-invariant inner product on $\g$ 
 defined, at each $\eta_1, \eta_2\in \g$ by 
 $$
 \langle \eta_1, \eta_2\rangle_q = \kappa((\eta_1)_Q(q), (\eta_2)_Q(q)).
 $$
  Therefore, there exists a {\it family} of bi-invariant metrics $\kappa_G^q$ on $G$ 
  defined by $\kappa_G^q(\eta_1^L(g), \eta_2^L(g)) = \langle\eta_1, \eta_2\rangle_q$.  
\end{example}

\begin{example}\label{Ex:StrongAbelian}{\bf The symmetry group $G$ is abelian.}  Consider a 
nonholonomic system $(\cM,\Omega_\subM|_\C,H_\subM)$
with a $G$-symmetry, and let $G$ be an abelian Lie group, then the Lie algebra $\g$ is also abelian 
and the kinetic energy metric satisfies $\kappa([(\eta_1)_Q,(\eta_2)_Q],(\eta_3)_Q) = 0$ 
for all $\eta_i \in \g$. Following Example~\ref{Ex:StrongV}, we have also that 
$\kappa([Y_1,Y_2],Y_3) = 0$ for all $G$-invariant sections $Y_i$ on $V$ and 
hence the kinetic energy is trivially strong invariant on $V$. 
\end{example}

\begin{example} \label{Ex:HorSym1} {\bf Horizontal symmetries.}
Consider a 
nonholonomic system $(\cM,\Omega_\subM|_\C,H_\subM)$ with a $G$-symmetry
satisfying Conditions $\cA$ and with the bundle $\g_S\to Q$  admitting a global basis of 
$G$-invariant horizontal symmetries $\{\eta_1,...,\eta_k\}$ of the bundle $\g\times Q \to Q$. 
Then the vector space generated by the constant sections $\eta_i$ is an abelian subalgebra 
$\mathfrak{s}$ of $\g$ and the kinetic energy metric is strong invariant on $S$. 
\end{example}

\subsection{Determining the horizontal gauge momenta (in global coordinates)}\label{Ss:MomEq-Coord}

Consider a nonholonomic system $(\M, \Omega_\subM|_\C,H_\subM)$ with a $G$-symmetry satisfying Conditions $\cA$. From now on, we will also assume that the $G$-symmetry verifies that the manifold $Q/G$ has dimension 1 or equivalently the rank of any horizontal space $H$ defined as in \eqref{Eq:TQ=H+V} is 1. That is, we add a fourth assumption to Conditions $\cA$
\begin{ConditionA4} The $G$-symmetry satisfies that the manifold $Q/G$ has dimension 1. 
\end{ConditionA4}

Now, let us consider the horizontal distribution $H$ defined in \eqref{Eq:TQ=H+V}.
\begin{definition}\label{Def:S-orth} We say that $H$ is {\it $S$-orthogonal} if it is given by
$$
H := S^\perp \cap D,
$$
where the orthogonal space to $S$ is taken with respect to the kinetic energy metric.  
\end{definition}

The $S$-orthogonality of $H$ implies that $H$ is a $G$-invariant distribution while Condition\,$(\cA4)$  guarantees that it is trivial and thus it admits a ($G$-invariant) global generator.

Now, let $(\M, \Omega_\subM|_\C,H_\subM)$ be a nonholonomic system with a $G$-symmetry satisfying Conditions $(\cA 1)$-$(\cA 4)$ (that is, the $G$-symmetry satisfies Conditions $\cA$ and Condition $(\cA 4)$). Then there is a global $G$-invariant basis $\mathfrak{B}_{\g_S} = \{\xi_1,...,\xi_k\}$ of sections of $\g_S$ and, as usual, we denote $Y_i : = (\xi_i)_Q$ the corresponding sections on $S$. 
If we denote by $\rho_Q:Q\to Q/G$ the orbit projection and assuming that the horizontal 
space $H$ is $S$-orthogonal, then there exists a globally defined section $X_0$ generating 
the horizontal bundle $H\to Q$ that is $\rho_Q$-projectable. Hence $\{X_0, Y_1,...,Y_k\}$ defines a global basis of $D = H\oplus S$. Following splitting \eqref{Eq:TQ=H+S+W}, 
we also consider a (possible non global) basis 
$\{Z_1,...,Z_N\}$ of the vertical complement $W$ and we denote by $(v^0,v^1,...,v^k,w^1,...,w^N)$ 
the coordinates on $TQ$ associated to the basis 
\begin{equation}\label{Eq:BasisTQ}
\mathfrak{B}_{TQ} = \{X_0,Y_1,...,Y_k,Z_1,...Z_N\},
\end{equation}
(for short we write the coordinates $(v^0, v^i, w^a)$ associated to the basis 
$\mathfrak{B}_{TQ} = \{X_0, Y_j, Z_a\}$). 
If $\mathfrak{B}_{T^*Q} = \{X^0, Y^i, Z^a\}$ is the basis of $T^*Q$ dual to $\mathfrak{B}_{TQ}$, we denote by $(p_0,p_i,p_a)$ the induced coordinates on $T^*Q$.  
Then the constraint submanifold $\M$ is described as
\[
  \M = \{(q, p_0,p_i,p_a) \in T^*Q \ : \ p_a = \kappa_{a{\mbox{\tiny{$0$}}}} v^0 + \kappa_{aj}v^j\},
\]
where $p_0 = \kappa_{\mbox{\tiny{$00$}}} v^0 + \kappa_{{\mbox{\tiny{$0$}}}i}v^i$ and $p_i = \kappa_{i{\mbox{\tiny{$0$}}}} v^0 + \kappa_{ij}v^j$ with $\kappa_{\mbox{\tiny{$AB$}}} = \kappa(X_A, X_B)$ for $X_A,X_B\in \mathfrak{B}_{TQ}$ (i.e., $A,B\in \{0,i,a\}$). %$\kappa_{{\mbox{\tiny{$0$}}}a} = \kappa(X_0,Z_a)$, $\kappa_{ja} = \kappa(Y_j, Z_a)$ and $\kappa_{{\mbox{\tiny{$0$}}}j} = \kappa(X_0,Y_j)$.
We now define the dual basis 
\begin{equation}\label{Eq:BasisTMT*M}
\mathfrak{B}_{T^*\!\M} = \{ \mathcal{X}^0, \mathcal{Y}^i, \mathcal{Z}^a, dp_0,dp_i\} \qquad \mbox{and} \qquad \mathfrak{B}_{T\M} = \{ \mathcal{X}_0, \mathcal{Y}_i, \mathcal{Z}_a, \partial_{p_0},\partial_{p_i}\}
\end{equation}
of  $T^*\M$ and $T\M$ respectively, where $\mathcal{X}^0 = \tau_\subM^*X^0$, $\mathcal{Y}^i=\tau_\subM^*Y^i$, $\mathcal{Z}^a=\tau_\subM^*Z^a$. Observe that, by the $G$-invariance of $p_0$ 
and $p_i$, $\mathcal{Y}_i = (\xi_i)_\subM$ and, moreover, by \eqref{Eq:J_i}  
\[
J_i = {\bf i}_{\mathcal{Y}_i} \Theta_\subM =  p_i.
\]

We now write the {\it momentum equation} \eqref{Eq:MomEq} in (global) coordinates, 
defined by the basis $\mathfrak{B}_{TQ}$ in \eqref{Eq:BasisTQ}.

\begin{lemma}\label{L:MomEq-Coord}
Suppose that the $G$-symmetry satisfies Conditions\,$(\cA1)$-$(\cA4)$ and the horizontal distribution $H$ in \eqref{Eq:TQ=H+V} is $S$-orthogonal.
In coordinates associated to the basis \eqref{Eq:BasisTQ}, a function ${\mathcal J} \in C^\infty(\M)$ 
of the form ${\mathcal J} = f_iJ_i$ is a $G$-invariant horizontal gauge momentum of the nonholonomic system $(\M, \Omega_\subM|_\C, H_\subM)$
if and only if the coordinate functions $f_i\in C^\infty(Q)^G$ satisfy 
\begin{equation}\label{Eq:MomEq-Coord}
v^lv^j\left(\, f_i \kappa(Y_j,[Y_i, Y_l])\,  \right) + (v^0)^2 \left(\, f_i \kappa(X_0, [Y_i, X_0])\, \right) + v^0v^j P_{0j}=0,
\end{equation}
where $P_{0j} :=  f_i ( \kappa(Y_j, [ Y_i, X_0]) + \kappa(X_0,[Y_i, Y_j]) ) - \kappa_{ij}X_0(f_i)$.
\end{lemma}

\begin{proof}
We will show that \eqref{Eq:MomEq-Coord} is the coordinate version of the momentum equation \eqref{Eq:MomEq}. 
First, observe that the 2-form $\langle J, \sigma_{\g_S}\rangle$ is semi-basic with respect to the bundle $\tau_\subM:\M\to Q$. Let us denote by $\mathcal{X}_1$, $\mathcal{X}_2$ any element in the subset $\{{\mathcal X}_0, {\mathcal Y}_1,...,{\mathcal Y}_k\}$ of the basis $\mathfrak{B}_{T\M}$ in \eqref{Eq:BasisTMT*M}, and by $X_1:=T\tau_\subM(\mathcal{X}_1)$ and $X_2:=T\tau_\subM(\mathcal{X}_2)$ the corresponding elements in the basis of $\mathfrak{B}_{TQ}$. Then  we have
\begin{eqnarray*}
\langle J,\mathcal{K}_\subW\rangle (\mathcal{X}_1,\mathcal{X}_2) &= &  p_a\, dZ^a (X_1,X_2) = -p_a Z^a([X_1,X_2]) = -(\kappa_{{\mbox{\tiny{$0$}}} a}v^0 + \kappa_{ja} v^j)Z^a ([X_1,X_2]),\\
\langle J, d^\C{\mathcal Y}^i \otimes \xi_ i\rangle(\mathcal{X}_1,\mathcal{X}_2) &= & p_i d{\mathcal Y}^i(\mathcal{X}_1,\mathcal{X}_2) = -p_i Y^i ([X_1,X_2]) = -(\kappa_{{\mbox{\tiny{$0$}}} i} v^0 + \kappa_{ij} v^j)Y^i ([X_1,X_2]), \\ 
& = & -\kappa_{ij} v^jY^i ([X_1,X_2]), 
\end{eqnarray*}
since $\kappa_{{\mbox{\tiny{$0$}}}i} =0$ by the $S$-orthogonality of $H$. Using that $[X_1,X_2]\in \Gamma(V)$ (observe that $[Y_i,Y_j]\in \Gamma(V)$ since $V$ is integrable, and $[X_0,Y_i]\in \Gamma(V)$ since $X_0$ is $\rho_\subQ$-projectable, see Lemma \ref{L:invariance}) then $[X_1,X_2] = Z^a([X_1,X_2]) Z_a + Y^j([X_1,X_2]) Y_j$ and thus 
$$
\langle J, \sigma_{\g_S}\rangle(\mathcal{X}_1,\mathcal{X}_2) = -v^0\kappa(X_0, [X_1,X_2]) - v^j \kappa(Y_j, [X_1,X_2]).
$$
Second,  using that $T\tau_\subM(X_\nh(q,p)) = v^0X_0 + v^iY_ i$ (recall that $X_\nh$ is a {\it second order equation}) and also recalling that the functions $f_i$ are $G$-invariant on $Q$, we obtain that the {\it momentum equation} in Proposition~\ref{Prop:MomEq1} is written as 
$$
0  = f_i v^0\langle J, \sigma_{\g_S}\rangle(\mathcal{Y}_i, \mathcal{X}_0) + f_i v^j\langle J, \sigma_{\g_S}\rangle(\mathcal{Y}_i,\mathcal{Y}_j) + p_i v^0 X_0(f_i).
$$
Putting together the last two equations we obtain \eqref{Eq:MomEq-Coord}.
\end{proof} 

\begin{remark}
If the horizontal distribution $H$ is not chosen to be $S$-orthogonal, then the {\it momentum equation} \eqref{Eq:MomEq-Coord} is modified in one of the terms:
 $$
 v^lv^j\left(\, f_i \kappa(Y_j,[Y_i, Y_l])\,  \right) + (v^0)^2 \left(\, f_i \kappa(X_0, [Y_i, X_0]) - \kappa_{0i} X_0(f_i)\, \right) + v^0v^j P_{0j}=0.
 $$
In order to obtain the simplest form of the coordinate version of the {\it momentum equation}, we require the orthogonality condition between $H$ and $S$. 
\end{remark}

As a consequence of Lemma \ref{L:MomEq-Coord}, we can state the main result of the paper.  

%Consider the $G$-invariant basis \eqref{Eq:Basis_gS} of $\g_S$ and \eqref{Eq:BasisTQ} of $TQ$, where $(\xi_i)_\subQ = Y_i$.

%Recall that we denote by $\{Y_1,...,Y_k\}$ a (global) basis of $G$-invariant sections on $S$ and by $\xi_i$ the corresponding $G$-invariant sections on $\g_S$ such that $(\xi_i)_\subQ = Y_i$. Also, we denote by $X_0$ a $G$-invariant section on $H$.  

\begin{theorem}\label{T:Main}
Consider a nonholonomic system $(\cM,\Omega_\subM|_\C,H_\subM)$
with a $G$-symmetry satisfying Conditions $(\cA 1)$-$(\cA 4)$ and with a $S$-orthogonal horizontal space $H$.
%$\textup{rank} (TQ) - \textup{rank} (V) =1$, where $V$ is the tangent space to the $G$-orbits. 
Moreover assume that the kinetic energy metric is strong invariant on $S$ and that
 \[
   \kappa(X_0, [Y, X_0]) = 0
 \]  
 for  $X_0$  a $\rho$-projectable vector field on $Q$ taking values in  $H$ 
 and for all $Y\in\Gamma(S)$. Then 
 \begin{enumerate}
  \item[$(i)$] the system admits $k=\textup{rank}(S)$ $G$-invariant (functionally independent) 
  horizontal gauge momenta. 
  \end{enumerate}
  Moreover, let us consider a $G$-invariant basis $\mathfrak{B}_{\g_S} = \{\xi_1,...,\xi_k\}$ of $\g_S$,
  with $Y_i=(\xi_i)_Q$, and define the $G$-invariant functions $R_{ij}$ on $Q$ given by 
 \begin{equation}\label{Eq:R_ij}
 R_{ij} = \kappa^{il} [\kappa(Y_l, [ Y_j, X_0]) + \kappa(X_0,[Y_j, Y_l])],
 \end{equation}
 where $\kappa^{il}$ are the elements of the matrix $[\kappa|_S]^{-1}$ and $[\kappa|_S]$ 
 is the matrix given by the elements $\kappa_{il}$. 
If $\bar{X}_0$ is a globally defined vector field on $Q/G$ such that 
$T\rho_\subQ(X_0) = \bar{X}_0$, then
  \begin{enumerate}
  \item[$(ii)$] the $k$ solutions  $f^l=(\bar{f}^l_1,...,\bar{f}^l_k)$ for $l=1,...,k$ of the linear system of ordinary differential equations on $Q/G$ given by 
 \begin{equation}\label{Eq:ODESystem}
R_{ij}  \bar{f}_j- \bar{X}_0(\bar{f}_i) = 0,
 \end{equation}
 define $k$ (functionally independent)  $G$-invariant horizontal gauge momenta given by 
 $$
 {\mathcal J}^l = f^l_i J_i ,
 $$
 for $J_i= {\bf i}_{\xi_\subM}\Theta_\subM$ (the functions defined in \eqref{Eq:J_i}) and $f^l_i=\rho^*\bar{f}^l_i$, $l=1,...,k$.
 \end{enumerate}
\end{theorem}

\begin{proof}
Let us consider the $G$-invariant basis $\mathfrak{B}_{\g_S} = \{\xi_1,...,\xi_k\}$ in \eqref{Eq:Basis_gS} with $Y_i=(\xi_i)_Q$ for $i=1,...,k$ and the basis $\mathfrak{B}_{TQ}$ and $\mathfrak{B}_{T^*Q}$ in \eqref{Eq:BasisTQ}. Then, from Lemma \ref{L:MomEq-Coord} we have that ${\mathcal J}= f_iJ_i$, for $f_i\in C^\infty(Q)^G$ is a horizontal gauge momentum if and only if equation \eqref{Eq:MomEq-Coord} is satisfied.  
Since \eqref{Eq:MomEq-Coord} is a second order polynomio in the variables $(v^0, v^i)$, it is zero when its associated matrix is skew-symmetric, that is when
\begin{enumerate}
\item[$(i)$] $\kappa(Y_j,[Y_i, Y_l]) = - \kappa(Y_l,[Y_i, Y_j])$, for all $i,j,l=1,...,k$,
\item[$(ii)$] $f_i\kappa(X_0, [Y_i, X_0]) = 0$,
\item[$(iii)$] $P_{0j} = 0$, for all $j=1,...,k$. 
\end{enumerate}
First we observe that items $(i)$ and $(ii)$ are trivially satisfied by the hypotheses of the theorem (item $(i)$ is just the definition of strong invariance). Second, we prove that item $(iii)$ determines the system of ordinary differential equations \eqref{Eq:ODESystem} defining the $G$-invariant functions $f_i$.  

Let us define the matrix $[N]$ with entries $N_{lj} = \kappa(Y_l, [ Y_j, X_0]) + \kappa(X_0,[Y_j, Y_l])$ and $[\kappa|_S]$ the kinetic energy matrix restricted to $S$ (which is symmetric and invertible with elements $\kappa_{li}$).  Then, the condition $P_{0j} = 0$ is written in matrix form as 
$  [N] f =  [\kappa|_S]X_0(f)$ for $f = (f_1,...,f_k)^t$, which is equivalent to  $R.f = X_0(f)$ for $R$ the matrix with entries $R_{ij} = [\kappa|_S]^{il} N_{lj} $. Therefore, item $(iii)$ is satisfied if and only if the functions $f= (f_1,...,f_k)$ are a solution of the linear system of differential equations defined on $Q$
\begin{equation}\label{Eq:System}
R_{ij}f_j - X_0(f_i) = 0, \quad \mbox{for each } i=1,...,k.
\end{equation}

Since $X_0\in \Gamma(H)$ is $\rho_\subQ$-projectable, then there is a (globally defined) vector field $\bar{X}_0$ on $Q/G$ such that $T\rho_\subQ(X_0) = \bar{X}_0$. Moreover, $R_{ij}$ are also $G$-invariant functions ($\kappa, X_0$ and $Y_i$ are $G$-invariant), and thus we conclude that the system 
\eqref{Eq:System} is well defined on $Q/G$.  That is, \eqref{Eq:System} represents a (globally defined) linear system of $k$ ordinary differential equations for the functions $(\bar f_1,..., \bar f_k)$ on $Q/G$, that is written as 
\begin{equation}\label{Eq:SystemODE}
 R_{ij} \bar{f}_j - \bar X_0(\bar f_i) = 0, \quad \mbox{for each } i=1,...,k.
\end{equation}
where $R_{ij}$ are viewed here as functions on $Q/G$.  The system \eqref{Eq:SystemODE} admits $k$ independent solutions $ \bar{f}^l = (\bar f_1^l,..... , \bar f_k^l)$ for $l=1,...,k$. Moreover, $f^l=(f_1^l,..... , f_k^l)$ with $f_i^l = \rho^*(\bar f_i^l)$ are $k$ independent solutions of \eqref{Eq:System}  and hence ${\mathcal J}^l =  f_i^l J_i$ are (functionally independent) $G$-invariant horizontal gauge momenta for $l=1,...,k$.

It is important to note that item $(iii)$ is the only item determining the functions $f_i$, while the other two items are intrinsic conditions imposed on the nonholonomic system.

\end{proof}

\begin{remark}
  The momentum equation \eqref{Eq:MomEq-Coord} does not depend on the potential energy 
  function but only on the $G$-invariance of it. As a consequence, the horizontal gauge 
  momentum $\mathcal{J}$, defined from Theorem~\ref{T:Main}, is a first integral of 
  $(\M,\Omega_\subM|_\C,H_\subM = \kappa|_\subM + U)$ for {\it any G-invariant} potential 
  energy function $U$ on $Q$. Such a property, called weak-Noetherinity, has been first 
  observed and studied in \cite{FGS2008,FGS2009,FGS2012}.
\end{remark}

\begin{corollary}\label{C:Cond-Friends} 
Consider a  nonholonomic system $(\cM,\Omega_\subM|_\C,H_\subM)$ with a 
$G$-symmetry satisfying Conditions $(\cA 1)$-$(\cA 4)$ and with a strong invariant kinetic energy on $S$.  
If the horizontal space $H$, defined in \eqref{Eq:TQ=H+V}, is orthogonal to the vertical space 
$V$ (with respect to the kinetic energy metric), then the system admits automatically 
$k=\textup{rank}(S)$ $G$-invariant (functionally independent) horizontal gauge momenta. 
\end{corollary}

\begin{proof} If $V^\perp =  H$ then $H$ is $S$-orthogonal and also $\kappa(X_0, [Y_i, X_0]) = 0$ for all $i=1,...,k$.
Thus we are under the hypothesis of Theorem \ref{T:Main}.  
\end{proof}

%\textcolor{blue}{\note{Check the following Lemma because I do not remember well and I did not make the computations again... tired of our friend. Is it necessary????   }}

\begin{remark}\label{R:EquivConditions} 
 Since it is not always possible to choose $H = V^\perp$ with $H\subset D$, 
 in some examples we have to check that $\kappa(X_0, [X_0,Y])=0$ for all 
 $Y\in \Gamma(S)$. This condition is equivalently written as $\kappa(X_0, [X_0,Y_i])=0$ 
 for all $i=1,...,k$ where $Y_i = (\xi_i)_Q$ with $\xi_i$ elements of the $G$-invariant basis 
 $\mathfrak{B}_{\g_S}$ in \eqref{Eq:Basis_gS}, which is identically expressed as 
 $(\pounds_{X_0} \kappa )(X_0,Y_i) = 0$ or $\kappa(\nabla_{X_0}Y_i,X_0) = 0$ 
 for $\nabla$ the Levi-Civita connection associated to the kinetic energy metric. 
\end{remark}

\noindent{\bf Guiding Example: nonholonomic oscillator.} The nonholonomic oscillator describes a particle 
in $Q=S^1\times\R\times S^1$ with a Lagrangian given by $L=\frac{m}{2}(\dot x^2+ \dot y^2+ \dot z^2) - U(y)$ 
and constraints in the velocities $\dot z = y\dot x$. The constraint distribution is given by 
$D=\textup{span} \{Y:= \partial_x + y\partial_z, \partial_y\}$.  The Lie group 
$G=S^1\times S^1$ acts on $Q$ so that $V=\textup{span}\{\partial_x, \partial_z\}$ 
and leaves $D$ and $L$ invariant. Then $S=\textup{span}\{Y\}$ and the kinetic energy metric is trivially strong invariant on $S$ since  $\textup{rank} (S) =1$ (in fact, it is strong invariant on $V$, 
see Example~\ref{Ex:StrongAbelian}). Moreover, we see that 
$V^\perp = \textup{span}\{\partial_y \} \subset D$ and hence
defining the horizontal space $H:= V^\perp$, Corollary \ref{C:Cond-Friends} 
guarantees the existence of one $G$-invariant horizontal gauge momentum.

Next, we will follow Theorem \ref{T:Main} to compute the horizontal gauge momentum $\mathcal{J}$ for this example.    Let us consider the basis $\mathcal{B}_{TQ}=\{X_0 = \partial_y, Y = \partial_x + y\partial_z, \partial_z\}$ of $TQ$ with coordinates $(v^0, v^{\mbox{\tiny{$Y$}}}, v^z)$. Observe that this basis induces the vertical complement of the constraints $W=\textup{span}\{\partial_z\}$. Then on $T^*Q$  we have the dual basis $\mathcal{B}_{T^*Q} = \textup{span}\{dy , dx, \epsilon := dz-ydx\}$ with coordinates $(p_0, p_{\mbox{\tiny{$Y$}}}, p_z)$. The constraint submanifold $\M$ is given by $\M=\{x,y,z,p_0,p_{\mbox{\tiny{$Y$}}},p_z) \ : \ p_z = \frac{y}{1+ y^2} p_{\mbox{\tiny{$Y$}}}\}$.

Recall that $G$ acts on $Q$ defining a principal bundle $\rho_\subQ:Q\to Q/G$ so that $\rho_\subQ(x,y,z) = y$. 
The Lie algebra of the symmetry group is $\g=\R^2$ and $\g_S = \textup{span}\{\xi = (1,y)\}$ while $\g_W = \textup{span}\{(0,1)\}$. Following \eqref{Eq:J_i}, the element $\xi\in \Gamma(\g_S)$ defines the function $J_\xi := \langle J^\nh, \xi\rangle =p_{\mbox{\tiny{$Y$}}}$  and the horizontal gauge momentum will be written as ${\mathcal J}= f(y) p_{\mbox{\tiny{$Y$}}}$ ($f$ is already considered as a $G$-invariant function on $Q$). 
 
\noindent {\it The  momentum equation from Proposition~\eqref{Prop:MomEq1}:} 
The function ${\mathcal J}$ is a horizontal gauge momenta if and only if $f$ 
satisfies that $ f(y) \langle J, \sigma_{\g_S}\rangle (\xi_\subM, X_\nh) + 
p_{\mbox{\tiny{$Y$}}} X_\nh(f)= 0$.  Since $d^\C dx=0$ then $\langle J, 
d^\C dx \otimes \xi\rangle = 0$ and thus the momentum equation remains 
\begin{equation}%\label{Ex:NHPart:1} 
   f(y) \langle J, \mathcal{K}_\subW\rangle (\xi_\subM, X_\nh) +p_{\mbox{\tiny{$Y$}}} f'(y) = 0.    
\end{equation}

\noindent {\it The differential equation of Theorem~\ref{T:Main}:} Next, we write the momentum equation in coordinates as it is expressed \eqref{Eq:ODESystem}.  Since $\textup{rank}(S) =1$, the ordinary differential equation to be solved, for $f=f(y)$, is $R_{\mbox{\tiny{$YY$}}}f - f'=0$ for  
 $$
 R_{\mbox{\tiny{$YY$}}} = \tfrac{1}{\kappa(Y,Y)} \, \kappa(Y, [Y , \partial_y]) =  - \tfrac{y}{1+y^2}.
  $$
 Therefore, the solution of the ordinary differential equation 
 \begin{equation}\label{Eq:NHPart-ODE}
  \tfrac{y}{1+y^2} f + f' = 0,
 \end{equation}
 gives the (already known) horizontal gauge momenta ${\mathcal J}= \frac{1}{\sqrt{1+y^2}}p_{\mbox{\tiny{$Y$}}}$ (which in canonical coordinates gives ${\mathcal J} = \sqrt{1+y^2}\, p_x$).

\subsection{A geometric interpretation: horizontal gauge symmetries as parallel sections} \label{Ss:ParallelTransport}

In this section, we will see how a horizontal gauge symmetry can be constructed by 
parallel transporting an element $\xi_0\in (\g_S)_{q_0}$, for $q_0\in Q$, along the dynamics using a specific affine connection.  
Consider the splitting $TQ = H \oplus S \oplus W$ of the tangent bundle, in which we not only take 
the distribution $H$ to be $S$-orthogonal, but we also choose the vertical complement $W$ orthogonal to $S$:
 \begin{equation*} 
   W := S^\perp \cap V.
 \end{equation*} 

On the bundle $\g_S\to Q$, we define the affine connection $\widehat\nabla:\mathfrak{X}(Q) \times \Gamma(\g_S) \to \Gamma(\g_S)$ given, at each $X\in\mathfrak{X}(Q)$ and $\xi\in\Gamma(\g_S)$,  by 
\begin{equation}\label{Eq:AffConexion}
\widehat\nabla_X\, \xi := A_S( \nabla_X \, \xi_\subQ), 
\end{equation}
where $\nabla: \mathfrak{X}(Q)\times \mathfrak{X}(Q)\to \mathfrak{X}(Q)$ is the Levi-Civita connection with respect to the kinetic energy metric and $A_S:TQ \to \g$ is the bundle map defined in \eqref{Eq:DefAs}.  Observe that, since $\textup{Im}(A_S) = \g_S$, then $\widehat\nabla$ is well defined.

\begin{remark} It is straightforward to check that $\widehat\nabla$ is, in fact, an affine connection. Moreover, this connection is related with the {\it nonholonomic connection} restricted to the bundle $\g_S\to Q$ (see e.g., \cite{CCdLdD}).  
\end{remark}

Next, we will modify this affine connection using a {\it gauge transformation}.\footnote{In this case, the terminology gauge transformation is used to modify an affine connection using gauge theory, \cite{Naber,Nash,RS}.  In Section \ref{Ss:Hamiltonization} a gauge transformation is used to modify almost Poisson brackets.} Assuming Conditions $(\cA 1)$-$(\cA 4)$, we denote by $\mathfrak{B}_{\g_S} = \{\xi_1,...,\xi_k\}$ a global $G$-invariant basis of sections of the bundle $\g_S\to Q$ and we recall the basis $\mathfrak{B}_{TQ}$ and $\mathfrak{B}_{T^*Q}$ defined in \eqref{Eq:BasisTQ}:
\begin{equation}\label{Eq:BasisParallel}
\mathfrak{B}_{\g_S} = \{\xi_1,...,\xi_k\}, \qquad \mathfrak{B}_{TQ} = \{X_0,Y_i,Z_a\} \quad \mbox{and} \quad \mathfrak{B}_{T^*Q} = \{X^0, Y^i, Z^a\},
\end{equation}
where $Y_i = (\xi_i)_\subQ$ for $i=1,...,k$.

%and recall from \eqref{Def:sigma} the $\g$-valued 2-form $\bar{\sigma}_{\g_S}$ on $Q$ associated to the basis $\mathfrak{B}_{\g_S}$. This form, induces a $V$-valued 2-form given, at each $X,Y\in TQ$, by 
%$$
%\bar{\sigma}_V(X,Y) : = (\bar{\sigma}_{\g_S} (X,Y) )_\subQ.
%$$

\begin{definition} \label{Def:SigmaConnection}
The $\Sigma${\it-connection} is the affine connection 
$\overset{\textit{\tiny{$\Sigma$}}}{\nabla} : \mathfrak{X}(Q) \times \Gamma(\g_S) \to \Gamma(\g_S)$
defined, for $X\in \mathfrak{X}(Q)$ and $\zeta\in \Gamma(\g_S)$, by 
\[
  \overset{\textit{\tiny{$\Sigma$}}}{\nabla}_X \, \zeta := \widehat\nabla_X \, \zeta + \Sigma(X, \zeta_Q)
\] 
where $\widehat\nabla$ is the affine connection defined in \eqref{Eq:AffConexion} 
and $\Sigma$ is the $\g_S$-valued bilinear form $\Sigma = \Sigma^l \otimes \xi_ l$ 
where $\Sigma^l$ are the bilinear forms given, in the basis \eqref{Eq:BasisParallel}, by  
\[
  \Sigma^l = -(\widehat\Gamma_{0j}^l + R_{lj} ) X^0\otimes Y^j - \widehat\Gamma_{ij}^l Y^i \otimes Y^j,
\]
where $\hat\Gamma_{0j}^l$ and $\hat\Gamma_{ij}^l$ are the Christoffel symbols 
of the affine connection $\widehat\nabla$ and $R_{ij}$ are the functions defined in \eqref{Eq:R_ij}.
\end{definition}

\begin{remark}
The $\Sigma$-connection is still an affine connection since $\Sigma$ is a bilinear form, 
which does not need to be skew-symmetric. For short, we may write 
$\overset{\textit{\tiny{$\Sigma$}}}{\mathcal{\nabla}} := \widehat\nabla + \Sigma$ 
and observe that the $\Sigma$-connection is a {\it gauge covariant derivative}, \cite{Naber,Nash,RS}.   
\end{remark}

Next, we show that a horizontal gauge symmetry is a  parallel section of $\g_S\to Q$ with 
respect to the $\Sigma$-connection.   For that purpose, let us denote by $c(t)\in \M$ the 
integral curve of $X_\nh$ and by $\gamma(t) = \tau_\M(c(t))$ the corresponding curve on $Q$.

\begin{theorem}\label{T:ParallelTransport}
Let $(\M, \Omega_\subM|_\C, H_\subM)$ be a nonholonomic system with 
 a $G$-symmetry satisfying Conditions $(\cA 1)$-$(\cA 4)$, with a strong invariant kinetic energy 
 on $S$ and such that the horizontal space $H$ in \eqref{Eq:TQ=H+V} is $S$-orthogonal.
 Let us denote by $\gamma(t)$ the curve on $Q$ given by  $\gamma(t) := \tau_\subM(c(t))$ 
 where $c(t)$ is the integral curve of $X_{\emph\nh}$.
 If $\kappa(X_0,[Y,X_0])=0$ for all $Y\in \Gamma(S)$ and $X_0$ 
 a $\rho$-projectable vector field on $Q$ taking values in $H$, then the parallel transport of 
 $\zeta_0\in (\g_S)_{q_0}$, for $q_0\in Q$, with respect to the $\Sigma$-connection along 
 the nonholonomic dynamics $\gamma(t)$ on $Q$ passing through $q_0$, generates a 
 horizontal gauge symmetry. 
 In other words, if a $G$-invariant section $\zeta \in \Gamma(\g_S)$ satisfies that 
 $\zeta(q_0) = \zeta_0$ and 
 \[
   \overset{\textit{\tiny{$\Sigma$}}}{\mathcal{\nabla}}_{\dot \gamma(t)}\, \zeta = 0,
 \] 
 then the function ${\mathcal J}_\zeta = \langle J^{\emph\nh}, \zeta \rangle \in C^\infty(\M)$ 
 is a horizontal gauge momentum. 
\end{theorem}

\begin{proof} %Consider the basis  $\mathfrak{B}_{TQ}=\{X_0, Y_1,...,Y_l,Z_1,...,Z_N\}$ of $TQ$  where $(\xi_i)_Q = Y_i$. 
Denote by $\mathfrak{B}_{\g_S}=\{\xi_1,...\xi_k\}$   a global $G$-invariant basis  
of the bundle $\g_S\to Q$ and then a $G$-invariant section $\zeta$ of $\g_S$ 
is written as $\zeta = f_j \xi_j$ for $f_j\in C^\infty(Q)^G$.  %Let us denote by $c(t) \in \M$ is the integral curve of $X_\nh$ passing though $(q_0, p_0)\in \M$ and by  $\gamma(t) = \tau_\subM (c(t)) \in Q$ such that $\gamma(0) = q_0$. 
Since $\dot \gamma(t)= T\tau_\subM(X_\nh) = v^0 X_0 + v^iY_i$, then 
$(T\tau_\subM X_\nh)(f_l) = v^0X_0(f_l)$ and 
\begin{equation*}
  \overset{\textit{\tiny{$\Sigma$}}}{\nabla}_{\dot \gamma(t)}\zeta =  f_j \overset{\textit{\tiny{$\Sigma$}}}{\nabla}_{\dot \gamma(t)} \xi_j + T\tau_\subM (X_\nh)(f_l) \xi_l   = v^0 (  f_j \Gamma_{0j}^l + X_0(f_l) ) \xi_l + v^i f_j\Gamma_{ij}^l \xi_l,
\end{equation*}
where $\Gamma_{0j}^l, \Gamma_{ij}^l$ are the Christoffel symbols of 
$\overset{\textit{\tiny{$\Sigma$}}}{\nabla}$ in the basis \eqref{Eq:BasisParallel}, 
i.e., $\overset{\textit{\tiny{$\Sigma$}}}{\nabla}_{X_0}\, \xi_j = \Gamma_{0j}^l \xi_l$ 
and $\overset{\textit{\tiny{$\Sigma$}}}{\nabla}_{Y_i}\, \xi_j = \Gamma_{ij}^l \xi_l$.
By the Def.~\ref{Def:SigmaConnection}, $\Gamma_{0j}^l = \hat\Gamma_{0j}^l + 
\Sigma^l(X_0, (\xi_j)_Q)$ and $\Gamma_{ij}^l = \hat\Gamma_{ij}^l + \Sigma^l((\xi_i)_Q, (\xi_j)_Q)$. 
Then $\Gamma_{0j}^l = \hat\Gamma_{0j}^l + \Sigma^l_{0j} = -R_{lj}$ and $\Gamma_{ij}^l = \hat\Gamma_{ij}^l + \Sigma^l_{ij} = 0$.
We conclude that $\overset{\textit{\tiny{$\Sigma$}}}{\nabla}_{\dot \gamma(t)}\zeta = 0$ 
if and only if the functions $(f_1,...,f_k)$ are a solution of the system $- R_{lj}f_j + X_0(f_l) = 0$, 
which means, by Theorem \ref{T:Main}, that $\zeta= f_i\xi_i$ is a horizontal gauge symmetry.  
Observe that we are assuming that $v^0\neq 0$ which is true except in a measure zero set. 
\end{proof}

\noindent{\bf Guiding Example: nonholonomic oscillator.} 
Let us continue with the example describing the nonholonomic oscillator studied 
in Section \ref{Ss:MomEq-Coord}, but in this case, we will consider 
$W = S^\perp \cap V = \textup{span}\{ Z:= -y \frac{\partial}{\partial x} + \frac{\partial}{\partial z}\}$ 
and we recall that $H = \textup{span}\{ X_0:=  \frac{\partial}{\partial y}\}$ and 
$S = \textup{span}\{ Y:= \frac{\partial}{\partial x} + y \frac{\partial}{\partial z}\}$.
Denoting by $\xi = (1,y)$ the $G$-invariant generator of $\Gamma(\g_S)$, 
the Christoffel symbols of $\widehat\nabla$ are given by 
\[
\widehat{\nabla}_{X_0}\xi = \widehat{\Gamma}_{\mbox{\tiny{$0Y$}}}^{\mbox{\tiny{$Y$}}} \xi =  
\tfrac{y}{1+y^2} \xi \quad \mbox{and} \quad \widehat{\nabla}_{Y}\xi = 
\widehat{\Gamma}_{\mbox{\tiny{$YY$}}}^{\mbox{\tiny{$Y$}}} \xi = 0.
\]
Therefore, we observe that $\overset{\textit{\tiny{$\Sigma$}}}{\mathcal{\nabla}} = \widehat\nabla$ 
since, using Def.~\ref{Def:SigmaConnection}, the  $\g$-valued bilinear form  
$\Sigma = \Sigma^{\mbox{\tiny{$Y$}}} \otimes \xi = 0$, where  
\[
\Sigma^{\mbox{\tiny{$Y$}}} = -(\widehat \Gamma_{\mbox{\tiny{$0Y$}}}^{\mbox{\tiny{$Y$}}} +R_{\mbox{\tiny{$YY$}}}) dy \otimes (\tfrac{1}{1+y^2}(dx+ydz)) - \widehat\Gamma_{\mbox{\tiny{$YY$}}}^{\mbox{\tiny{$Y$}}} \tfrac{1}{(1+y^2)^2}(dx+ydz) \otimes (dx+ydz) =0.
\]   
Following Theorem \ref{T:ParallelTransport}, $\zeta= f(y)\xi$ is a $G$-invariant horizontal 
gauge symmetry if and only if $\widehat{\nabla}_{\dot \gamma} \zeta = 0$.

\section{Existence of horizontal gauge momenta and related consequences on the dynamics and geometry of the systems}\label{Sec:Consequences}

\subsection{Integrability and hamiltonization of the reduced dynamics} \label{Ss:BroadHamiltonization}
As we saw in Section~\ref{Ss:InitialSetting}, a nonholonomic system $(\cM,\Omega_\subM|_\C,H_\subM)$ 
with a $G$-symmetry can be reduced to the quotient manifold $\M/G$ and the reduced dynamics 
is given by integral curves of the vector field $X_{\red}$ 
on $\M/G$ defined in \eqref{Eq:RedDyn}. Moreover, since the hamiltonian function 
$H_\subM$ on $\M$ is $G$-invariant as well, it descends to a {\it reduced hamiltonian 
function} $H_\red$ on the quotient $\M/G$, i.e., $H_\subM = \rho^*H_\red$, 
and as expected, it is a first integral of $X_\red$. 
The following Lemma will be used in the subsequence subsections.
\begin{lemma}\label{L:dimM/G} If  $(\M, \Omega_\subM|_\C, H_\subM)$ is a nonholonomic system 
with a $G$-symmetry satisfying Conditions $(\cA1)$, $(\cA2)$ and $(\cA4)$ 
then $\textup{dim}(\M/G) = k+2$, where $k = \textup{rank}(S)$. 
\end{lemma}
\begin{proof}
From \eqref{Eq:TQ=H+S+W}, we have that $D = H\oplus S$ and thus we observe 
that $\textup{rank}(D) =k+1$, since $\textup{rank}(H) = \textup{dim}(Q/G) = 1$ 
and $\textup{rank}(S) =k$.  Then $\textup{dim}(\M) = \textup{dim}(Q) + \textup{rank}(D)$ 
and hence, since $G$ acts on $T^*Q$ by the lifted action,  $\textup{dim}(\M/G) 
= \textup{dim}(Q/G) + \textup{rank}(D) = k+2$.
\end{proof}

\subsubsection*{Integrability of the reduced system}

In this Section, we recall the concept of `broad integrability' and we 
show that the reduced dynamics $X_\red$ on $\M/G$ of a nonholonomic system 
$(\cM,\Omega_\subM|_\C,H_\subM)$ with a
$G$-symmetry satisfying the hypotheses of Theorem~\ref{T:Main}, 
is integrable by quadratures or geometric integrable,\footnote{We 
recall that integrability by quadratures is also called {\it geometric integrability}, see \cite{SVN}.} 
and if some compactness hypothesis are satisfied it is also `broadly integrable'.  
In order to perform our analysis we identify broad integrability,
which extends complete, or better non-commutative, integrability outside the Hamiltonian framework, 
with quasi-periodicity of the dynamics.
We base our analysis on the characterization of quasi-periodicity outside the hamiltonian framework, 
introduced in \cite{bogoyavlenskij} (see also \cite{fedorov,FG2002,zung2016}). 

\begin{definition}\label{Def:BroadInt}
 A vector field $X$ on a manifold $M$ of dimension $n$, is called {\it broad integrable}, if
  \begin{enumerate}
    \item[$(i)$] there exists a submersion $F=(f_1,\ldots, f_{n-d}): M\longrightarrow \bR{n-d}$
    with compact and connected level sets, whose components $f_1,\ldots, f_{n-d}$ are
    first integrals of $X$, i.e. 
    $X(f_i) = 0$, for all  $i=1,\ldots, n-d$;
    \item[$(ii)$] there exists $d$ linearly independent vector fields, $Y_1,\ldots, Y_d$ on $M$ 
    tangent to the level sets of the first integrals (i.e., $Y_\alpha(f_i) = 0 $ 
    for all $\alpha = 1,\ldots, d$ and for all $i=1,\ldots,n-d$) 
    that pairwise commute and commute with $X$.\footnote{We recall that the vector fields $Y_1,\ldots, Y_d$ 
    are also called {\it dynamical symmetries} of $X$.}
  \end{enumerate}
\end{definition}

As in the hamiltonian case, being broad integrable, has important consequences 
in the characterization of the dynamics and the geometry of the phase space:

\begin{theorem}[\cite{bogoyavlenskij,fedorov,zung2016}] \label{T:B-integrability} 
  Let $M$ be a manifold of dimension $n$.  If the vector field $X$ on $M$ is broad integrable, then 
  \begin{itemize}
    \item[{\it (i)}] for each $c\in\bR{n-d}$, the level sets $F^{-1}(c)$ of $F$ on $M$ are diffeomorphic to 
    $d$--dimensional tori;
    \item[{\it (ii)}] the flow of $X$ is conjugated to a linear flow on the fibers of $F$. Precisely,
    for each $c\in\bR{n-d}$, there exists a neighbourhood $\cU$ of $F^{-1}(c)$ in $M$
    and a diffeomorphism 
    \[\begin{aligned}
     \Phi: & \; \cU\longrightarrow F(\cU)\times \bT{d} \\
       & m \longrightarrow \Phi(m) = (F(m),\varphi(m))
    \end{aligned}\]
    which conjugate the flow of $X$ on $\cU$ to the linear flow
    \[
      \dot F = 0\,,\qquad  \dot\varphi = \omega(F)\,;
    \]
    on $F(\cU)\times\bT{d}$, for certain functions $\omega_i:F(\cU)\longrightarrow\bR{}$.
  \end{itemize}
\end{theorem}

Now, we go back to our nonholonomic system $(\cM,\Omega_\subM|_\C,H_\subM)$ with a $G$-symmetry.
If we assume that the hypotheses of Theorem~\ref{T:Main} are satisfied, then the nonholonomic 
system admits $k=\textup{rank}(S)$ (functionally independent) $G$-invariant horizontal 
gauge momenta. 
This fact, plus recalling that $H_\red$ is a first integral of $X_\red$ and the fact that reduced 
manifold $\M/G$ has dimension $k+2$,  ensures that the reduced 
dynamics $X_\red$ is integrable by quadratures. 
Moreover, if the joint level sets of the first integrals are connected and compact
the reduced dynamics satisfies the hypothesis of Theorem~\ref{T:B-integrability}  
and it is then broad integrable on circles.
We can summarize these integrability issues as follows. 

\begin{theorem} \label{T:reduced-integrability}
Consider  a nonholonomic system  $(\M, \Omega_\subM, H_\subM)$ with a $G$-symmetry 
satisfying Conditions $(\cA1)$-$(\cA4)$.
If the hypotheses of Theorem~\ref{T:Main} are fulfilled, then
\begin{enumerate}
  \item[$(i)$] The vector field $X_{\emph\red}$ admits $k+1$ (functionally independent) first integrals 
  $\{\bar{\mathcal J}_1,\ldots,\bar{\mathcal J}_k, H_{\emph\red}\}$
  on $\M/G$, where  $H_{\emph\red}$ is the reduced hamiltonian;
  \item[$(ii)$] The map $F=(\bar{\mathcal J}_1,\ldots,\bar{\mathcal J}_k,H_{\emph\red}):\M/G 
  \longrightarrow \bR{k+1}$ is a surjective submersion. 
  The non equilibrium orbits of the reduced dynamics $X_{\emph\red}$ 
  are given by the joint level sets of $(\bar{\mathcal J}_1,\ldots,\bar{\mathcal J}_k, H_{\emph\red})$, 
  and hence the reduced dynamics is integrable by quadratures;
  \item[$(iii)$] If the map $F=(\bar{\mathcal J}_1,\ldots,\bar{\mathcal J}_k, H_{\emph\red}):\M/G 
  \longrightarrow \bR{k+1}$ is proper, then the reduced dynamics 
  is broad integrable and the reduced phase space 
  inherits the structure of a $S^1$-principal bundle.
\end{enumerate}
\end{theorem} 

\begin{proof} 
Given that $\dim \M/G = k+2$ and that we have $k+1$ (functionally independent) 
   first integrals of the reduced dynamics $X_\red$, namely the $k$ horizontal gauge momenta
   $\bar \cJ_1,\ldots,\bar \cJ_k$ from Theorem~\ref{T:Main} and the reduced Hamiltonian $H_\red$,
   the reduced dynamics is integrable by quadratures.   
   Items $(ii)$ and $(iii)$ follow immediately from Definition~\ref{Def:BroadInt} and Theorem~\ref{T:B-integrability}.
\end{proof}

%Observe that Theorem~\ref{T:reduced-integrability} is directly implied by Theorem~\ref{T:Main}  and the previous observations.

%***************************

%\noindent {\bf Guiding Example: nonholonomic particle}
%\textcolor{blue}{to be completed.  But I dont think it is necessary this guiding example here}

%\textcolor{red}{I agree, we may include it at the end of this section, but in case we need to modify a bit the particle if we want to have periodic orbits on the reduced space and if we want to implement reconstruction}

\subsubsection*{Hamiltonization}\label{Ss:Hamiltonization}
The non-hamiltonian character of a nonholonomic system can also be seen by the fact that the dynamics  is not described by a symplectic form or a Poisson bracket. % (see \eqref{Eq:NHDyn}). 
More precisely, as we have seen in Section \ref{Ss:InitialSetting}, the restriction of the 2-form $\Omega_\subM$ on the distribution $\C$ is nondegenerate and hence it 
%satisfy that $\Omega_\subM|_\C$ is a nondegenerate 2-section. 
allows to define the {\it nonholonomic bracket} $\{\cdot, \cdot\}_\nh$ on functions on $\M$ 
(see \cite{ShaftMashke,Marle98,IdLMM}), given, for each $f\in C^\infty(\M)$, by
\begin{equation}\label{Eq:NHBracket}
X_f = \{\cdot, f\}_\nh \mbox{ \ if and only if \ } {\bf i}_{X_f}\Omega_\subM|_\C = df |_\C,
\end{equation}
where $(\cdot)|_\C$ denotes the point-wise restriction to $\C$. The nonholonomic bracket is an {\it almost Poisson bracket} on $\M$ 
(see Appendix \ref{A:Hamiltonization} for more details) with characteristic distribution given by
the nonintegrable distribution $\C$ and we say that it {\it describes the dynamics} since the nonholonomic vector field $X_\nh$ is hamiltonian with respect to the bracket and the hamiltonian function $H_\subM$, i.e., 
\begin{equation}\label{Eq:DynamicsBracket}
X_\nh = \{\cdot, H_\subM\}_\nh.
\end{equation}

In this framework, we use the triple $(\M,\{\cdot, \cdot\}_\nh, H_\subM)$ to define a nonholonomic system. 

If the nonholonomic system admits a $G$-symmetry, then the nonholonomic bracket $\{\cdot,\cdot\}_\nh$ is $G$-invariant and it defines an almost Poisson bracket $\{\cdot, \cdot\}_\red$ on the quotient space $\M/G$ given, for each $\bar{f},\bar{g}\in C^\infty(\M/G)$,  by 
\begin{equation}\label{Eq:RedBracket}
  \{\bar{f},\bar{g}\}_\red \circ \rho(m) = \{ \bar{f}\circ \rho, \bar{g}\circ \rho\}_\nh (m), \qquad m\in\M,
\end{equation}
where $\rho:\M\to \M/G$ is, as usual, the orbit projection (see App.~\ref{A:Hamiltonization}).  %Moreover, since the hamiltonian function $H_\subM$ on $\M$ is $G$-invariant as well, it descends to a {\it reduced hamiltonian function} $H_\red$ on the quotient $\M/G$, i.e., $H_\subM = \rho^*H_\red$.
The reduced bracket $\{\cdot, \cdot\}_\red$ describes the reduced dynamics $X_\red$ (defined in \eqref{Eq:RedDyn}) since
\[
X_\red = \{\cdot, H_\red\}_\red.
\]

The {\it hamiltonization problem} studies whether the reduced dynamics $X_\red$ is hamiltonian with respect to a  Poisson bracket on the reduced space $\M/G$ (that might be a different bracket from $\{\cdot, \cdot\}_\red$).  

One of the most important consequences of Theorem \ref{T:Main} is related with the hamiltonization problem as the following theorem shows.

\begin{theorem} \label{T:BalYapu19}
 If a nonholonomic system $(\M, \{\cdot, \cdot\}_{\emph\nh}, H_\subM)$ with a $G$-symmetry verifying 
 Conditions $(\cA 1)$-$(\cA 4)$ satisfies the hypotheses of Theorem \ref{T:Main}, 
 then there exists a $\textup{rank}$ 2-Poisson bracket 
 $\{\cdot, \cdot\}_{\emph\red}^{B_{\emph{\mbox{\tiny{\!H\!G\!M}}}}}$ on $\M/G$  
 describing the reduced dynamics: 
 \[
   X_{\emph\red} = \{\cdot, H_{\emph\red}\}_{\emph\red}^{B_{\emph{\mbox{\tiny{\!H\!G\!M}}}}},
 \]
 for $H_{\emph\red}:\M/G\to \R$ the reduced hamiltonian. 
\end{theorem}

The problem of finding the bracket $\{\cdot, \cdot\}_\red^{B_{\mbox{\tiny{\!H\!G\!M}}}}$, 
once $k$ horizontal gauge momenta exist, 
was already studied in \cite{GNMontaldi,BalYapu19}). 
However here, in the light of the techniques introduced to prove Theorem~\ref{T:Main},
we take a different path to put in evidence the role played by the {\it momentum equation}.  
More precisely, first we study how different choices of a (global $G$-invariant) basis 
$\mathfrak{B}_{\g_S}$ of $\Gamma(\g_S)$ generate different 
rank 2-Poisson brackets on $\M/G$.  If the nonholonomic system admits $k$ (functionally 
independent $G$-invariant) horizontal gauge symmetries then there will be a rank 2-Poisson 
bracket $\{\cdot, \cdot \}_\red^{\B_{\mbox{\tiny{H\!G\!M}}}}$ that describes the dynamics which 
is defined by choosing the basis of $\Gamma(\g_S)$ given by the horizontal gauge symmetries. 
Then we show how $\{\cdot, \cdot\}_{\red}^{B_{\mbox{\tiny{\!H\!G\!M}}}}$  depends on the 
system of differential equations \eqref{Eq:ODESystem}.
For the basic definitions regarding Poisson brackets, bivector fields and gauge 
transformations see Appendix \ref{A:Hamiltonization}.

Let us consider a 2-form $B$ on $\M$ that is semi-basic with respect to the bundle $\tau_\subM:\M\to Q$. The gauge transformation of $\{\cdot, \cdot\}_\nh$ by the 2-form $B$ gives the almost Poisson bracket $\{\cdot,\cdot\}_\B$ defined, at each $f \in C^\infty(\M)$, by
$$
{\bf i}_{X_f}(\Omega_\subM + B)|_\C = df |_\C \quad \mbox{if and only if} \quad X_f = \{\cdot , f \}_\B.
$$

%(the semi-basic condition on $B$ guarantees that $(\Omega_\subM + B)|_\C$ is nondegenerate and hence $\{\cdot, \cdot\}_\B$ is well defined, see Rmk.~\ref{R:Semi-basic}). 
 %then $(\Omega_\subM + B)|_\C$ is nondegenerate (see Rmk.~\ref{R:Semi-basic} and \cite{BN2011}) and we can

%The almost Poisson brackets $\{\cdot, \cdot\}_\nh$ and $\{\cdot, \cdot\}_\B$ are gauge related (see \eqref{Eq:NHBracket}). 
If the 2-form $B$ is $G$-invariant, then the  bracket $\{\cdot, \cdot\}_\B$ is also $G$-invariant and it can be reduced to an almost Poisson bracket $\{\cdot, \cdot\}_{\red}^{\B}$ on the quotient manifold $\M/G$ given, at each $\bar{f}, \bar{g}\in C^\infty(\M/G)$, by
\begin{equation}\label{Eq:RedBracketB}
\{\bar{f},\bar{g}\}^\B_\red \circ \rho(m) = \{ \bar{f}\circ \rho, \bar{g}\circ \rho\}_\B (m),
\end{equation}
where $m\in \M$, Diag.~\eqref{Diag:Hamilt} (see also \cite{GN2008,BN2011}).
%That is, we obtain an almost Poisson manifold $(\M/G, \{\cdot , \cdot\}_\red^\B)$. 
%Next, we will study how to choose the 2-form $B$ so that $\{\cdot,\cdot\}_\red^\B$ is Poisson on $\M/G$.  

%, we study the effect of performing a gauge transformation of the nonholonomic bracket $\{\cdot, \cdot\}_\nh$ by the 2-form $B_\sigma=\langle J,\sigma_{\g_S}\rangle$ (responsible for the {\it momentum equation} in Prop.~\ref{Prop:MomEq1}).

%Given a global basis $\mathfrak{B}_{\g_S}$ of the bundle $\g_S\to Q$, recall the functions $J_i$ in \eqref{Eq:J_i} and the 2-form $\langle J, \sigma_{\g_S}\rangle$ defined in Definition~\ref{Def:Jsigma}.  

Let $\mathfrak{B}_{\g_S}$ be a global $G$-invariant basis of $\Gamma(\g_S)$ and recall from \eqref{Eq:J_i} the associated $G$-invariant momenta $J_i$. 

\begin{proposition}\label{Prop:SigmaGauge}
 Consider a nonholonomic system $(\M, \{\cdot, \cdot\}_{\emph\nh}, H_\subM)$ 
 with a $G$-symmetry satisfying Conditions $(\cA1)$-$(\cA3)$. Given a (global $G$-invariant) basis 
 $\mathfrak{B}_{\g_S}$ of $\Gamma(\g_S)$, the associated 2-form 
 $B_\sigma = \langle J,\sigma_{\g_S}\rangle$  induces a gauge transformation of the nonholonomic 
 bracket $\{\cdot, \cdot\}_{\emph\nh}$ so that 
 \begin{itemize}
  \item[$(i)$] the gauge related bracket $\{\cdot, \cdot\}_{B_\sigma}$ on $\M$ is $G$-invariant;
  \item[$(ii)$] The induced reduced bracket $\{\cdot, \cdot\}_{\emph\red}^{B_\sigma}$ on $\M/G$ 
  is Poisson with symplectic leaves given by the common level sets of the momenta $\bar{J}_i$, 
  where $\bar{J}_i \in C^\infty(\M/G)$ so that $\rho^*\bar{J}_i = J_i$.  In particular, if Condition 
  $(\cA 4)$ is satisfied, then the Poisson bracket $\{\cdot, \cdot\}_{B_\sigma}$ has 2-dimensional leaves.
 \end{itemize}
\end{proposition} 

\begin{proof}
$(i)$ By construction, we see that the 2-form $\langle J, \sigma_{g_S}\rangle$ is semi-basic with respect to the bundle $\M\to Q$ and, by Lemma \ref{L:sigma}, it is $G$-invariant as well. Therefore, the gauge transformation by the 2-form $\langle J, \sigma_{g_S}\rangle$ defines a $G$-invariant almost Poisson bracket  $\{\cdot, \cdot\}_{B_\sigma}$. 

$(ii)$ The $G$-invariant bracket $\{\cdot, \cdot\}_{B_\sigma}$ induces, on the quotient space $\M/G$, 
an almost Poisson bracket $\{\cdot, \cdot\}_\red^{B_\sigma}$. 
It is shown\footnote{In the notation of \cite{BalYapu19}, $B_{\sigma}$ corresponds to the 2-form $B_1$ but for any $G$-invariant basis of $\Gamma(\g_S)$.  The bracket $\{\cdot, \cdot\}_\red^{B_\sigma}$ is denoted by $\{\cdot, \cdot\}_\red^1$ in the cited reference.}  in \cite[Prop.3.9]{BalYapu19} that $\{\cdot, \cdot\}_\red^{B_\sigma}$ is a Poisson bracket with symplectic leaves given by the common level sets of the momenta 
$\bar{J}_i \in C^\infty(\M/G)$.
\end{proof}
 
Note that the reduced nonholonomic vector field $X_\red$ might not be tangent to the foliation of the bracket $\{\cdot, \cdot\}_\red^{B_\sigma}$.

\begin{definition}\label{Def:Hamiltonization}
 We say that a nonholonomic system $(\M, \{\cdot,\cdot\}_\nh, H_\subM)$ with a $G$-symmetry is {\it hamiltonizable by a gauge transformation} if there exists a $G$-invariant 2-form $B$ so that $\{\cdot, \cdot\}_\red^\B$ is Poisson\footnote{In more generality the bracket $\{\cdot, \cdot \}_\red^\B$ can be conformally Poisson} and 
 \begin{equation}\label{Eq:Hamiltonization}
 X_\red = \{\cdot, H_\red\}_\red^\B,
 \end{equation}
  for $H_\red:\M\to \R$ the reduced hamiltonian. 
  \end{definition}

\begin{definition}\label{Def:DynGauge}\cite{BN2011}
A gauge transformation by a 2-form $B$ of the nonholonomic bracket $\{\cdot, \cdot\}_\nh$ is {\it dynamical} if $B$ is semi-basic with respect to the bundle $\M\to Q$ and  ${\bf i}_{X_\nh} B = 0.$ That is, if $B$ induces a bracket $\{\cdot, \cdot\}_{\B}$ that describes the nonholonomic dynamics: $X_\nh = \{\cdot , H_\subM \}_\B.$
\end{definition}

%A dynamical gauge transformation of $\{\cdot,\cdot\}_\nh$ by $B$ guarantees that the resulting almost Poisson bracket $\{\cdot,\cdot\}_\B$ still describes the nonholonomic dynamics
Therefore, once we know that different 2-forms of the type $B_\sigma$ produce different Poisson brackets on the reduced space, we need to find the one that is {\it dynamical}, if it exists.

%(c.f. \eqref{Eq:DynamicsBracket}) and as a consequence the induced reduced bracket $\{\cdot, \cdot\}_\red^\B$  satisfies \eqref{Eq:Hamiltonization}. 

\medskip

%The proof of Theorem \ref{T:BalYapu19} is based in finding a dynamical gauge transformation $B$ so that the induced reduced bracket $\{\cdot , \cdot \}^\B_\red$ is Poisson.  Let us denote by $\mathfrak{B}_{\g_S} = \{\xi_1, ...\xi_k\}$ a global basis of sections of the bundle $\g_S\to Q$ and by  $\mathfrak{B}_{TQ}=\{X_0, Y_i, Z_a\}$ and $\mathfrak{B}_{T^*Q} = \{X^0, Y^i, \epsilon^a\}$ the dual basis on $TQ$ and $T^*Q$ respectively given in \eqref{Eq:BasisTQ}, where $(\xi_i)_Q = Y_i$.
Observe that if the system admits $k$ ($G$-invariant) horizontal gauge momenta, then we have a preferred basis ${\mathfrak B}_{\mbox{\tiny{HGS}}} = \{\zeta_1, ..., \zeta_k\}$ of $\Gamma(\g_S)$ given by the horizontal gauge symmetries.  Let us denote by $\sigma_{\mbox{\tiny{HGS}}}$ the 2-form $\sigma_{\g_S}$ (defined in \eqref{Def:sigma}), computed with respect to the basis ${\mathfrak B}_{\mbox{\tiny{HGS}}}$ and $B_{\mbox{\tiny{HGS}}} : = \langle J, \sigma_{\mbox{\tiny{HGS}}} \rangle$. The proof of Theorem \ref{T:BalYapu19} is based on the following two facts: on the one hand, $B_{\mbox{\tiny{HGS}}}$ defines a dynamical gauge transformation and on the other hand (by Proposition~\ref{Prop:SigmaGauge}) the resulting reduced bracket $\{\cdot, \cdot\}_{\red}^{B_{\mbox{\tiny{H\!G\!M}}}}$ is Poisson.
%\begin{proposition}\label{Prop:Hamilt}
%Let $(\M, \{\cdot, \cdot\}_{\emph\nh}, H_\subM)$ be a nonholonomic system with a $G$-symmetry satisfying Conditions $\cA$ and consider a (global) $G$-invariant basis $\mathfrak{B}_{\g_S} = \{\xi_1, ...\xi_k\}$ of $\Gamma(\g_S)$ and its associated basis $\mathfrak{B}_{TQ}$. Then, the 2-form $B_{\mbox{\tiny{HGS}}} : = \langle J, \sigma_{\mbox{\tiny{HGS}}} \rangle$ can be written as
 %\begin{equation*}
  %\begin{split}
%   B_{\mbox{\tiny{HGS}}}  & = \langle J, \mathcal{K}_\subW\rangle - \langle J, R_{ij} \mathcal{X}^0 \wedge \mathcal{Y}^j \otimes \xi_j\rangle + \langle J, d\mathcal{Y}^i\otimes \xi_i\rangle,\\
 %  & = p_ad^\C \varepsilon^a - J_i R_{ij} \mathcal{X}^0\wedge \mathcal{Y}^j + J_i d^\C \mathcal{Y}^i,
  %\end{split}
 %\end{equation*}
% where $R_{ij}$ and $J_i$ are the functions defined in \eqref{Eq:R_ij} and \eqref{Eq:J_i} respectively, and $\mathcal{X}^0 = \tau_\subM^*X^0$, $\mathcal{Y}^i = \tau_\subM^*Y^i$, $\varepsilon^a = \tau_\subM^*\epsilon^a$ are the corresponding forms on $\M$. 
%\end{proposition}
%
%\begin{proof}
%
%\end{proof}

\noindent {\it Proof of Theorem \ref{T:BalYapu19}}. 
Under the hypotheses of Theorem \ref{T:Main}, the nonholonomic system admits $k$ $G$-invariant horizontal gauge momenta $\{\mathcal{J}_1,...,\mathcal{J}_k\}$ with the corresponding $G$-invariant horizontal gauge symmetries that generate a basis $\mathfrak{B}_{\mbox{\tiny{HGS}}} =  \{\zeta_1,...\zeta_k\}$ of $\Gamma(\g_S)$. Following \cite[Thm.~3.7]{BalYapu19} and, in particular \cite[Corollary~3.13]{BalYapu19} since $\textup{rank}(H)=1$, the 2-form $ B_{\mbox{\tiny{HGS}}} = \langle J, \sigma_{\mbox{\tiny{HGS}}} \rangle $ associated to the basis $\mathfrak{B}_{\mbox{\tiny{HGS}}}$ induces a dynamical gauge transformation and hence the induced reduced bracket $\{\cdot, \cdot\}_{\red}^{B_{\mbox{\tiny{H\!G\!M}}}}$ describes the reduced dynamics: $X_\red =  \{\cdot, H_\red\}_{\red}^{B_{\mbox{\tiny{H\!G\!M}}}}$. This bracket is then Poisson with symplectic leaves defined by the common level sets  of the horizontal gauge momenta $\{\mathcal{J}_1,...,\mathcal{J}_k\}$ (Proposition~\ref{Prop:SigmaGauge}).

\hfill $\square$

The following diagrams compare Proposition~\ref{Prop:SigmaGauge} with Theorem \ref{T:BalYapu19}.  
The first diagram illustrates the case when we perform a gauge transformation by a 2-form $B_{\sigma}$ (associated to the choice of a basis $\mathfrak{B}_{\g_S}$ of $\Gamma(\g_S)$, Proposition~\ref{Prop:SigmaGauge}) while the second one illustrates the case when the 2-form is $B_{\mbox{\tiny{HGS}}}$ (associated to the basis $\mathfrak{B}_{\mbox{\tiny{HGS}}}$ given by horizontal gauge momenta, Theorem~\ref{T:BalYapu19}).  In both cases, we obtain that the resulting reduced brackets $\{\cdot, \cdot \}_\red^{\B_\sigma}$ and $\{\cdot, \cdot \}_\red^{B_{\mbox{\tiny{H\!G\!M}}}}$ are Poisson.  %So, at a first step we see that we already define a ``new'' bracket on $\M/G$ that has better integrability properties than the reduced nonholonomic bracket. 
However, $\{\cdot, \cdot \}_\red^{\B_\sigma}$ might not describe the reduced dynamics since $B_\sigma$ is not necessarily dynamical. On the other hand, $B_{\mbox{\tiny{HGS}}}$ is always dynamical and thus the reduced bracket $\{\cdot, \cdot \}_\red^{B_{\mbox{\tiny{H\!G\!M}}}}$ describes the dynamics: $X_\red = \{\cdot, H_\red\}_\red^{B_{\mbox{\tiny{H\!G\!M}}}}$.
\begin{equation*}
\xymatrix{  (\M, \{\cdot, \cdot \}_\nh, H_\subM)  \ar[d]_{\mbox{\tiny{reduction} }} \ar[r]^{ \ \ \ \ \ \mbox{\tiny{\begin{tabular}{c} gauge transf \\ by $B_\sigma$ \end{tabular} } } }
&  (\M, \{\cdot, \cdot \}_{\B_\sigma}) \ar[d]\\
 (\M/G, \{\cdot, \cdot \}_\red, H_\red) & (\M/G, \{\cdot, \cdot \}_\red^{\B_\sigma})   }
\qquad \qquad \xymatrix{  (\M, \{\cdot, \cdot \}_\nh, H_\subM)  \ar[d]_{\mbox{\tiny{reduction} }} \ar[r]^{\mbox{\tiny{\begin{tabular}{c} {\it dynamical} gauge transf \\ by $B_{\mbox{\tiny{HGS}}}$ \end{tabular} } } }
&  (\M, \{\cdot, \cdot \}_{B_{\mbox{\tiny{H\!G\!M}}}}, H_\subM) \ar[d]\\
 (\M/G, \{\cdot, \cdot \}_\red, H_\red) & (\M/G, \{\cdot, \cdot \}_\red^{B_{\mbox{\tiny{H\!G\!M}}}}, H_\red)   }
\end{equation*}

\begin{remark}
Under the hypotheses of Theorem \ref{T:BalYapu19}, the functions 
$\{H_\subM, \mathcal{J}_1, ..., \mathcal{J}_k\}$ are in involution with respect to the bracket 
$\{\cdot, \cdot \}_{B_{\mbox{\tiny{H\!G\!M}}}}$, where $\{\mathcal{J}_1, ..., \mathcal{J}_k\}$ 
are the horizontal gauge momenta defined by Theorem~\ref{T:Main}.
In addition, also the reduced functions $\{H_{\red}, \bar{\mathcal{J}}_1, ..., \bar{\mathcal J}_k\}$
 on $\M/G$ are in involution with respect to the reduced bracket  
 $\{\cdot, \cdot \}_{\red}^{B_{\mbox{\tiny{H\!G\!M}}}}$. However these functions are not 
 necessarily in involution with respect to the brackets $\{\cdot, \cdot\}_\nh$ and $\{\cdot, \cdot\}_\red$ respectively. 
\end{remark}

In many cases, the horizontal gauge symmetries cannot be explicitly written, instead they are defined in terms of the solutions of the system of differential equations \eqref{Eq:ODESystem}. Next Theorem gives the formula to write explicitly the dynamical gauge transformation $B_{\mbox{\tiny{HGS}}}$ (and as a consequence the Poisson bracket $\{\cdot, \cdot\}_{\red}^{B_{\mbox{\tiny{H\!G\!M}}}}$) in a chosen basis $\mathfrak{B}_{\g_S}$ that is not necessarily given by the horizontal gauge symmetries.  Examples \ref{Ex:Solids} and \ref{Ex:BallSurface} make explicit the importance of the following formula. 

\begin{theorem}\label{T:FormulaB} Consider a nonholonomic system described by the triple $(\M, \{\cdot, \cdot\}_{\emph\nh}, H_\subM)$ with a $G$-symmetry verifying  Conditions $(\cA 1)$-$(\cA 4)$. Let $\mathfrak{B}_{\g_S} = \{\xi_1,...,\xi_k\}$ be a global $G$-invariant basis of $\Gamma(\g_S)$ and $X_0$ a $\rho$-projectable vector field on $Q$ generating the $S$-orthogonal horizontal space $H$. If the hypotheses of Theorem \ref{T:Main} are satisfied, then the 2-form $B_{\mbox{\tiny{HGS}}}$ %associated to the basis of horizontal gauge symmetries $\mathfrak{B}_{\mbox{\tiny{HGS}}}$ 
is written with respect to the basis $\mathfrak{B}_{\g_S}$ as
   \begin{equation} \label{Eq:Bhgm}
   \begin{split}
   B_{\mbox{\tiny{HGS}}}  : = \langle J, \sigma_{\mbox{\tiny{HGS}}} \rangle   & = \langle J, \mathcal{K}_\subW\rangle - \langle J, R_{ij} \mathcal{X}^0 \wedge \mathcal{Y}^j \otimes \xi_j\rangle + \langle J, d\mathcal{Y}^i\otimes \xi_i\rangle, \\ 
    & = p_ad^\C \varepsilon^a - J_i R_{ij} \mathcal{X}^0\wedge \mathcal{Y}^j + J_i d^\C \mathcal{Y}^i,
  \end{split}
 \end{equation} 
  for $R_{ij}$ and $J_i$ the functions defined in \eqref{Eq:R_ij} and \eqref{Eq:J_i} respectively, and $\mathcal{X}^0 = \tau_\subM^*X^0$, $\mathcal{Y}^i = \tau_\subM^*Y^i$, $\varepsilon^a = \tau_\subM^*\epsilon^a$ the corresponding forms on $\M$. %Moreover, $B_{\mbox{\tiny{HGS}}}$ defines a dynamical gauge transformation.
\end{theorem}

\begin{proof}
In order to prove formula \eqref{Eq:Bhgm}, consider the basis $\mathfrak{B}_{\g_S} = \{\xi_1,...,\xi_k\}$ (not necessarily given by horizontal gauge symmetries), and define the corresponding functions $J_i$ as in \eqref{Eq:J_i}.
If we denote by $F$ the fundamental matrix of solutions of the system of ordinary differential equations \eqref{Eq:ODESystem} (i.e., the columns of $F$ are the independent solutions $(f_1^l,...,f_k^l)$) and by $R$ the $k\times k$-matrix with entries $R_{ij}$, then
\begin{equation}\label{Eq:FundMatrix}
R. F = X_0(F) \quad \mbox{and} \quad \mathcal{J} = F^T {\bf J}, \quad \mbox{where} \quad  \mathcal{J} = \left( \! \begin{smallmatrix*}[c]  \mathcal{J}_1\\ \vdots \\ \mathcal{J}_k  \end{smallmatrix*} \! \right) \quad \mbox{and} \quad {\bf J} = \left( \! \begin{smallmatrix*}[c]  J_1\\ \vdots \\ J_k  \end{smallmatrix*} \! \right).
\end{equation}
Moreover, let us denote by ${\mathcal Y}^i_{\mbox{\tiny{HGS}}}$ the 1-forms on $\M$ such that ${\mathcal Y}^i_{\mbox{\tiny{HGS}}}((\zeta_l)_\subM) = \delta_{il}$ and ${\mathcal Y}^i_{\mbox{\tiny{HGS}}} |_{\mathcal{H}} = {\mathcal Y}^i_{\mbox{\tiny{HGS}}}|_{\mathcal W} = 0$. Then if ${\mathcal Y}_{\mbox{\tiny{HGS}}} = ({\mathcal Y}^1_{\mbox{\tiny{HGS}}} , ..., {\mathcal Y}^k_{\mbox{\tiny{HGS}}})^T$ we have that ${\mathcal Y}_{\mbox{\tiny{HGS}}} = F^{-1} {\mathcal Y}$ where $\mathcal{Y}= ({\mathcal Y}^1 , ..., {\mathcal Y}^k)^T$.
 Hence
 \begin{equation*}
  \begin{split}
   \langle J, d^\C {\mathcal Y}^i_{\mbox{\tiny{HGS}}} \otimes \zeta_i\rangle  = & \ \mathcal{J}^T. \, d^\C{\mathcal Y}_{\mbox{\tiny{HGS}}} = {\bf J}^T F d^\C (F^{-1} {\mathcal Y}) =  {\bf J}^T F X_0(F^{-1}) \mathcal{X}^0 \wedge {\mathcal Y} + {\bf J}^T F F^{-1}d^\C {\mathcal Y} \\ 
   = &  - {\bf J}^T F(F^{-1} X_0(F) F^{-1}) \mathcal{X}^0 \wedge {\mathcal Y} + {\bf J}^T d^\C {\mathcal Y} = - {\bf J}^T R \mathcal{X}^0 \wedge {\mathcal Y} + {\bf J}^T d^\C {\mathcal Y}\\
   = & -  J_i R_{ij} \mathcal{X}^0 \wedge {\mathcal Y}^j + \langle J, d^\C {\mathcal Y}^i\otimes \xi_i\rangle.   
  \end{split}
 \end{equation*}
Finally, we conclude, using Definition~\ref{Def:Jsigma}, that
$$
   B_{\mbox{\tiny{HGS}}}  = \langle J, \mathcal{K}_\subW\rangle +  \langle J, d^\C {\mathcal Y}^i_{\mbox{\tiny{HGS}}} \otimes \zeta_i\rangle =  p_ad^\C \varepsilon^a - J_i R_{ij} \mathcal{X}^0\wedge \mathcal{Y}^j + J_i d^\C \mathcal{Y}^i.
 $$
\end{proof}

%In some cases, the horizontal gauge symmetries cannot be explicitly written, instead they are given as a solution of the system of differential equation \eqref{Eq:ODESystem}. The importance of Theorem \ref{T:BalYapu19}$(i)$ resides in the possibility of writing explicitly the gauge transformation $B_{\mbox{\tiny{HGS}}}$ and the Poisson bracket $\{\cdot, \cdot\}_{\emph\red}^{B_{\mbox{\tiny{H\!G\!M}}}}$ in those cases, see for instance Examples \ref{Ex:Solids} and \ref{Ex:BallSurface}. 

\medskip

Following Example~\ref{Ex:HorSym1} and Corollary~\ref{C:HorSym1}, next we observe 
that a system that admits a basis of $\g_S \to Q$ given by $G$-invariant horizontal 
symmetries is hamiltonizable without the need of a gauge transformation 
(i.e., $B_{\mbox{\tiny{HGS}}}=0$ in this case). 

\begin{corollary}[of Theorem~\ref{T:BalYapu19} and Corollary~\ref{C:HorSym1}, Horizontal symmetries] \label{C:HorSym2}
Let $(\M, \{\cdot, \cdot\}_{\emph\nh}, H_\subM)$ be a nonholonomic system with a $G$-symmetry 
satisfying Conditions $\cA$ and with the bundle $\g_S\to Q$  admitting a basis of 
$G$-invariant horizontal symmetries. Then, the reduced bracket 
$\{\cdot, \cdot\}_{\emph\red}$ on $\M/G$ is twisted Poisson with characteristic distribution 
given by the common level sets of the horizontal gauge momenta.  If Condition $(\cA 4)$ is fulfilled,  
$\{\cdot, \cdot\}_{\emph\red}$ is a $\textup{rank}\, 2$-Poisson bracket. 
\end{corollary}

\begin{proof}
It can be observed from \eqref{Eq:Bhgm} that $B_{\mbox{\tiny{H\!G\!M}}} = 0$ 
when the basis $\mathfrak{B}_{\g_S}$ is given by constant sections (this   was 
also proven in \cite{BalYapu19}). However, it is easier to see a direct proof of this fact: 
if $\eta\in \g$ is a horizontal symmetry, then  ${\bf i}_{\eta_\subM} \Omega_\subM |_\C =  dJ_{\eta} |_\C$, 
thus $\pi_\nh^\sharp(dJ_\eta) = -\eta_\subM$ and hence $\pi_\red^\sharp(d\bar{J_\eta}) =0$. 
Then the reduced bracket $\{\cdot,\cdot\}_\red$ admits $k$ Casimirs. Since the rank of the 
characteristic distribution of $\{\cdot,\cdot\}_\red$ is $dim (\M/G) - k$, by Lemma~\ref{L:dimM/G} 
we conclude that its  characteristic distribution integrable and given by the common 
level sets of the horizontal gauge momenta.  Following Remark \ref{R:RegularTwisted}, 
the reduced bracket $\{\cdot, \cdot\}_\red$ is twisted Poisson. Since Condition $(\cA 4)$ 
implies that $\textup{rank}(H)=1$, then $dim (\M/G) = 2+ k$ and thus the characteristic 
distribution of $\{\cdot , \cdot\}_\red$ has 2-dimensional leaves. 
Therefore, the foliation is symplectic and $\{\cdot, \cdot\}_\red$ is Poisson.

\end{proof}

 %If the system admits $\{\mathcal{J}_1,...,\mathcal{J}_k\}$ $G$-invariant horizontal gauge momenta then we denote by $\mathfrak{B}_{\mbox{\tiny{HGS}}} = \{\zeta_1, ..., \zeta_k\}$ the basis of $\Gamma(\g_S)$ given by the horizontal gauge symmetries, we get that 

\subsection{Horizontal gauge momenta and broad integrability of the complete system}\label{Sec:c-integrability}

%\subsection{Reconstruction and broad integrability of the complete system}
In the previous subsections we have studied the dynamics and the geometry of the reduced system.
Under the hypotheses of Theorem~\ref{T:Main} the reduced dynamics is integrable by quadratures,
and if the joint level sets of the first integrals are connected and compact  the reduced 
dynamics consists of periodic orbits or equilibria. Moreover
the reduced system is hamiltonizable via a rank-2 Poisson structure, whose (global) Casimirs
are the $k$ horizontal gauge momenta.  In this Section we aim to obtain information 
on the dynamics and geometry of the complete system. 
We will then focus in the case in which the reduced dynamics 
is periodic and, by using techniques of reconstruction theory, we will see that if the symmetry 
group $G$ is compact, then the  dynamics of the complete systems is quasi-periodic on tori of 
dimension at most {\it rank}$\,G+1$, where {\it rank}$ \,G$ denotes the rank of the group, i.e.
the dimension of the maximal abelian subgroup of $G$. If the symmetry group $G$ 
is not compact, the complete dynamics can be either quasi-periodic on tori or an unbounded copy of 
$\bR{}$, depending on the symmetry group. Some details on these aspects are reviewed 
in Appendix~\ref{app:rec}, but see also \cite{AM1997,FPZ2020}.
We thus show how the broad integrability 
of the complete dynamics of these type of systems is deeply related to their 
symmetries, that are able to produce, not only the right 
amount of dynamical symmetries, but also 
the complementary number of first integrals. 
We will then apply these results to the example of a heavy homogeneous ball that rolls without sliding 
inside a convex surface of revolution (see Section~\ref{Ex:BallSurface}). This case
presents a periodic dynamics in the reduced space, and a broadly integrable complete 
dynamics on tori of dimension at most three, thus re-obtaining the results 
in \cite{hermans,FG2007}.

We say that a $G$-invariant subset $\cP$ of $\cM$ is a {\it relative periodic orbit for} $X_\nh$, 
if it invariant by the flow and its projection on $\cM/G$ is a periodic orbit of $X_\red$.  Now, we can summarize these results as follows.

\begin{theorem}\label{T:reconstruction}
   Let us consider a nonholonomic system $(\cM,\Omega_\subM|_\C, H_\subM)$ 
   with a $G$-symmetry satisfying Conditions $(\cA1)$-$(\cA4)$.
   Assume that the hypotheses of Theorem~\ref{T:Main} are fulfilled, and that the 
   reduced dynamics is periodic, then
   \begin{enumerate}
   \item[$(i)$] if the group $G$ is compact, the flow of $X_\emph{\nh}$ on a relative periodic orbit $\cP$
   is quasi--periodic with at most $rank\, G+1$ frequencies and the phase space if fibered in tori 
   of dimension up to rank$\, G+1$. 
   \item[$(ii)$] if $G$ is non--compact, the flow of $X_\emph{\nh}$ over a periodic orbit is either 
   quasi--periodic, or a copy of $\bR{}$, that leaves every compact subset of $\mathcal P$.\footnote{From 
   now on we will call {\it escaping} a dynamical behaviour that leaves every compact subset of $\mathcal P$.}
   \end{enumerate}
\end{theorem}

\begin{proof}
To prove this result we combine the results on integrability of the reduced system given by 
Theorem~\ref{T:reduced-integrability} with the results on reconstruction theory from periodic orbits recalled in Appendix~\ref{app:rec}.

More precisely, we confine ourselves to the subspace of the reduced space $\M/G$ in which 
the dynamics is periodic. Then, if the symmetry group is compact, the reconstructed dynamics
is generically quasi-periodic on tori of dimension $d+1$, where $r$ is the rank of the group \cite{field1990,krupa,hermans,CDS}.
The phase space, or at least a certain region of it, has the structure of a $\bT{d+1}$ fiber bundle,
(see \cite{FG2007} for details on the geometric structure of the phase space in this case).
On the other hand if the group is not compact, the reconstructed orbits are quasi-periodic or
a copy of $\bR{}$ that `spirals' toward a certain direction.
\end{proof}
%\textcolor{red}{I would be more explicit in the connection with Theorem \ref{T:Main}: Thm. \ref{T:Main} implies the existence of $fk$ hgm and then by Theorem 4.4. we have that the dynamics is periodic and bla bla....  } \marginpar{\textcolor{red}{Is it ok now, or do you prefer to be more explicit?}}

\section{Examples}\label{S:Examples}

\subsection{The snakeboard}\label{Ex:Skate-1}

The snakeboard is a derivation of the skateboard where the rider is allowed to generate a rotation 
in the axis of the wheels creating a torque so that the board spins about a vertical axis, 
see \cite{OLMB,BKMM}. We denote by $r$ the distance from the center of the board to 
the pivot point of the wheel axes, by $m$ the mass of the board, by $\mathbb{J}$ the inertial 
of the rotor and by ${\mathbb J}_1$ the inertia of each wheel. Following \cite{BKMM}
we assume that the parameters are chosen such that $\bJ+2\bJ_1+\bJ_0 = mr^2$, where 
$\bJ_0$ denotes the inertia of the board. The snakeboard is then modelled 
on the manifold $Q = SE(2) \times S^{1} \times S^{1}$ with coordinates $q=(\theta, x, y, \psi, \phi)$, 
where $(\theta,x,y)$ represent the position and orientation of the board, $\psi$ is the angle 
of the rotor with respect to the board, and $\phi$ is the angle of the front and back 
wheels with respect to the board (in this simplified model they are assumed to be equal). 

%% =======================================================
\begin{figure}[h]%[!htbp]
\begin{center}
%\vskip 5truecm
%\begin{picture}
{\small
%\put(45,0)
{\scalebox{.7}{\includegraphics*{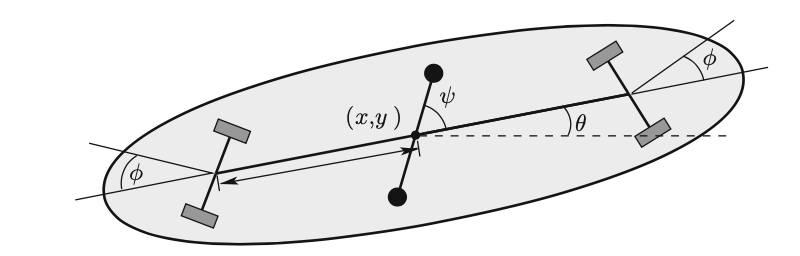}}}
\put(-170,25){$r$}
}
%\end{picture}
\caption{\small The snakeboard.}
\end{center}
\end{figure}
%% =======================================================

The Lagrangian is given by
\begin{equation*}
L(q,\dot{q}) = \frac{1}{2} m (\dot{x}^2 + \dot{y}^2 + r^2 \dot{\theta}^2) + \frac{1}{2} {\mathbb J} \dot{\psi}^2 + {\mathbb J} \dot{\psi} \dot{\theta} + {\mathbb J}_0 \dot{\phi}^2.
\end{equation*}
The nonholonomic constraints impose that the front and back wheels 
roll without sliding and hence the constraint 1-forms are defined to be
\begin{equation}\label{epsbdef}
\begin{split}
\omega^1 & = -\sin(\theta +\phi)\, dx + \cos(\theta+\phi)\,dy-r \cos\phi\,d\theta,\\ 
\omega^2 & = -\sin(\theta -\phi)\, dx + \cos(\theta-\phi)\,dy+r \cos\phi\,d\theta.
\end{split}
\end{equation}
Note that $\omega^1$ and $\omega^2$ are independent whenever $\phi \neq \pm\pi/2$.  Therefore, we define the configuration manifold $Q$ so that $q=SE(2) \times S^1\times (-\pi/2,\pi/2)$. The constraint distribution $D$ is given by
\begin{equation}
\label{Ex:Snake:D}
D = \textup{span} \{ Y_\theta := \sin\phi\partial_{\theta} -r\cos\phi\cos\theta \partial_x -r\cos\phi\sin\theta \partial_y , \, \partial_\psi , \,  \partial_{\phi} \}.
\end{equation}

\noindent {\bf The existence of horizontal gauge momenta}. The system is invariant with respect to the free and proper action on $Q$ of  $G = SE(2)\times S^1$ given by 
\[
\Phi((\alpha, a,b; \beta),(\theta, x, y, \psi, \phi) ) = (\theta+\alpha, x\cos \alpha - y\sin\alpha + a, x\sin\alpha + y\cos\alpha + b , \psi +\beta, \phi),
\] 
and hence $V=\textup{span}\{\partial_\theta, \partial_{\psi}, \partial_x, \partial_y\}$ and  $S = \textup{span}\{Y_\theta, \partial_{\psi}\}$ (see \cite{BKMM}). First, we observe that $[Y_\theta, \partial_{\psi}]=0$ and hence the kinetic energy metric is trivially strong invariant on $S$.  Second,  $H := \textup{span}\{ \partial_\phi\}$ and it is straightforward to check that $V^\perp = H$.  Then, by Corollary \ref{C:Cond-Friends}$(i)$ the system admits 2 (functionally independent) $G$-invariant horizontal gauge momenta.  

%In order to compute them, following Section \ref{Ss:HGM}, we observe that $W$ can be chosen to be $W= \textup{span}\{\partial_x, \partial_y\}$ and therefore, the bundles $\g_S\to Q$ and $\g_W \to Q$ are given by
 %$$
 %\g_S = \textup{span}\{\xi_1= (1,\sigma_1(q),\sigma_2(q),0), \xi_2=(0,0,0,1)\} \quad \mbox{and} \quad \g_W = \textup{span}\{(0,1,0,0), (0,0,1,0)\},
 %$$
 %where $(\xi_1)_\subQ = Y_\theta$ and $(\xi_2)_\subQ = \partial_{\psi}$. Observe that $\xi_1$ and $\xi_2$ are $G$-invariant sections of $\g_S\to Q$.  

\noindent {\bf The computation of the of horizontal gauge momenta}. 
Let us consider the adapted basis to $TQ = D\oplus W$, given by 
$\mathfrak{B}_{TQ}=\{Y_\theta, \partial_\psi, \partial_\phi, Z_1,Z_2\}$, where 
$$
Z_1 := \frac{1}{2\cos\phi} \left(-\sin\theta \partial_x +\cos\theta \partial_y - \frac{1}{r}\partial_\theta\right)
\qquad \mbox{and} \qquad Z_2 := \frac{1}{2\cos\phi} \left(-\sin\theta \partial_x +\cos\theta \partial_y + \frac{1}{r}\partial_\theta \right).
$$  
Denoting by $(p_\theta, p_\psi,p_\phi,p_1, p_2)$ the coordinates on $T^*Q$ associated 
to the dual basis  
$$
\mathfrak{B}_{T^*Q} = \{ \alpha_\theta: = -\tfrac{1}{r\cos\phi}(\cos \theta dx +\sin\theta dy), d\psi, d\phi, \omega^1, \omega^2\},
$$
we obtain that 
\[
 \M = \left\{(q ; p_\theta, p_\psi,p_\phi,p_1,p_2) \ : \ p_1 = - p_2 = -\tfrac{1}{2}\Big( \tfrac{(mr^2-\mathbb{J})\sin\phi}{r\cos\phi\, \Delta} p_\theta  + \tfrac{mr\cos\phi}{\Delta} p_\psi \Big)  \right\},
\]
where $\Delta = \Delta(\phi) = mr^2 - \bJ\sin^2\phi$ (recall that $\Delta(\phi)>0$, since $mr^2>\bJ$). 

We consider the global basis of $\g_S$ given by $\mathfrak{B}_{\g_S} = \{ \xi_1 = (\sin\phi, -r\cos\phi \cos\theta +y, -r\cos\phi\sin\theta -x;0), \xi_2 = (0,0,0;1)\}$, and we observe that $(\xi_1)_\subQ = Y_\theta$ and $(\xi_2)_\subQ = \partial_{\psi}$.  Following \eqref{Eq:J_i}, $J_1 = \langle J^\nh, \xi_1 \rangle = p_\theta$ and $J_2 = \langle J^\nh, \xi_2 \rangle = p_\psi$.

The function $\mathcal{J}= f_\theta(\phi)p_\theta + f_\psi(\phi) p_\psi$ is a horizontal gauge momentum 
if and only if  $R. f = f'$ where $R$ is the $2\times 2$ matrix given in \eqref{Eq:ODESystem}, 
$f = (f_\theta, f_\psi)^t$ and $f'=(f'_\theta, f'_\psi)$ for $f'_\theta = \tfrac{d}{d\phi}f_\theta$ (analogously for $f'_\psi$).  
In our case, using that $\{Y_\theta, \partial_\psi\}$ is a basis of $S$ and $X_0 = \partial_\phi$, we obtain 
\begin{equation*}
 R = [\kappa|_S]^{-1} N, \qquad \mbox{for} \ [\kappa|_S] = 
 \left( \! \begin{smallmatrix*}  m r^2 & \mathbb{J}\sin\phi \\ \mathbb{J}\sin\phi & \bJ  \end{smallmatrix*} \! \right)  \ \mbox{and} \ N =  \left( \! \begin{smallmatrix*} 0 & \ \ 0 \\ -\mathbb{J}\cos\phi & \ \ 0 \end{smallmatrix*}\!\right).
 \end{equation*}
 Hence, we arrive to the linear system
 \begin{equation}\label{Ex:Snake:ODE}
 \tfrac{\cos\phi}{\Delta} \left( \! \begin{array}{cc} \mathbb{J}\sin\phi  & 0 \\ - mr^2\cos\phi &  0  \end{array} \! \right)
 \left( \! \begin{array}{c}  f_\theta \\ f_\psi  \end{array} \! \right) =\left( \! \begin{array}{c}  f'_\theta \\ f'_\psi  \end{array} \! \right),
 \end{equation}
 which admits 2 independent solutions:  $f^1= (f^1_\theta,f^1_\psi)$, 
 with $f^1_\theta = \frac{1}{\sqrt{2 \Delta}}$, $f^1_\psi = -f^1_\theta\sin\phi $,  and $f^2 =(0,1)$. Therefore the horizontal gauge momenta can be written as
\begin{equation}\label{Ex:Snake:J1J2}
    \mathcal{J}_1 = \tfrac{1}{\sqrt{2 \Delta}} \;
    ( p_\theta- p_\psi\, \sin\phi ) \qquad \mbox{and} \qquad   \mathcal{J}_2 = p_\psi.
 \end{equation}

\begin{remarks}
\begin{enumerate}
 \item[$(i)$] On the one hand, since $\xi_2$ is a horizontal symmetry, it is expected to have $\mathcal{J}_2= p_\psi$ conserved (Cor.~\ref{C:HorSym1}). On the other hand, the horizontal gauge momentum $\mathcal{J}_1$ is realized 
 by a non-constant section $\zeta_1$ and, as far as we could search, $\mathcal{J}_1$ has 
 not appeared in the literature yet. Moreover, using that $ H_\subM = \frac{1}{2} \left(\frac{ p_\theta^2}{\Delta}\, - 2\frac{\sin\phi}{\Delta}\, p_\theta p_\psi+
            \frac{m r^2}{\bJ\, \Delta}\, p_\psi^2 +\frac{p_\phi^2}{2\bJ_0}\, \right),$ it is possible to check our results.  
 \item[$(ii)$] The horizontal gauge momenta \eqref{Ex:Snake:J1J2} can also be obtain from the {\it momentum equation} 
 in Proposition~\ref{Prop:MomEq1}, which in case is written as 
 $f_\theta\langle J, \sigma_{\g_S} \rangle (Y_\theta, X_\nh) + f_\psi\langle J, \sigma_{\g_S} 
 \rangle (\partial_\psi, X_\nh) + p_\theta X_\nh(f_\theta) + p_\psi X_\nh(f_\psi) = 0.$
\end{enumerate}
\end{remarks}

%\noindent {\bf The Lie groupoid associated.}  The horizontal gauge symmetries $\zeta_1, \zeta_2\in \Gamma(\g_S)$ are given $\zeta_1 = f(\xi_1 - \xi_2)$ and $\zeta_2 = \xi_2$.   Then, the associated curves $\gamma_i := \gamma_{\zeta_i} : Q\times \R \to Q\times G$ are given by 
%$$
%\gamma_1 (t,q) = (q; f(\phi) t, \sigma_1(q)f(\phi) t, \sigma_2(q)f(\phi) t, -f(\phi) t) \quad \mbox{and} \quad \gamma_2(t,q) = (q; 0,0,0,t).
%$$
%Since $\zeta_2$ is a constant section (it is independent of the point $q$), we se that $\gamma_2$ is a curve $\gamma_2:\R\to G$ inducing a 1-parameter subgroup of $G$. However, $\gamma_2$ depends strictly on $q$, and thus we have a Lie groupoid associated given by $Q\times \R\rightrightarrows Q$ with 
%$$
%\ss(q,t) = q \quad \mbox{and} \quad \tt(q,t) = (\theta + f(\phi) t, x +\sigma_1(q)f(\phi) t, y+\sigma_2(q)f(\phi) t,\psi -f(\phi) t)
%$$
%
% While $\zeta_2$ is generated by rotations of the rotor with respect to the board.  
%
%\textcolor{blue}{Physics interpretation? there is $\mathcal{G}$-action on $Q$?}

\medskip

\noindent {\bf Hamiltonization and integrability.} The system descends to the quotient 
manifold $\M/G$ equipped with coordinates $ ( \phi, p_\phi, p_\theta, p_\psi)$.  
The $G$-invariant horizontal gauge momenta ${\mathcal J}_1, {\mathcal J}_2$ in 
\eqref{Ex:Snake:J1J2} and the hamiltonian function $H_\subM$, also descend to 
functions $\bar{\mathcal J}_1, \bar{\mathcal J}_2$ and $H_\red$ on $\M/G$. 

\medskip

\noindent {\it Integrability.}  
Since the reduced space $\M/G$ is $4$-dimensional, 
Theorem~\ref{T:reduced-integrability} guarantees that the reduced dynamics is integrable by quadratures. 
We observe that the reduced system is not periodic, thus we can say nothing generic on the complete dynamics 
or on the geometry of the phase space.

\medskip

\noindent{\it Hamiltonization.} 
Theorem \ref{T:BalYapu19} guarantees that the system is Hamiltonizable. In order to write the Poisson bracket on $\M/G$ that describes the dynamics, we compute the 2-form $B_{\mbox{\tiny{HGS}}}$ in terms of the basis $\mathfrak{B}_{TQ}=\{Y_1:=Y_\theta, Y_2:=\partial_\psi, X_0:=\partial_\phi, \partial_x, \partial_y\}$ using Theorem \ref{T:FormulaB}. Let us denote by $R_{ij}$ the elements of the matrix $R$ in \eqref{Ex:Snake:ODE}, and then 
\begin{equation*}
  B_{\mbox{\tiny{HGS}}}  = \langle J, \mathcal{K}_\subW \rangle - p_\theta(R_{11} d\phi \wedge d\theta + R_{12} d\phi \wedge d\psi) - p_\psi (R_{21} d\phi\wedge d\theta + R_{22} d\phi\wedge d\psi) + p_\theta d\alpha_\theta.
  \end{equation*}
  First, we observe that 
  $$ 
  \langle J, \mathcal{K}_\subW\rangle |_\C = \iota^*(p_1)d\omega^1 + \iota^*(p_2)d\omega^2 |_\C =  -\left( \tfrac{(mr^2-\mathbb{J})\sin\phi}{\cos\phi\, \Delta} p_\theta  + \tfrac{mr^2\cos\phi}{\Delta} p_\psi\right)  \, d\phi\wedge \alpha_\theta |_\C
 $$
 Second, we observe that 
 $$
 (R_{11}p_\theta + R_{21}p_\psi) d\phi \wedge \alpha_\theta = \left(\tfrac{\mathbb{J}\sin\phi\cos\phi}{\Delta} p_\theta - \tfrac{mr^2 \cos\phi}{\Delta}p_\psi \right)d\phi\wedge\alpha_\theta.
 $$
 Finally, using that $p_\theta d\alpha_\theta|_\C =  p_\theta \tan\phi\, d\phi\wedge\alpha_\theta$  we obtain that $B_{\mbox{\tiny{HGS}}} = 0$.

As a consequence of Theorem \ref{T:BalYapu19}  the reduced bracket $\pi_\red$  which is given by
$$
\pi_\red = \partial_\phi \wedge \partial_{p_\phi} + \tfrac{\cos\phi}{\Delta} (\mathbb{J} \sin\phi \, p_\theta - mr^2 p_\psi ) \partial_{p_\phi}\wedge \partial_{p_\theta},
$$
is a Poisson bracket on $\M/G$ with $\bar{\mathcal{J}}_1$ and $\bar{\mathcal{J}}_2$ playing the role of Casimirs. 
The reduced nonholonomic vector field is then 
\[ 
  X_\red = \{\cdot,  H_\red\}_\red\,.
\]
%%Since the map $(\bar{\mathcal J}_1,\bar{\mathcal J}_2, H_{\emph\red}):\M/G \longrightarrow \bR{3}$ 
%%\textcolor{blue}{is proper, then we conclude that the dynamics of the reduced system is periodic.}
%Therefore,  the reduction of the nonholonomic bracket $\pi_\nh$ is a Poisson bracket on $\M/G$. That is, the nonholonomic bracket $\pi_\nh$ on $\M$ is given by
%$$
%\pi_\nh = \mathcal{Y}_\theta \wedge \partial_{p_\theta} + \partial_\psi \wedge \partial_{p_\psi} + \partial_\phi \wedge \partial_{p_\phi} - \mathcal{F}(\phi) (p_\theta - p_\psi) \partial_{p_\theta}\wedge \partial_{p_\phi},
%$$
%where $\mathfrak{B}_{T\M}= \{\mathcal{Y}_\theta,  \partial_\psi,  \partial_\phi, \partial_x, \partial_y, \partial_{p_\theta}, \partial_{p_\psi}, \partial_{p_\phi} \}$ is the dual basis of $\mathfrak{B}_{T^*\!\M}= \left\{d\theta, d\psi,   d\phi, \epsilon^x, \epsilon^y, d{p_\theta}, \right.$  $\left.  d{p_\psi}, d{p_\phi} \right\}$. 
%Hence, on $\M/G$ with coordinates $(\phi, p_\phi, p_\theta, p_\psi)$ the reduced nonholonomic bracket is written as 
%$$
%\pi_\red = \partial_\phi \wedge \partial_{p_\phi} - \mathcal{F}(\phi) (p_\theta - p_\psi) \partial_{p_\theta}\wedge \partial_{p_\phi}.
%$$
\begin{remark} The $G$-symmetry considered in this paper is different than the one considered in \cite{BalYapu19, balseiro2014}, therefore the reduced bracket obtained here is not the same as the one presented in these citations.  Moreover, in \cite{BalYapu19,balseiro2014}, the snakeboard was described by a twisted Poisson bracket (with a 4-dimensional foliation) while here, we  show that the snakeboard can be described by a rank 2-Poisson bracket. 
\end{remark}

\medskip

\noindent{\bf The horizontal gauge momenta as parallel sections}. %The horizontal gauge symmetry $\zeta$ can be computed as the parallel transport of elements $\zeta_0\in \g_S|_{q_0}$ with respect to the $\Sigma$-connection $\overset{\textit{\tiny{$\Sigma$}}}{\nabla} = \hat{\nabla} + \Sigma$ (see Definition~\ref{Def:SigmaConnection}).  
Consider the basis $\bar{\mathfrak{B}}_{TQ}=\{Y_1:=Y_\theta, Y_2:=\partial_\psi, X_0:=\partial_\phi, \bar{Z}_1, \bar{Z}_2\}$ where $\bar{Z}_1, \bar{Z}_2$ generate the distribution $W= S^\perp \cap V$. %(observe that we do not need to compute $W$ explicitly).   
The Christoffel symbols of the affine connection $\hat\nabla$ % : \mathfrak{X}(Q) \times \gamma(\g_S) \to \gamma(\g_S)$ 
coincide with the ones of the Levi-Civita connection and then
$$
\hat\Gamma_{01}^1 = -\tfrac{\bJ \sin\phi\, \cos\phi}{\Delta}\,,\qquad 
\hat\Gamma_{01}^2 = \tfrac{m r^2 \cos\phi}{\Delta} 
\qquad \mbox{and} \qquad \hat\Gamma_{02}^1 = - \hat \Gamma_{02}^2 = 0.
$$
Following Def.~\ref{Def:SigmaConnection} we get that $\Sigma = \Sigma^\theta \otimes \xi_1 + \Sigma^\psi \otimes \xi_2 = 0$.
Therefore, $\overset{\textit{\tiny{$\Sigma$}}}{\nabla} = \hat{\nabla}$ and then the horizontal gauge symmetries $\zeta = f_1(\phi)\xi_1 + f_2(\phi)\xi_2$ is determined by the condition that they are parallel along the dynamics with respect to the $\hat{\nabla}$ connection, i.e., 
\begin{equation}\label{Ex:Snake:ParTransp}
\hat \nabla_{\dot \gamma} \zeta = 0,
\end{equation}
for $\dot \gamma = T\tau_\subM(X_\nh)$.  %Following the proof of Theorem~\ref{T:ParallelTransport}, it is straightforward to see that for $\dot \gamma = v^\phi \partial_\phi + v^\theta Y_\theta + v^\psi \partial_\psi$, we obtain that \eqref{Ex:Snake:ParTransp} is equivalent to the system of differential equations \eqref{Ex:Snake:ODE}.  

\subsection{Solids of Revolution} \label{Ex:Solids}
Let $\mathcal{B}$ be a strongly convex body of revolution, i.e., a body which is geometrically and dynamically symmetric under rotations about a given axis 
(\cite{CDS,balseiro2017}). Let us assume that the surface ${\bf S}$ of $\mathcal{B}$ 
is invariant under rotations around a given axis, which in our case is chosen to be $e_3$. Then its principal moments 
of inertia are $\mathbb{I}_1=\mathbb{I}_2$ and $\mathbb{I}_3$. 

%% =======================================================
\begin{figure}[h]%[!htbp]
\begin{center}
%\vskip 5truecm
%\begin{picture}
{\small
%\put(45,0)
{\scalebox{.3}{\includegraphics*{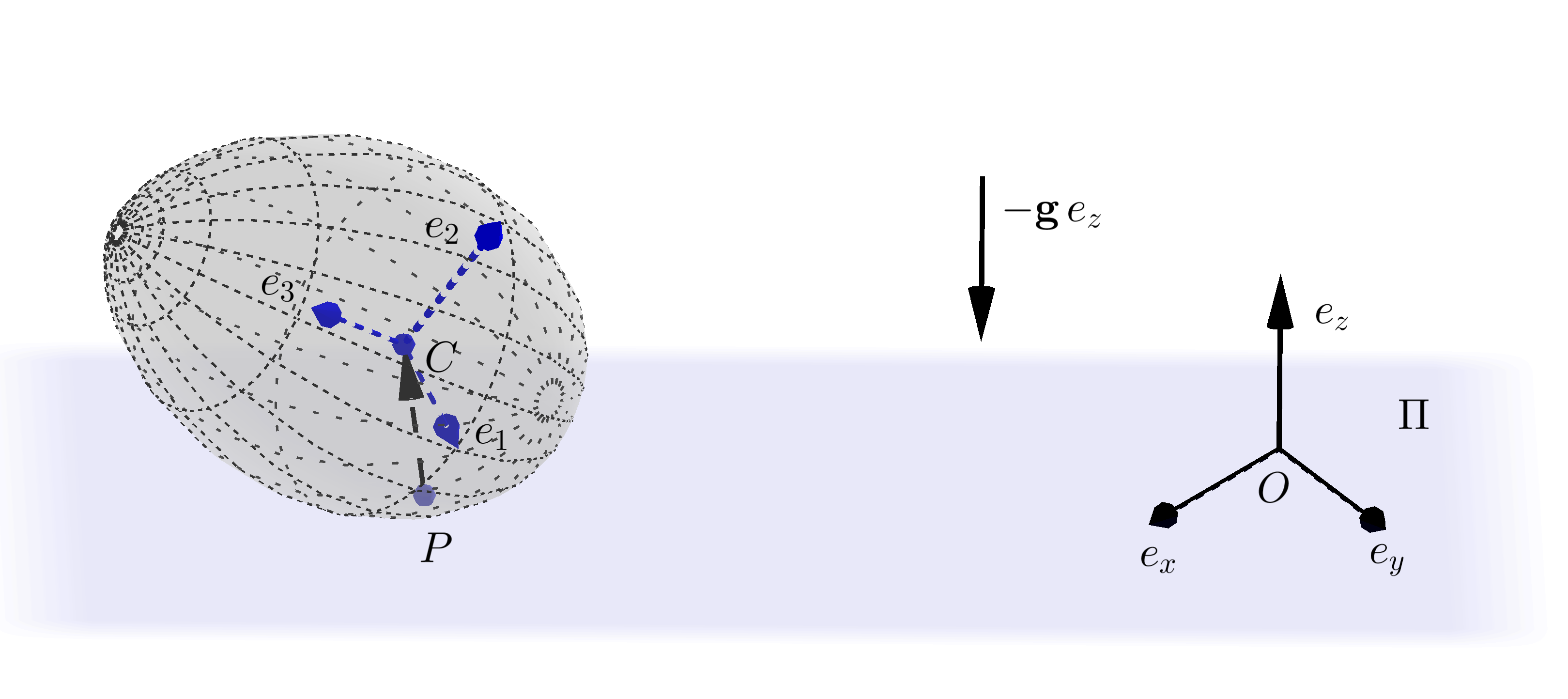}}}
}
%\end{picture}
\caption{\small Solid of revolution rolling on a horizontal plane.}
\end{center}
\end{figure}
%% =======================================================

The position of the body in $\R^3$ is given by the coordinates $(g, {\bf x})$ where 
$g\in SO(3)$ is the orientation of the body with respect to an inertial frame $(e_x,e_y,e_z)$ and 
${\bf x} = (x,y,z) \in \R^3$ is the position of the center of mass. Denoting by ${\bf m}$ 
the mass of the body, the lagrangian $L:T(SO(3)\times \R^3) \to \R$ is given by 
\[
L(g,{\bf x}; \vecOm, \dot{\bf x}) = \frac{1}{2}\langle \mathbb{I} \vecOm, \vecOm\rangle + \frac{1}{2} {\bf m} ||\dot{\bf x}||^2 + {\bf m}{\bf g}\langle {\bf x}, e_3\rangle,
\]
where $\vecOm = (\Omega_1, \Omega_2, \Omega_3)$ is the angular velocity in body coordinates, $\langle \cdot, \cdot \rangle$ represents the standard pairing in $\R^3$ and ${\bf g}$ the constant of gravity.  

Let $s$ be the vector from the center of mass of the body to a fixed point on the surface ${\bf S}$. If we denote by $\vecgamma=(\gamma_1, \gamma_2, \gamma_3)$ the third row of the matrix $g\in SO(3)$, then $s$ can be written as $s:S^2 \to {\bf S}$ so that 
\[
s(\vecgamma) = (\varrho(\gamma_3) \gamma_1, \varrho(\gamma_3)\gamma_2, \zeta(\gamma_3)),
\] 
where $\varrho$ and $\zeta$ are the smooth functions defined in \cite{CDS}. Therefore
\[
s(\vecgamma) = \varrho \vecgamma - Le_3,
\] 
where $\varrho = \varrho(\gamma_3)$, $\zeta=\zeta(\gamma_3)$ and $L= L(\gamma_3) = \varrho\gamma_3 - \zeta$.  The configuration space is described as
 $$
Q=\{(g,{\bf x})\in SO(3)\times \R^3 \ : \ z = - \langle \vecgamma, s \rangle \},
$$ 
and it is diffeomorphic to $SO(3)\times \R^2$. 
The nonholonomic constraint describing the rolling without sliding are written as 
$$
\vecOm\times s + {\bf b} = 0,
$$
where ${\bf b} = g^t\dot{\bf x}$ (with $g^t$ the transpose of $g$).  

Let us consider the (local) basis of $TQ$ given by $\{X_1^L, X_2^L, X_3^L, \partial_{x}, \partial_{y}\}$, where $X_i^L$ are the left invariant vector fields on $SO(3)$ and we denote the corresponding coordinates on $TQ$ by $(\vecOm, \dot x, \dot y).$ Then the constraint distribution $D$ is given by
$ D = \textup{span}\{ X_1, X_2,X_3\}$ where 
$$
X_i := X_i^L + (\vecalpha \times s)_i \partial_{x} + (\vecbeta \times s)_i \partial_{y} + (\vecgamma\times s)_i\partial_{z},
$$
for $\vecalpha$ and $\vecbeta$ the first and second rows of the matrix $g\in SO(3)$. 
The constraints 1-forms are 
$$
\epsilon^1 = dx -\langle \vecalpha, s\times \vecL \rangle \quad \mbox{and}\quad \epsilon^2= dy - \langle \vecbeta, s \times \vecL\rangle, 
$$
where $\vecL = (\lambda_1, \lambda_2, \lambda_3)$ are the (Maurer-Cartan) 1-forms on $SO(3)$ dual to the left invariant vector fields $\{X_L^1, X_L^2, X_L^3\}$. 

\noindent {\bf The symmetries}. The Lagrangian and the constraints are invariant with respect 
to the action of  the special Euclidean 
group $SE(2)$ acting on $Q$, at each $(g;x,y)\in Q$, by
$$
\Psi ((h;a,b)), (g;x,y)) = (\tilde h.g; h. (x,y)^t + (a,b)^t)\,,
$$
where $h\in SO(2)$ is an orthogonal $2\times 2$ matrix and 
$\tilde h = {\mbox{\scriptsize{$\left(\begin{array}{cc} h & 0 \\[-3pt] 0 & 1 \end{array}\right)$}}} \in SO(3)$.  
The symmetry of the body makes also the system invariant with respect to 
the right $S^1$-action on $Q$ given by $\Psi_{S^1}(h_\theta, (g,x,y)) = (g\tilde h_\theta^{-1}, h_\theta (x,y)^t)$, 
where we identify $\theta\in S^1$ with the orthogonal matrix $h_\theta\in SO(2)$.  

Therefore, the symmetry group of the system is the Lie group $G = S^1\times SE(2)$, 
with associated Lie algebra $\mathfrak{g}\simeq \R\times \R\times \R^2$. 
The vertical space $V$ is given by 
$$
V=\textup{span}\{ (\eta_1)_\subQ = -X_3^L -y\partial_{x} + x \partial_{y}, \ (\eta_2)_\subQ = \langle \vecgamma,{\bf X}^L \rangle - y \partial_{x} + x \partial_{y}, \ (\eta_3)_\subQ = \partial_{x} , \  (\eta_4)_\subQ = \partial_{y} \},
$$
where $\eta_i$ are the canonical Lie algebra elements in $\mathfrak{g}$ and ${\bf X}^L = (X_1^L,X_2^L,X_3^L)$. 
We observe that the action is not free, since $(\eta_i)_\subQ(g,x,y) $ 
are not linearly independent at $\gamma_3=1$.  We check that the dimension 
assumption \eqref{Eq:dim-assumption} is satisfied: $TQ = D +V$. Let us choose 
$W = \textup{span}\{ \partial_{x}, \partial_{y}\}$ as vertical complement of the constraints 
and then the basis of $TQ$ adapted to the splitting 
\eqref{Eq:TQ=D+W} is ${\bf B}_{TQ} = \{X_1, X_2, X_3, \partial_{x}, \partial_{y}\}$,  
with  dual basis given by ${\bf B}_{T^*Q} = \{\lambda_1, \lambda_2, \lambda_3, \epsilon^1, \epsilon^2\}$. 
The associated coordinates on $T_q^*Q$ are $({\bf M}, K_1,K_2)$ for ${\bf M} = (M_1, M_2,M_3)$ 
and the submanifold $\M$ of $T^*Q$ is then described by 
\begin{equation}\label{Ex:Solids:M}
\M = \{(g,x,y; {\bf M}, K_1, K_2) \ : \ K_1 =  {\bf m} \langle \vecalpha, s\times \vecOm \rangle, \quad K_2 = {\bf m} \langle \vecbeta , s \times \vecOm\rangle \},
\end{equation}
where ${\bf M} = \mathbb{I} \vecOm + m s\times (\vecOm \times s)$.
The horizontal gauge momenta are functions on $\M$ linear in the coordinates $M_i$. 

\noindent {\bf The existence of horizontal gauge momenta.}
First, we observe that the $G$-action satisfies Conditions $(\cA 1)$-$(\cA 4)$ outside $\gamma_3= \pm 1$ and thus, in what follows, we will work on the manifolds $\widetilde Q \subset Q$ and $\widetilde \M\subset \M$ defined by the condition $\gamma_3 \neq \pm 1$. 
Second, we consider the splitting 
\begin{equation}\label{Ex:Solids:H+S+W} 
T\widetilde Q = H \oplus S\oplus W, 
\end{equation}
where $S=D\cap V = \textup{span}\{Y_1:= X_3, Y_2:= \langle \gamma, {\bf X}\rangle \}$, with ${\bf X} = (X_1, X_2,X_3)$ and $H$ is generated by $X_0= \gamma_1 X_2 - \gamma_2X_1$ (observe that $H= S^\perp\cap D$).  Now, we check that the kinetic energy is strong invariant on $S$: in this case, it is enough to see that $\kappa([Y_1,Y_2],Y_1) = 0$ and $\kappa([Y_1, Y_2], Y_2)=0$. These two facts are easily verified using simply that $[X_i^L, X_j^L]=X_k^L$ for $i,j,k$ cyclic permutations of $1,2,3$. In the same way, we also check that $\kappa(X_0, [Y_i, X_0])=0$, for $i=1,2$.   Therefore, by Theorem \ref{T:Main}, we conclude that the system admits $2=\textup{rank}(S)$ $G$-invariant (functionally independent) horizontal gauge momenta $\mathcal{J}_1$, $\mathcal{J}_2$ on $\widetilde\M$ (recovering the results in \cite{BMK2002,CDS}). 

\noindent {\bf The computation of the 2 horizontal gauge momenta}. In order to compute the horizontal gauge momenta, we consider the basis $\mathfrak{B}_{\g_S}$ of $\Gamma(\g_S \to \widetilde Q)$, defined by 
$$
\mathfrak{B}_{\g_S} = \{\xi_1:= (1;0,(h_1,h_2)), \xi_2 := (0;1,(g_1,g_2))\},
$$
where $h_1= h_1(g,x,y) = y + \varrho\beta_3$, $h_2 = h_2(g,x,y) = -x-\varrho \alpha_3$ and $g_1= g_1(g,x,y) = y -L\beta_3$, $g_2 = g_2(g,x,y) = -x +L\alpha_3$. %Moreover, $(\xi_1)_\subM = -X_3 -(M_2\partial_{M_1} - M_1\partial_{M_2})$ and $(\xi_2)_\subM = \langle \vecgamma, {\bf X}\rangle$. 
The components of the nonholonomic momentum map, in the basis $\mathfrak{B}_{\g_S}$, are given by 
$$
J_1 = \langle {\mathcal J}^\nh, \xi_1 \rangle = {\bf i}_{(\xi_1)_\subM} \Theta_\subM =  - M_3 \qquad \mbox{and}\qquad J_2 = \langle {\mathcal J}^\nh, \xi_2 \rangle = {\bf i}_{(\xi_2)_\subM} \Theta_\subM = \langle \vecgamma, {\bf M}\rangle,
$$
where we are using that $(\xi_1)_Q= Y_1$ and $(\xi_2)_Q= Y_2$, see \eqref{Eq:J_i}.   Then, a function $\mathcal{J}= f_1J_1 + f_2J_2$ is a horizontal gauge momentum if and only if the coordinate functions $(f_1,f_2)$ satisfy the {\it momentum equation} \eqref{Eq:MomEq}
$$
f_1\langle J, \sigma_{\g_S} \rangle (\mathcal{Y}_1, X_\nh) + f_2\langle J, \sigma_{\g_S}\rangle (\mathcal{Y}_2, X_\nh) -M_3  X_\nh(f_1) + \langle \vecgamma, {\bf M}\rangle X_\nh(f_2) = 0.
$$
That is, considering the basis, $\mathfrak{B}_{T\widetilde Q}  = \{ X_0 , \ Y_1, \ Y_2  , \ \partial_x, \ \partial_y\}$, the $G$-invariant coordinate functions $(f_1 = f_1(\gamma_3), f_2 = f_2(\gamma_3))$ are the solutions of the system of ordinary differential equations (defined on $\widetilde Q/G$)
\begin{equation}\label{Ex:Solids:ODE}
R \left( \!\!\begin{array}{c} f_1 \\ f_2 \end{array} \! \! \right )= \left(\!\!\begin{array}{c} \bar{X}_0(f_1) \\ \bar{X}_0(f_2) \end{array} \!\! \right), \qquad \mbox{for} \ R=[\kappa|_S]^{-1} [N],  
\end{equation}
% $$
% \left( \begin{array}{cc} R_1^1 & R_1^2 \\ R_2^1 & G^2_2 \end{array} \right) \left( \begin{array}{c} f_1 \\ f_2 \end{array} \right) = \left( \begin{array}{c} f'_1 \\ f'_2 \end{array} \right),
% $$
where $\bar{X}_0 = T\rho_{\widetilde Q} (X_0) = (1-\gamma_3^2) \partial_{\gamma_3}$, the matrix $[N]$ has elements $N_{lj} = \kappa(Y_l,[Y_i, X_0]) - \kappa(X_0, [Y_i, Y_l])$ that in this case gives
\begin{equation*}%\label{Ex:Solids:MatrixH} 
[N] = m(1-\gamma_3^2)\left( \!\!\begin{array}{cc} -\varrho A   & \varrho (B -\langle \vecgamma, s\rangle) \\ L A - \varrho \langle \vecgamma, s\rangle   & -L B \end{array} \!\! \right ) 
\end{equation*}
for $A = \varrho' (1-\gamma_3^2) - \varrho \gamma_3$ and $B = L' (1-\gamma_3^2) - L\gamma_3- \langle \vecgamma, s\rangle $ (with $(\cdot)' = \tfrac{d}{d\gamma_3} (\cdot)$) and 
$$[\kappa|_S]  = \left(\!\! \begin{array}{cc} \mathbb{I}_3 +m\varrho^2 (1-\gamma_3^2) & -\mathbb{I}_3\gamma_3 - Lm\varrho (1-\gamma_3^2) \\ 
-\mathbb{I}_3\gamma_3 - Lm\varrho (1-\gamma_3^2) &\langle \vecgamma, \mathbb{I}\vecgamma\rangle + L^2m(1-\gamma_3^2) \end{array} \!\! \right).
$$
The system \eqref{Ex:Solids:ODE} admits two independent solutions $\bar{f}^1=(\bar{f}^1_1,\bar{f}^1_2)$ and $\bar{f}^2=(\bar{f}^2_1,\bar{f}^2_2)$ on $\widetilde Q/G$ and therefore we conclude that the two ($G$-invariant) horizontal gauge momenta $\mathcal{J}_1$ and ${\mathcal J}_2$ are  
\begin{equation}\label{Ex:Solids:HGM}
\mathcal{J}_1 = -f^1_1M_3 + f^1_2\langle \vecgamma, {\bf M}\rangle \quad \mbox{and} \quad {\mathcal J}_2 = -f^2_1M_3 + f^2_2\langle \vecgamma, {\bf M}\rangle,
\end{equation}
where $f^i_j = \rho^*\bar{f}^i_j$ for $i,j=1,2$.

\begin{remark}
\begin{enumerate}
  \item[$(i)$] For $f = (f_1, f_2)$, the system \eqref{Ex:Solids:ODE} is equivalently written as 
  $(1-\gamma_3^2)^{-1}R f = f'$. Therefore, we recover the system of ordinary 
  differential equations from \cite{BMK2002,CDS,balseiro2017} (and \cite{BGM96} 
  for the special case of the Tippe-Top and of the rolling disk).
  %\textcolor{blue}{We also observe that the horizontal gauge momenta $\cJ_1$ and $\cJ_2$ were  obtained in  for the special case}
  \item[$(ii)$] The $G$-invariant horizontal gauge momenta ${\mathcal J}_1$, ${\mathcal J}_2$ 
  descend to the quotient $\widetilde \M/G$ as functions $\bar{\mathcal J}_1$, $\bar{\mathcal J}_2$ 
  that are functionally independent. It has been proven in \cite{CDS} that the functions 
  $\bar{\mathcal J}_1$, $\bar{\mathcal J}_2$ can be extended to the whole {\it differential space} $\M/G$.  In this case, it makes sense to talk about $2=\textup{rank}(\g_S)$ horizontal gauge momenta.  
\end{enumerate}    
\end{remark}

\medskip

\noindent{\bf Integrability and hamiltonization.} 
The nonholonomic dynamics $X_\nh$ defined on $\widetilde\M$ can be reduced to $\widetilde\M/G$ obtaining the vector field $X_\red$ (see \eqref{Eq:RedDyn}).
Using the basis $\mathfrak{B}_{T\widetilde Q}=  \{ X_0 , Y_1, Y_2, \partial_x, \partial_y\}$ and its dual basis of $T^*\widetilde{Q}$
\begin{equation}\label{Ex:Solids:BasisTQ}
  \mathfrak{B}_{T^*\widetilde Q}  = \left\{ X^0 := \frac{\gamma_1\lambda_2 - \gamma_2\lambda_1}{1-\gamma_3^2}, \ Y^1 := \gamma_3\frac{ \gamma_1 \lambda_1+\gamma_2\lambda_2}{1-\gamma_3^2} - \lambda_3, \ Y^2 := \frac{ \gamma_1 \lambda_1+\gamma_2\lambda_2}{1-\gamma_3^2}, \ \epsilon^1, \ \epsilon^2   \right\},
 \end{equation}
 we denote by $(v^0,v^1, v^2,v^x, v^y)$ and $(p_0,p_1,p_2,K_1, K_2)$ the associated coordinates on $T\widetilde Q$ and $T^*\widetilde Q$ respectively. The reduced manifold $\widetilde \M/G$ is represented by the coordinates  $(\gamma_3, p_0,p_1,p_2)$.
 
\medskip

%\begin{equation*}
%\bar{\M}_{sing} = 
%\{ (\pm 1, 0, 0, \tau_4, 0) \in \M/G \: : \: \tau_4 \in \R \},
%\end{equation*}
%and corresponds to the configuration where the body of revolution is spinning over one of the two poles which remains fixed on the plane.

%\begin{remark}
% If we add the variable  we obtain that $\M/G$ is 
%represented by the following semi-algebraic set of $\R^5$
%\begin{equation*}
%\M/G = \{ (\gamma_3,p_0,p_1,p_2,p_3) \in \R^5 \: : \: |\gamma_3|\leq 1, p_3 \geq 0, p_0^2 + p_1^2 = (1-\gamma_3^2) p_3 \},
%\end{equation*}
%recovering the results in \cite{CDS}. 
%\end{remark}

\noindent {\it Integrability.}  
Theorem~\ref{T:reduced-integrability} guarantees that the reduced system on 
$\widetilde \M/G$ admits three functionally independent first integrals, namely 
two horizontal gauge momenta $\bar{\mathcal J}_1$ and $\bar{\mathcal J}_2$, 
and the reduced energy $H_\red$. 
Since $\textup{dim}(\widetilde \M/G) = 4$, the reduced dynamics is integrable by quadratures. 
However, the reduced dynamics is not generically
periodic, and therefore we can say nothing generic on the complete dynamics or on the geometry of the phase space.

\medskip

\noindent {\it Hamiltonization.} Even though the hamiltonization of this example has been studied in \cite{balseiro2017,GNMontaldi}, here we see it as a direct consequence of Theorem \ref{T:Main}.  
That is, since this nonholonomic system satisfies the hypotheses of Theorem \ref{T:Main}, 
it is {\it hamiltonizable by a gauge transformation} (Def.~\ref{Def:Hamiltonization}). 
The reduced bracket $\{\cdot, \cdot \}_\red^{B_{\mbox{\tiny{HGM}}}}$ on $\widetilde \M/G$ 
defines a rank-2 Poisson structure, with 2-dimensional leaves given by the common level 
sets of $\bar{\mathcal J}_1$ and $\bar{\mathcal J}_2$, 
that describes the (reduced) dynamics.

In what follows we show how the 2-form $B_{\mbox{\tiny{HGM}}}$, inducing the dynamical gauge transformation that defines $\{\cdot, \cdot \}_\red^{B_{\mbox{\tiny{HGM}}}}$, depends directly on the ordinary system of differential equations \eqref{Ex:Solids:ODE}.
Consider the basis $\mathfrak{B}_{T\widetilde Q}$ and $\mathfrak{B}_{T^*\widetilde Q}$ given in  \eqref{Ex:Solids:BasisTQ} and following Theorem~\ref{T:FormulaB},
$$
B_{\mbox{\tiny{HGM}}} = \langle J,\sigma_{\mbox{\tiny{HGM}}} \rangle = \langle J, \mathcal{K}_\subW\rangle - J_i R_{ij} \mathcal{X}^0 \wedge \mathcal{Y}^j +J_id\mathcal{Y}^i,
$$
where $\mathcal{X}^0=\tau^*_{\tilde\subM} X^0$ and $\mathcal{Y}^i=\tau^*_{\tilde\subM} Y^i$ for $i=1,2$ 
are the corresponding 1-forms on $\widetilde\M$.
Using \eqref{Ex:Solids:M} we have that (see \cite{balseiro2017}),  
\begin{equation*}
 \begin{split}
\langle J, \mathcal{K}_\subW\rangle|_\C & = K_1 \, d\epsilon^1|_\C + K_2\, d\epsilon^2|_\C \\
& =  m\varrho \langle \vecgamma, s\rangle \langle \vecOm, d\vecL\rangle - m (\varrho^2 \langle \vecOm, \vecgamma\rangle + \varrho' c_3) \langle \vecgamma, d\vecL\rangle + m(\varrho L \langle \vecOm, \vecgamma\rangle + L'c_3) d\lambda_3|_\C .
\end{split}
\end{equation*}
Now, recalling the definition of $X^0$, $Y^1$ and $Y^2$ in $\mathfrak{B}_{T^*Q}$ \eqref{Ex:Solids:BasisTQ}, we compute the term 
\begin{equation*}
 \begin{split}
 J_i R_{ij} {\mathcal X}^0 \wedge \mathcal{Y}^j & =J_i \, R_{i1} {\mathcal X}^0 \wedge \mathcal{Y}^1 + J_i\,  R_{i2} {\mathcal X}^0 \wedge \mathcal{Y}^2 \\
%& =(1-\gamma_3^2)^{-1} (  J_l \,\kappa^{lj}H_{j1}d\gamma_3 \wedge \mathcal{Y}^1 + J_l\,  \kappa^{lj}H_{j2}d\gamma_3\wedge \mathcal{Y}^2) \\ 
& = (1-\gamma_3^2)^{-1} (v^l N_{l1}\langle \vecgamma , d\vecL\rangle + v^l N_{l2}\, d\lambda_3), \\
& = - m (\varrho^2 \langle \vecOm, \vecgamma\rangle + \varrho' c_3) \langle \vecgamma, d\vecL\rangle + m(\varrho L \langle \vecOm, \vecgamma\rangle + L'c_3) d\lambda_3.
 \end{split}
\end{equation*}
where we use that $v^1 = (1-\gamma_3^2)^{-1} (\langle \vecgamma,\vecOm\rangle \gamma_3 - \Omega_3)$ and $v^2= (1-\gamma_3^2)^{-1} (\langle \vecgamma,\vecOm\rangle  - \gamma_3\Omega_3)$. 
Finally, since $dY^i = 0$ for $i=1,2$, we obtain that 
$$B_{\mbox{\tiny{HGM}}} = m\varrho \langle \vecgamma, s\rangle \langle \vecOm, d\vecL\rangle,$$ 
recovering the dynamical gauge transformation from \cite{balseiro2017,GNMontaldi}.  
For the explicit formulas for the brackets, see \cite{balseiro2017}.

\medskip

%By completness we write the reduced equation of motion of this example 
% \begin{equation}\label{Ex:Solid:RedDyn}\begin{aligned}
%\dot \gamma_ 3   = &   , \\
%  \dot p_0 = &   , \\
%  \dot p_1 = &    ,\\
 % \dot p_2 = &   ,
%\end{aligned}\end{equation}
\begin{remarks}
\begin{enumerate}
\item[$(i)$] Since the $G$-action on $\M$ is proper but not free, the quotient $\M/G$ is a stratified differential space, \cite{CDS,balseiro2017} with a 4 dimensional regular stratum given by $\widetilde \M/G$ and a 1-dimensional singular stratum, associated to $S^1$-isotropy type, that is described by the condition $\gamma_3 =\pm 1$. Moreover, the relation between the coordinates on $T^*\widetilde Q$ relative to the basis ${\bf B}_{T^*\widetilde Q}$ and $\mathfrak{B}_{T^*\widetilde Q}$ is 
\begin{equation*}
p_0 = \gamma_1 M_2 - \gamma_2 M_1, \quad p_1 = \gamma_1 M_1 + \gamma_2 M_2, \quad p_2 =  M_3,
%\end{split}
\end{equation*}
Therefore, adding $p_3 = M_1^2+ M_2^2$, we conclude that the coordinates $(\gamma_3, p_0,p_1,p_2,p_3)$ on $\M/G$ are the same coordinates used in \cite{CDS,cushman1998}.

\item[$(ii)$] It is straightforward to write the equations of motion on $\widetilde \M/G$ in the variables $(\gamma_3, p_0,p_1,p_2)$ for the reduced hamiltonian  $H_\red$ recovering the equations in \cite{CDS,cushman1998}. This equations can be used to check the results in this section, however we stress that there is no need to compute them to find the horizontal gauge momenta, nor to study the integrability or the hamiltonization of the system.

\item[$(iii)$] The Routh sphere, the ellipsoid rolling on a plane and the falling disk \cite{cushman1998,CDS,BKK2015}, are seen as particular cases of this example. 
\end{enumerate}
\end{remarks}

\medskip

\noindent{\bf The horizontal gauge momenta as parallel sections}.  %Next, following Section~\ref{Ss:ParallelTransport},  we study the horizontal gauge symmetries $\zeta_1, \zeta_2$ as parallel sections with respect to the $\Sigma$-connection $\overset{\textit{\tiny{$\Sigma$}}}{\nabla} := \hat \nabla + \Sigma$ (see Definition~\ref{Def:SigmaConnection}). 
Let us consider the basis $\mathfrak{B}_{T\widetilde{Q}} =\left\{X_0, Y_1,Y_2,\right.$ $\left. Z_1,Z_2\right\}$ where $X_0, Y_1,Y_2$ are the vector fields defined previously but $Z_1,Z_2$ generate the distribution $W$ which, now, is chosen to be $W=S^\perp \cap V$. %Using the relation of the affine connection $\hat \nabla: \mathfrak{X}(Q)\times \Gamma(\g_S)\to \Gamma(\g_S)$ with the Levi-Civita connection we get that t
The Christoffel symbols of $\hat \nabla$, in the basis $\mathfrak{B}_{T\widetilde{Q}}$ and $\mathfrak{B}_{\g_S}$, are given by 
{\small{ 
$$
\left(\!\!\begin{array}{c} \hat \Gamma_{01}^1 \\ \hat \Gamma_{01}^2 \end{array} \!\! \right) = \frac{1}{2} [\kappa|_S]^{-1} \!\! \left( \!\!\!\!\begin{array}{c} \kappa'_{11} (1-\gamma_3^2) \\ (\kappa'_{12} +m\varrho B)(1-\gamma_3^2) - H_{21} \end{array} \! \!\!\! \right )  \mbox{ \ and \  }  \left(\!\!\begin{array}{c} \hat \Gamma_{02}^1 \\ \hat \Gamma_{02}^2 \end{array} \!\! \right) = \frac{1}{2}[\kappa|_S]^{-1}\!\! \left( \!\!\!\!\begin{array}{c}  (\kappa'_{12} +m A L)(1-\gamma_3^2) - H_{12}   \\      \kappa'_{22} (1-\gamma_3^2)  \end{array} \!\!\! \! \right ),
$$
}}
and $\hat \Gamma_{ij}^1=\hat \Gamma_{ij}^2 = 0$.  
Following Def.~\ref{Def:SigmaConnection}, the bilinear form $\Sigma = \Sigma^1 \otimes \xi_1 + \Sigma^2\otimes \xi_2$ is given by
\begin{equation*}
  \Sigma^1  = -(\hat \Gamma_{0j}^1 + R_{1j}) X^0 \wedge Y^j \quad \mbox{and} \quad  \Sigma^2  = -(\hat \Gamma_{0j}^2 + R_{2j}) X^0 \wedge Y^j,
\end{equation*}
where the functions $R_{ij}$ are given in \eqref{Ex:Solids:ODE}. Then, the horizontal gauge symmetries can be seen as parallel sections along the dynamics with respect to the $\Sigma$-connection: 
% Then, we can check that the section  $\zeta = f_1(\gamma_3)\xi_1+ f_2(\gamma_3)\xi_2$ is parallel, i.e., it satisfies that 
$$
\overset{\textit{\tiny{$\Sigma$}}}{\nabla}_{\dot \gamma}\zeta = 0.
$$ %if and only the functions $f_1, f_2$ satisfy the system of differential equations \eqref{Ex:Solids:ODE}. Since the condition $\kappa(X_0,[Y_i,X_0])=0$ is satisfied, we conclude that the parallel section $\zeta$ is a horizontal gauge symmetry. 

\subsection{A homogeneous ball on a surface of revolution} \label{Ex:BallSurface}

%\textcolor{blue}{I changed the size of the figure. Now, we have to accomodate the letters (I also took out $\tilde \Sigma$ because it is not used here)}
%\textcolor{red}{$\tilde \Sigma$ is indeed used, because the ball rolls without sliding on it, not on $\Sigma$}

Let us consider the holonomic system formed by a homogeneous sphere of mass ${\bf m}$ 
and radius $r>0$, which center $C$ is constrained to belong to a convex surface of revolution 
$\Sigma$ (i.e., the ball rolls on the surface $\tilde\Sigma$, see Figure \ref{Fig:BallonSurface}). 
The surface $\Sigma$ is obtained by rotating about the $z$-axis the graph of a 
convex and smooth function $\phi: \bR{}_+ \longrightarrow \bR{}$. 
Thus, $\Sigma$ is described by the equation $z = \phi(x^2+y^2)$. 
To guarantee smoothness and convexity of the surface, 
we assume that $\phi$ verifies that $\phi'(0^+) = 0$, $\phi'(s)>0$ and $\phi''(s)>0$, when $s>0$. 
To ensure that the ball has only one contact point with the surface we ask the curvature of 
$\phi(s)$ to be at most 1/r.
The configuration manifold $Q$ is $\bR{2}\times SO(3)$ with coordinates 
$(x,y,g)$ where $G$ is the orthogonal matrix fixing the attitude of the sphere and $(x,y)$ are the coordinates of $C$  with respect to
a reference frame with origin $O$ and $z$-axis coinciding with the figure axis of $\Sigma$.

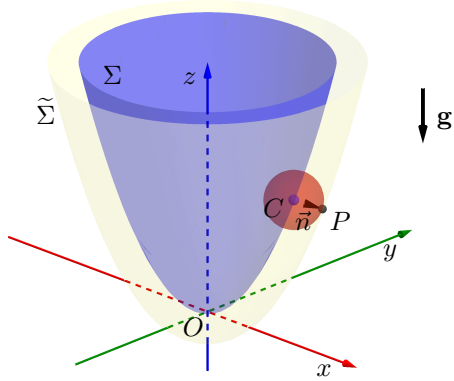
\begin{figure}[h]%[!htbp]
\begin{center}
%\vskip 5truecm
%\begin{picture}
{\small
%\put(45,0)
{\scalebox{.4}{\includegraphics*{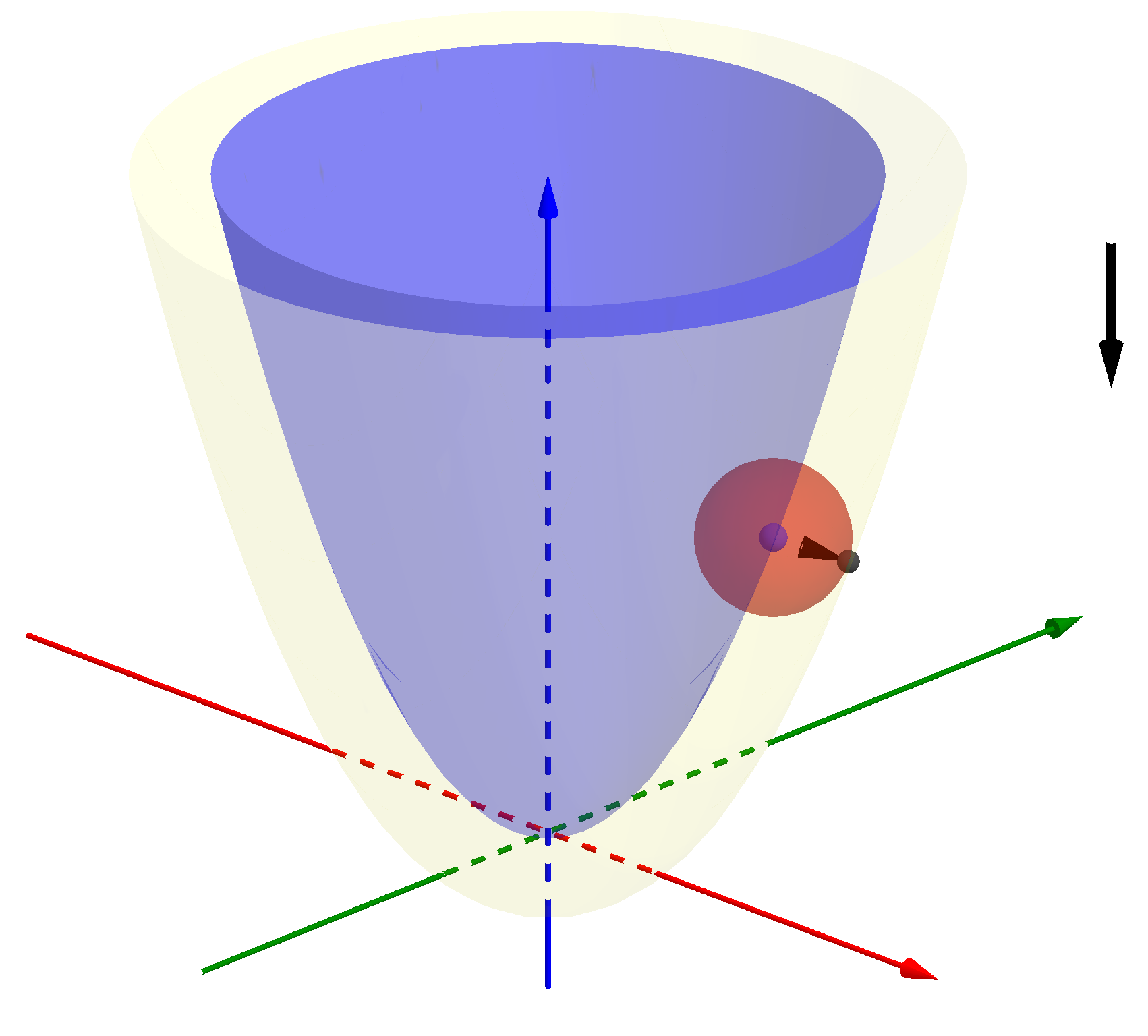}}}
\put(-150,100){$\widetilde\Sigma$}
\put(-125,115){$\Sigma$}
\put(-45,5){$x$}
\put(-20,50){$y$}
\put(-95,115){$z$}
\put(-95,20){$O$}
\put(-65,65){$C$}
\put(-53,60){$\vec n$}
\put(-40,60){$P$}
\put(0,100){$\bf g$}
}
%\end{picture}
\end{center}
\caption{\small The homogeneous ball on a convex surface of revolution.}
\label{Fig:BallonSurface}
\end{figure}
%% =======================================================

Let us denote by $n = n(x,y) $ the outward normal unit vector to $\Sigma$ with components $(n_1, n_2, n_3)$ given by
$$
\frac{n_1}{n_3} = 2x\phi', \quad \frac{n_2}{n_3} = 2y\phi' \quad \mbox{and} \quad n_3 = -\frac{1}{\sqrt{1+4(x^2 +y^2)(\phi')^2}}. 
$$
If $\omega=(\omega_1,\omega_2,\omega_3)$ is the angular velocity of the ball in the space frame, then the Lagrangian of the holonomic system on $TQ$ is
\begin{equation}\label{Ex:Ball:Lagrangian}
  L(x,y,g,\dot x,\dot y,\omega) = \frac{{\bf m}}{2n_3^2} \left((1-n_2^2)\dot x^2 + 2 n_1n_2\, \dot x\dot y + \dot y^2(1-n_1^2)\right) + \frac{1}{2} \langle \mathbb{I} \omega,\omega\rangle - {\bf m} {\bf g} \phi\,,
\end{equation}
where  ${\bf g}$ denotes the gravity acceleration and $\bI{}$ the moment of inertia of the sphere with respect to its center of mass.

\noindent {\bf Geometry of the constrained system.} 
%Let $(\omega_1,\omega_3,\omega_3)$ the components of the angular velocity with respect to the right invariant basis $\{X_1^R,X_2^R,X_3^R\}$ of 
%$TSO(3)$, i.e. $(\omega_1,\omega_3,\omega_3)$ are the component of the
%angular velocity with respect to the space frame.  
The ball rotates without sliding on the surface $\widetilde \Sigma$, and hence the nonholonomic constraints equations are 
\[
  \dot x  = - r \left( \omega_2 n_3 - \omega_3 n_2 \right)\,,\qquad
  \dot y  = - r \left( \omega_3 n_1 - \omega_1 n_3\right).
\]
We denote by $\{X_1^R,X_2^R,X_3^R\}$ the right invariant vector fields on $SO(3)$ and by $\{\rho_1,\rho_2,\rho_3\}$ the right Maurer-Cartan 1-forms, that form a basis
of $T^*SO(3)$ dual to $\{X_1^R,X_2^R,X_3^R\}$. Then the constraint 1-forms are given by
\[
  \epsilon^1 := dx - r \left( n_2 \rho_3 - n_3\rho_2  \right)\,,\qquad 
  \epsilon^2 := dy - r \left( n_3 \rho_1 - n_1 \rho_3 \right)\,.
\]
The constraint distribution $D$ defined by the annihilator of $\epsilon^1$ and $\epsilon^2$ has fiber, at $q= (x,y,g)$, given by
\begin{equation}\label{Ex:Ball:D}
  D_{q} = \textrm{span}\left\{ Y_x := \partial_x  -\frac{1}{r n_3}(n_2 X_n - X_2^R), \   Y_y := \partial_y +\frac{1}{r n_3}(n_1X_n - X_1^R), \   X_n \right\}\,,
\end{equation}
where $X_n: = n_1\, X_1^R+n_2\,X_2^R+n_3\,X_3^R$.
%$n\cdot \bf{X}^R$, with $\bf{X}^R=(X_1^R,X_^R_2,X^R_3)$
Consider the basis of $TQ$
\begin{equation}\label{eq:TQ-basis}
  {\bf B}_{TQ} = \left\{ Y_x,Y_y,X_n,Z_1,Z_2 \right\}\,,
\end{equation}
where $  Z_1 := \frac{1}{r n_3} X_2^R-\frac{n_2}{r n_3} X_n$ and $Z_2 := -\frac{1}{r n_3} X_1^R+\frac{n_1}{r n_2} X_n$ with associated coordinates 
$(\dot x,\dot y,\omega_n,w^1,w^2)$, for $\omega_n = n\cdot \omega=n_i \omega_i$, the normal component of the angular velocity $\omega$.
The dual frame of \eqref{eq:TQ-basis} is
\begin{equation}\label{eq:T*Q-basis}
  {\bf B}_{T^*Q} = \left\{ dx,dy, \rho_n,\epsilon^1,\epsilon^2, \right\}\,,
\end{equation}
where $\rho_n = n_i\rho_i$, with associated coordinates $(p_x,p_y,p_n,M_1,M_2)$ on $T^*Q$.  
The manifold $\M = \kappa^\sharp(D)$ is given by 
$$
\M = \left\{(x,y,g;p_x,p_y,p_n,M_1,M_2) \ : \ M_1 = \tfrac{-I}{I+mr^2}p_x, \ M_2 = \tfrac{-I}{I+mr^2}p_y \right\}.
$$

\noindent {\bf The symmetries.}
Consider the action $\Psi$ of the Lie group $G = SO(2)\times SO(3)$ on the manifold $Q$ 
given, at each $(x,y,g)\in Q$ and $(h_\theta, h)\in SO(2) \times SO(3)$,  by 
$$
\Psi_{(h_\theta,h)} (x,y,g) = (h_\theta (x,y)^t,  \tilde h_\theta g h ), 
$$
where $\tilde h_\theta$ is the $3\times 3$ rotational matrix of angle $\theta$ with respect to the $z$-axis.  In other words, $SO(3)$ acts on the right on itself and $SO(2)$ acts by rotations about the figure axis of the surface $\Sigma$. 
The Lagrangian \eqref{Ex:Ball:Lagrangian} and the constraints \eqref{Ex:Ball:D} are invariant with respect to the lift of this action to $TQ$ given by $\Psi_{(h_\theta,h)} (x,y,g,\dot x, \dot y, \omega) = (h_\theta (x,y)^t,  \tilde h_\theta g h, h_\theta  (\dot x, \dot y)^t, \omega )$.  The invariance of the kinetic energy and the constraints $D$ ensures that $\Psi$ restricts 
to an action on $\M$, that leaves the equations of motion invariant. 

The Lie algebra $\g$ of $G$ is isomorphic to $\bR{}\times\bR{3}$ with the infinitesimal generators 
$$
(1;{\bf 0})_Q = -y\partial_x+x\partial_y+X^R_3 \quad \mbox{and} \quad (0;{\bf e}_i)_Q = \alpha_i\,X_1^R+\beta_i\,X_2^R+\gamma_i\,X_3^R, \textrm{ for } i=1,2,3,
$$
where ${\bf e}_i$ denotes the $i$-th element of the canonical basis of $\bR{3}$ and, 
$\alpha = (\alpha_1,\alpha_2,\alpha_3)$, $\beta=(\beta_1,\beta_2,\beta_3)$ 
$\gamma=(\gamma_1,\gamma_2,\gamma_3)$ the rows of the matrix $g\in SO(3)$.  Observe that 
$(1;{\bf 0})_Q$ is an infinitesimal generator of the $SO(2)$-action and the others are
infinitesimal generators of the $SO(3)$-action. 
%The vertical distribution $V$ of the $G$-action on $Q$ has fibers
%\begin{equation}\label{eq:V-ball}
 % \begin{aligned}V_{q} = \textrm{span} & \left\{ U_0\,, \alpha_1\,X_1^R+\beta_1\,X_2^R+\gamma_1\,X_3^R\,,
 % \alpha_2\,X_1^R+\beta_2\,X_2^R+\gamma_2\,X_3^R\,, \right. \\
 % & \left. \alpha_3\,X_1^R+\beta_3\,X_2^R+\gamma_3\,X_3^R \right\}\,.\end{aligned}
%\end{equation}
We then underline that the $G$-symmetry satisfies the dimension assumption and it is proper and free whenever $(x, y) \neq (0, 0)$ (note that the rank of $V$ is 3 for $(x, y) = (0, 0)$ and it is 4 elsewhere, showing that the action is not even locally free). 

Let us denote by $\widetilde Q\subset Q$ and $\widetilde \M\subset \M$ the manifolds where the $G$-action is free, i.e. $(x,y)\neq(0,0)$.
The vertical distribution $S= D\cap V$ on $\widetilde{Q}$ has rank 2 with fibers
$$
S_{q} = \textrm{span}\{ Y_1 := -y Y_x + x Y_y, Y_2 := X_n \}.
$$
The bundle $\g_S\rightarrow Q$ has a global basis $\mathfrak{B}_{\g_S}$ of sections given by 
\[
  \mathfrak{B}_{\g_S}= \left\{ \xi_1 := \left(1; \frac{x}{r\,n_3},\frac{y}{r\,n_3},0\right), \xi_2 := (0; n\,g)   \right\}
\]
and we check that $(\xi_1)_Q = Y_1$ and $(\xi_2)_Q =Y_2$. Finally we observe that $\widetilde Q/G$ has dimension 1 ($\rho_{\tilde{Q}} : \tilde{Q}\to \tilde{Q}/G$ is given by $\rho_Q(x,y,g) = x^2 +y^2$) and hence the $G$-symmetry satisfies Conditions $(\cA 1)$-$(\cA 4)$ on $\widetilde Q$.

\medskip

\noindent {\bf The existence of horizontal gauge momenta.}    Using the basis \eqref{eq:TQ-basis} and the definition of $S$, we consider the decomposition 
$$
T\widetilde Q= H \oplus S \oplus W,
$$
where $W$ is a vertical complement of the constraints given by $W:= \textrm{span}\{Z_1,Z_2\}$ and $H:=S^{\perp}\cap D$ is generated by $X_0:= xY_x + yY_y$. As in Example \ref{Ex:Solids}, in this case, it is enough (and straightforward using that $n_3(x,y)$ is rotational invariant and that $[X_1^R,X_2^R] = -X_3^R$ for all cyclic permutations) to check that $\kappa([Y_1, Y_2],Y_1)= 0$ and $\kappa([Y_1,Y_2],Y_2) = 0$ to guarantee that the kinetic energy is strong invariant on $S$.  Finally, we also see that $\kappa(X_0,[Y_i,X_0]) = 0$, for $i=1,2$ . Therefore, following Theorem~\ref{T:Main}, the system admits two $G$-invariant (functionally independent) horizontal gauge momenta ${\mathcal J}_1$ and ${\mathcal J}_2$, 
showing that the first integrals obtained in \cite{Routh,hermans,zenkov1995,BMK2002,FGS2005}
can be obtained from the symmetry of the system as horizontal gauge momenta.  %The associated functions $\bar{\mathcal J}_1$ and $\bar{\mathcal J}_2$ on $\widetilde\M$ are exactly the first integrals $\cJ_1$ and $\cJ_2$ of Routh type defined on $\M/G$ studied in \cite{hermans,zenkov1995,BMK2002,FGS2005,sansonetto,FGS2009}). We will explain explicitly this relation in Remark \ref{R:Ball:Routh}.  

%We now show that also this system fits our approach, and in particular that
%the two `Routh' first integrals, $\cJ_1$ and $\cJ_2$ are indeed horizontal gauge momenta.

\medskip

\noindent {\bf The computation of the 2 horizontal gauge momenta.}
We now characterize the coordinate functions of the horizontal gauge symmetries written in the basis $\mathfrak{B}_{\g_S}$ on $\widetilde Q$.  That is, let us denote by 
$$
J_1:= {\bf i}_{Y_1} \Theta = -yp_x+xp_y \qquad \mbox{and} \qquad J_2:= {\bf i}_{Y_2} \Theta = p_n.
$$
Using the orbit projection $\rho_{\tilde Q}:\tilde{Q}\to \tilde{Q}/G$, a $G$-invariant function $f$ on $Q$ can be thought as depending on the variable $\tau=x^2+y^2$, i.e., $f = f(\tau)$. Following Theorem~\ref{T:Main}$(ii)$, a function $\mathcal{J}= f_1 J_1+ f_2J_2$ for $f_1,f_2\in C^\infty(Q)^G$ is a horizontal gauge momenta if and only if $(f_1,f_2)$ is a solution of the linear system of ordinary differential equations on $\widetilde Q/G$, 
\begin{equation}\label{Ex:Ball:ODE}
R \left( \!\!\begin{array}{c} f_1 \\ f_2 \end{array} \! \! \right)= \left(\!\!\begin{array}{c} \bar{X}_0(f_1) \\ \bar{X}_0(f_2) \end{array} \!\! \right) \quad \mbox{where} \quad R =   2\tau\left(  \begin{array}{cc} 0   & -2 \frac{rI}{E} n_3^2(2(\phi')^3 - \phi'') \\ \tfrac{A}{r}n_3^2     & 0 \end{array} \!\! \right ) 
\end{equation}
for $A = \phi' + 2\tau\phi''$ and $\bar{X}_0= T\rho_{\widetilde Q}(X_0)= 2\tau \tfrac{\partial}{\partial\tau}$. The matrix $R$ is computed using that $R = [\kappa|_S]^{-1} [N]$ 
where 
\begin{equation*}%\label{Ex:Ball:MatrixH} 
[N] = \frac{2I}{r}\tau\left( \!\!  \begin{array}{cc} 0   & -2\tau n_3^2(2(\phi')^3 - \phi'') \\ An_3^2     & 0 \end{array} \!\! \right ) \quad \mbox{and} \quad  [\kappa|_S]  = \left(\!\! \begin{array}{cc} \tfrac{E}{r^2}\tau  & 0  \\ 
0 & I \end{array} \!\! \right).
\end{equation*}
 
Since this system admits two independent solutions $f^1 = (f_1^1, f_2^2)$ and $f^2 = (f_1^2, f_2^2)$ on $\widetilde Q/G$, then the nonholonomic system admits two $G$-invariant horizontal gauge momenta $\mathcal{J}_1$, $\mathcal{J}_2$ defined on $\widetilde \M$ of the form
\begin{equation}\label{Ex:Ball:J1J2}
\mathcal{J}_1 = f_1^1 J_1 +  f_1^2 J_2 \qquad \mbox{and} \qquad \mathcal{J}_2 = f_1^2 J_1 +  f_2^2 J_2.
\end{equation}
recalling that $J_1 = -yp_x+xp_y$ and $J_2 = p_n$

\begin{remark}\label{R:Ball:Routh}  Let us denote by $\bar{\mathcal J}_1$, $\bar{\mathcal J}_2$ the functions on $\widetilde \M/G$ associated to \eqref{Ex:Ball:J1J2}.
\begin{itemize}
 \item[$(i)$] The (reduced) first integrals $\bar{\mathcal J}_1$, $\bar{\mathcal J}_2$ can be extended by continuity to the {\it differential space} $\M/G$ and thus $\mathcal{J}_1$, $\mathcal{J}_2$ are $G$-invariant functions on $\M$ (see \cite{FGS2005} for details) and in this case we say that the system admits $2=\textup{rank}(\g_S)$ horizontal gauge momenta.
 \item[$(ii)$] The system of differential equations \eqref{Ex:Ball:ODE} can be written as  
 $$
 R_1 f_2 = f_1' \qquad \mbox{and} \qquad R_2 f_1 = f_2',
 $$
 where $R_1 = R_1(\tau) = -2\tfrac{rI}{E} n_3^2(2(\phi')^3 - \phi'')$ and 
 $R_2 = R_2(\tau) = \tfrac{A}{r} n_3^2$.
 Hence  $\bar{\mathcal J}_1$, $\bar{\mathcal J}_2$ are first integrals of Routh type 
 found in \cite{hermans} (see also \cite{CDS,zenkov1995,BMK2002,sansonetto}) and shown to be horizontal gauge momenta in 
 \cite{FGS2005,FGS2009}.  
\end{itemize}
\end{remark}

\noindent {\bf Integrability and reconstruction.} The reduced integrability of this system was established in \cite{Routh}
and its complete broad integrability has been extensively studied in \cite{hermans,zenkov1995,BMK2002,FGS2005,FG2007}, using the existence of first integrals
$\cJ_1$ and $\cJ_2$, without relating their existence to the symmetry group. The symmetry
origin of $\cJ_1$ and $\cJ_2$ was announced in \cite{BGM96}, and then proved in \cite{sansonetto,FGS2009}.
Here we want to stress how Theorem~\ref{T:Main} can be applied and therefore the
reduced integrability of the system is ensured. That is, $\bar{\mathcal J}_1$, $\bar{\mathcal{J}}_2$, $H_\red$ are first integrals of the reduced dynamics $X_\red$ defined on the manifold $\tilde\M/G$ of dimension 4.  Moreover, as proved in \cite{hermans,zenkov1995} the
reduced dynamics is made of periodic motions or of equilibria, and hence, since the symmetry group
is compact, the complete dynamics is generically quasi-periodic on tori of dimension 3 
(see Theorem \ref{T:reconstruction} and \cite{hermans,FG2007}). 
Indeed one could can say more on the geometric structure of the phase space $\widetilde \M$
of the complete system, it is endowed with the structure of a fibration on tori of dimension at most 3 
(see \cite{FG2007} for a detailed study of the geometry
of the complete system on $\widetilde \cM$).

\medskip

\noindent {\bf Hamiltonization.} Even though the hamiltonization of this example has been studied in \cite{BalYapu19}, in this section we see the hamiltonization as a consequence of Theorem \ref{T:Main} and how the resulting Poisson bracket on $\widetilde \M/G$ depends on the linear system of ordinary differential equations \eqref{Ex:Ball:ODE}.  

By Theorem~\ref{T:BalYapu19},  the nonholonomic system is {\it hamiltonizable by a gauge transformation}; that is, on $\widetilde\M/G$ the reduced nonholonomic system is described by a Poisson bracket with 2-dimensional leaves given by the common level sets of the horizontal gauge momenta $\bar{\mathcal J}_1$, $\bar{\mathcal J}_2$, induced by \eqref{Ex:Ball:J1J2}, (recall that $\bar{\mathcal J}_i$ are the functions on $\widetilde\M/G$, such that $\rho^*(\bar{\mathcal J}_i) = {\mathcal J}_i$).

Following Theorem~\ref{T:FormulaB}, we compute the 2-form $B_{\mbox{\tiny{HGS}}}$, defining the dynamical gauge transformation, 
using the momentum equation \eqref{Ex:Ball:ODE}.  Since $dY^1|_D=0$, then 
\begin{equation}
   B_{\mbox{\tiny{HGS}}}  : =   \langle J, \mathcal{K}_\subW\rangle - p_1 R_{12} {\mathcal X}^0\wedge {\mathcal Y}^2 + p_2 R_{21}X^0\wedge {\mathcal Y}^1  + p_2 d\mathcal{Y}^2,
\end{equation}
where $\mathcal{X}^0 = \tau_\subM^*X^0$ and $\mathcal{Y}^i = \tau_\subM^*Y^i$. 
That is,  
\begin{equation*}
 \begin{split}
  \langle J, \mathcal{K}_\subW\rangle |_\C& =  M_1\, d\epsilon^1|_\C + M_2 \, d\epsilon^2|_\C, \\
  & = -\frac{Ir}{E(x^2+y^2)} \left( p_1 (\tfrac{1}{rn_3^2} + 2n_3^2A) {\mathcal X}^0\wedge {\mathcal Y}^2 + p_0 n_3(2\phi' n_3 + \tfrac{1}{r}{\mathcal Y}^1\wedge {\mathcal Y}^2 \right)|_\C,
 \end{split}
\end{equation*}
and using that $d{\mathcal Y}^2|_\C = \frac{(x^2+y^2)}{n_3} p_2 X^0\wedge {\mathcal Y}^1|_\C$ we obtain
\begin{equation}\label{Ex:Ball:B}
B_{\mbox{\tiny{HGS}}} =(x^2+y^2)p_2 ( \tfrac{1}{n_3}+2\tfrac{A}{r} n_3^2 )  {\mathcal X}^0\wedge {\mathcal Y}^1 + \tfrac{rI}{E} ( \tfrac{1}{rn_3} + 2\phi') (p_1      {\mathcal X}^0\wedge {\mathcal Y}^2 - p_0 n_3^2   {\mathcal Y}^1\wedge {\mathcal Y}^2).
\end{equation}

\begin{remark} Since the action is not free, $\M/G$ is a semialgebraic variety that consists in two strata: a singular 1-dimensional stratum corresponding 
to the points in which the action is not free; and the four dimensional regular stratum $\widetilde \M/G$ (where the action is free).   Moreover, analyzing the change of coordinates between ${\bf B}_{T^*Q}$ and ${\mathfrak B}_{T^*Q}$ we get 
$$
\tau= x^2 + y^2, \ p_0 = xp_x+ yp_y,\  p_1= -yp_x+xp_y, \ p_2 = p_n, 
$$
and adding $p_3 = p_x^2+p_ y^2$ we recover the coordinates used in \cite{hermans,FGS2005} 
on $\widetilde \M/G$. 
\end{remark}

\begin{remark}
Since the convexity of the function $\phi$ that parametrizes the surface $\Sigma$ is not strictly used, this example also describes the geometry and dynamics of a homogeneous ball rolling on surface of revolution such that its normal vector fields has $n_3\neq 0$. 
\end{remark}

%\subsubsection{Particular cases: the cylinder}

\subsection{Comments on the hypothesis of Theorem \ref{T:Main}: examples and counterexamples}\label{S:Comments}

Theorem~\ref{T:Main} shows that a nonholonomic system with symmetries satisfying certain hypotheses admits the existence of $k$ functionally independent
$G$-invariant horizontal gauge momenta. 
Next, assuming Conditions $(\cA 1)$-$(\cA 3)$, we study what may happen if the other hypotheses of 
Theorem \ref{T:Main} are not satisfied. In particular we study three cases: when the metric is not 
strong invariant,  when $\kappa(X_0, [X_0,Y])$ is different from zero, and finally when Condition 
$(\cA 4)$ is not verified (i.e., $\textup{dim}(Q/G)\neq 1$).  For each case we give examples 
and counterexamples to illustrate our conclusions.  

% 
% 
% In this section we investigate what may if some of the conditions of 
% Theorem~\ref{T:Main} are not satisfied, and we will supply examples
% for the cases under investigation. We assume conditions $(\cA_1)$, $(\cA2)$ and
% $(\cA3)$ are satisfied, and investigate what happens if one among conditions 
% $(\cA4)$ and $(\cA5)$ is not. (See the Remark~\ref{rmk:cA} for a discussion
% on assumption $(\cA_1)$, $(\cA2)$ and $(\cA3)$).

%\begin{remark}\label{rmk:cA}
%  \textcolor{blue}{\bf We need a remark on conditions $(\cA_1)$, $(\cA2)$ and $(\cA3)$.
%  We can put it here or when these assumptions are introduced.}
%\end{remark}

\subsubsection*{Analyzing the strong invariance condition and $\kappa(X_0, [X_0,Y])= 0$}

Consider a nonholonomic system $(\cM,\Omega_\subM|_\C,H_\subM)$ with a $G$-symmetry 
satisfying Conditions $(\cA 1)$-$(\cA 4)$. 
Suppose that $(f_1, ..., f_k)$ is a solution of the system of differential 
equations \eqref{Eq:ODESystem}, then, from \eqref{Eq:MomEq-Coord}, 
we observe that $\mathcal{J}= f_jJ_i$ is a horizontal gauge momentum if and only if  
$$
  f_i \kappa(X_0, [Y_i, X_0]) = 0 \quad \mbox{and} \quad   f_i (\kappa(Y_j,[Y_i,Y_l]) + \kappa(Y_l,[Y_i,Y_j)) = 0, \mbox{ for each } j,l.
$$  
for a $S$-orthogonal horizontal space $H$.
That is, in some cases, even if $\kappa(X_0, [X_0,Y_{i_0}])\neq 0$ for some $Y_{i_0} \in \Gamma(S)$ or the metric is not strong invariant,  we may still have a horizontal gauge momentum.

We now present two examples that show the main features of these phenomenon. % and we will then drive some generic conclusions

\noindent {\bf The metric is not strong invariant on $S$.} The following is a mathematical example, that has the property that the metric is not strong invariant, and it admits only 1 horizontal gauge momenta even though the rank of the distribution $S$ is 3.
Precisely, consider the nonholonomic system on the manifold $Q=\R^3\times SE(2)$ with coordinates $(u,v,x)\in \R^3$ and $(y,z,\theta)\in SE(2)$ with Lagrangian given by 
$$
L(q,\dot q) = \frac{1}{2}\left(u^2+ v^2+ \dot x^2 +\dot y^2 + \dot z^2 + \dot \theta^2+ 4(\sin\theta \, \dot z + \cos\theta\, \dot y) \dot \theta\right),
$$
and constraints 1-forms given by
$$
\epsilon^u = du-(1+\cos x)d\theta \qquad \mbox{and} \qquad \epsilon^v = dv- \sin x d\theta.
$$
The symmetry is given by the action of the Lie group $G=\R^2\times SE(2)$ defined, at each $(a,b;c,d,\beta)\in G$, by
$$
\Psi((a,b;c,d,\beta),(u,v,x,y,z,\theta)) = (u+a,v+b,x, h_\beta \left(\!\begin{smallmatrix*} y \\ z \end{smallmatrix*}\!\right) +  \left(\!\begin{smallmatrix*} c \\ d \end{smallmatrix*}\!\right) , \theta +\beta),
                                                                          $$
where $h_\beta$ is the $2\times 2$ rotational matrix of angle $\beta$.  The distribution $S=D\cap V$ is generated by the $G$-invariant vector fields $\{Y_\theta, Y_1,Y_2\}$ given by 
$$
Y_\theta := \partial_\theta  + (1+\cos x)\partial_u + \sin x\partial_v, \, Y_1:= \cos\theta\partial_y +\sin\theta\partial_z, \, Y_2:= -\sin\theta\partial_y+\cos\theta\partial_z,
$$
and $X_0=\partial_x$ generates $H=S^\perp\cap D$. It is straightforward to check that Conditions $(\cA 1)$-$(\cA4)$ are satisfied and that $\kappa(X_0, [X_0,Y])= 0$ for all $Y\in \Gamma(S)$. However, the metric is not strong invariant on $S$: $\kappa (Y_2,[Y_\theta, Y_1]) = 1$ and $\kappa(Y_\theta, [Y_1,Y_2]) = 0$. From \eqref{Eq:MomEq-Coord}, we can observe that $\mathcal{J} =2p_1 + p_\theta$ is the only horizontal gauge momentum of the system in spite of the rank of $S$ being 3 (where, as usual, $p_1 = {\bf i}_{Y_1}\Theta_\subM$ and $p_\theta = {\bf i}_{Y_\theta}\Theta_\subM$). 

\medskip

\noindent {\bf Dropping condition $\kappa(X_0, [X_0,Y])= 0$.} 
%The dropping of condition $\kappa(X_0, [X_0,Y])= 0$ may lead to different situations. 
%In fact, consider a nonholonomic system $(L,D)$ with a symmetry satisfying Conditions $\cA$ and with $\textup{rank}(H)=1$. Suppose that $(f_1, ..., f_k)$ is a solution of the system of differential equations \eqref{Eq:ODESystem}, then, from \eqref{Eq:ODESystem}, we observe that $\mathcal{J}= f_jJ_i$ is a horizontal gauge momentum if and only if  $f_i \kappa(X_0, [Y_i, X_0]) = 0$.  Therefore, in some cases, even if $\kappa(X_0, [X_0,Y_i])\neq 0$ for some $Y_i \in \Gamma(S)$ we may still have a horizontal gauge momentum.   
We illustrate with a multidimensional nonholonomic particle the different scenarios 
obtained when $\kappa(X_0, [X_0,Y]) \neq 0$ for a section $Y\in \Gamma(S)$ 
(see Table~\ref{Tab:MultiNhPart}). 

Consider the nonholonomic system on $\bR{5}$ with Lagrangian $L(q,\dot q) = \frac{1}{2}\dot q\cdot \kappa\,\dot q - V(x_1)$, where $\kappa$ 
is the  kinetic energy metric 
  \[\kappa=\begin{psmallmatrix*}
    1& 0 & 1 & 0 &1\\
    0& 1 & 0 & 0 &0\\
    1& 0 & 1 & 0 &0\\
    0& 0 & 0 & 1 &1\\
    1& 0 & 0 & 1 &1\\
  \end{psmallmatrix*},\]
  and with the nonintegrable distribution $D$ given, at each $q=(x_1,\ldots,x_5)\in \bR{5}$, by 
  \[\begin{aligned}
    D_q  =\textrm{span}\{ & D_1= f(x_1)\, \partial_{x_1}+b(x_1)\, \partial_{x_3}+ c(x_1)\,\partial_{x_4}\,,
    D_2 = h(x_1)\, \partial_{x_1}+g(x_1)\, \partial_{x_2}\,, \\
    & D_3 = d(x_1)\, \partial_{x_1}+j(x_1)\,\partial_{x_4}+l(x_1)\,\partial_{x_5}\}\,,
  \end{aligned}\]
  where $b(x_1),c(x_1),d(x_1),f(x_1),g(x_1),h(x_1),j(x_1),l(x_1)$ 
  are functions on $\bR{5}$ depending only on the coordinate $x_1$.
  The group $\bR{4}$ of translations along the $x_2$, $x_3$, $x_4$ and $x_5$ directions
  acts on the system and leaves both the Lagrangian and the nonholonomic constraints invariant. It is straightforward to see that this $G$-symmetry satisfies Conditions $(\cA 1)$-$(\cA4)$.
  The fiber of the distribution $S$ over $q\in Q$ is $S_q = \textrm{span} \{
  Y_1 := f(x_1) D_2 - h(x_1) D_1\,,   Y_2 := h(x_1) D_3 - d(x_1) D_2 \}$. Since the translational Lie group $\R^4$ is abelian then the kinetic energy is strong invariant on $V$ (see Example~\ref{Ex:StrongAbelian}).  
  The distribution $H = S^\perp \cap D$ is generated by the vector field $X_0 = \beta_1(x_1)\,D_1 + \beta_2(x_1)\, D_2 + \beta_3(x_1)\, D_3$, for $\beta_1$, $\beta_2$ and $\beta_3$ suitable functions (defined on $\bR{5}$ but depending only on the coordinate $x_1$).

  %{\small \[ \begin{aligned}
  %  \beta_1(x_1) & = c(x_1) (j(x_1) + l(x_1)) (d(x_1) g(x_1)^2 - h(x_1)^2 l(x_1)) - 
  %    f(x_1) g(x_1)^2 (j(x_1)^2 + 2 j(x_1) l(x_1)  \\ 
  %    & + l(x_1) (d(x_1) + l(x_1))) + 
  %    b(x_1) (d(x_1)^2 g(x_1)^2 + h(x_1)^2 (j(x_1) + l(x_1))^2)\\  
  %  \beta_2(x_1) & = - h(x_1) [b(x_1)^2 (j(x_1)^2 + 2 j(x_1) l(x_1) + l(x_1) (d(x_1) + l(x_1))) + 
  %    c(x_1) l(x_1) (c(x_1) d(x_1) \\ 
  %    & - f(x_1) (j(x_1) + l(x_1))) +
  %    b(x_1) (j(x_1) + l(x_1)) (-c(x_1) d(x_1) +  f(x_1) (j(x_1) + l(x_1)))] \\
  %  \beta_3(x_1) & = c(x_1) g(x_1)^2 (-c(x_1) d(x_1) + 
  %    f(x_1) j(x_1)) + (f(x_1) (c(x_1) + f(x_1)) g(x_1)^2 \\
  %    & + c(x_1)^2 h(x_1)^2) l(x_1) +  b(x_1)^2 (-d(x_1) g(x_1)^2 \\ 
  %    & +  h(x_1)^2 l(x_1)) - b(x_1) (d(x_1) f(x_1) g(x_1)^2 + c(x_1) h(x_1)^2 (j(x_1) + l(x_1))).
  % \end{aligned}\]}
  % By construction, the distribution $H$ and $S$ are $G$-invariant. 
 
 %  The two terms $\kappa(X_0,[Y_1,X_0])$ and $\kappa(X_0,[Y_2,X_0])$
 %  do not vanish.  Moreover without 
   %a specific choice of the functions $b(x_1),c(x_1),d(x_1),f(x_1)$, $g(x_1),h(x_1),j(x_1),l(x_1)$,
  % the momentum equation \eqref{Eq:MomEq-Coord} does not admit any solution. 
   %Therefore this is an example in which, even if $\textrm{rank} \,S >1$, the system does not admit any horizontal gauge momentum, and this happens because the terms    $\kappa(X_0,[Y_i,X_0])$ for all $i=1,\ldots, k$ do not vanish and the momentum equation does not admit solution.
   
   For particular choices of the functions $b(x_1),c(x_1),d(x_1),f(x_1),g(x_1),h(x_1),j(x_1),l(x_1)$ the two terms $\kappa(X_0,[Y_1,X_0])$ and $\kappa(X_0,[Y_2,X_0])$ may not vanish.  The computations and their expression are rather long    
   and were implemented with Mathematica\@.  The next table shows different situations that we obtain:
   
   {\small\begin{table}[ht] \label{Tab:MultiNhPart}\begin{center} 
\begin{tabular}{|c|c|}
\hline
 \multicolumn{2}{|c|}{multidimensional nonholonomic particle ($\textup{rank}(S) =2$)} \\
 \hline
behaviour of $\kappa(X_0, [X_0,Y])$ & {\bf{\scriptsize{ \begin{tabular}{c} $\sharp$ horizontal \\ gauge momenta \end{tabular} }}} \\
  \hline
%  &&&\\
%  Vertical Disk & $\bR{2}\times\bT{2}$ & 2 & \textcolor{red}{\bf 2} & \textcolor{blue}{\bf 2} \\
&\\
  $\kappa(X_0,[Y_1,X_0]) = 0 $  and  $\kappa(X_0,[Y_2,X_0]) \neq 0 $   & 0 \\
  \hline
  &\\
   $\kappa(X_0,[Y_1,X_0]) = 0 $  and $\kappa(X_0,[Y_2,X_0]) \neq 0 $   & 1  \\
  \hline
  &\\
  $\kappa(X_0,[Y_1,X_0]) \neq 0 $ and $\kappa(X_0,[Y_2,X_0]) \neq 0 $  & 0 \\
  \hline
\end{tabular}
%\caption{Nonholonomic systems and related horizontal gauge momenta with respect to the symmetry.}
\end{center}
\end{table}

}

 %  \begin{itemize}
 %    \item 1 non-vanishing term of the type $\kappa(X_0,[Y_i,X_0])$ and zero HGM
 %    \item 1 non-vanishing term of the type $\kappa(X_0,[Y_i,X_0])$ and one HGM
 %    \item 0 non-vanishing term of the type $\kappa(X_0,[Y_i,X_0])$ and zero HGM
 %  \end{itemize}

%\begin{center}
%\begin{tabular}{|p{4cm}|p{8cm}|p{2cm}|c|}
%  \hline
%  &&&\\
%  {\bf Pippo} & {\bf Pluto} & {\bf ???} & {\bf Paperino}\\
%  &&&\\
%  \hline
%  \hline
%  {\bf Topolino} &  &  &  \\  
%  \cline{2-4}                           
%  & &  &  \\
%  \cline{2-4}                          
%  &  &  & \\
%  \hline
%  \hline
%  {\bf Minny} &  &  & \\
%  \cline{2-4}                           
%  & &  &  \\
%  \cline{2-4}                          
%  &  & & \\
%  \hline
%  \hline
%\end{tabular}
%\end{center}   

\subsubsection*{Cases when Condition\,$(\cA 4)$ is not satisfied (or $\textup{rank}(H)\neq 1$)} \label{Ss:rankHnot1}

When Condition $(\cA4)$ is not verified, it is still possible to work with the {\it momentum equation} 
stated in Proposition~\ref{Prop:MomEq1}. Basically, for the case when $\textup{rank}(H) =0$ 
we still have $\textup{rank}(S)$ horizontal gauge momenta, while if $\textup{rank}(H)>1$ 
we cannot say anything. 

\noindent {\bf If $\textup{rank}(H)=0$.}  In this case, $TQ = V$ which means that $Q \simeq G$. That is, consider a nonholonomic system $(L,D)$ on a Lie group $G$ for which the left action is a symmetry of the system.  Since the only $G$-invariant functions are constant, we need to check that, for a basis $\mathfrak{B}_{\g_S} = \{\xi_1, ...,\xi_k\}$ of $\Gamma(\g_S)$, the momentum equation \eqref{Eq:MomEq} is satisfied only for constant functions $f_i = c_i$. 
In this case, since $X_\nh \in \Gamma(\V)$, the coordinate momentum equation \eqref{Eq:MomEq-Coord}, for $f \in C^\infty(Q)^G$, remains
\begin{equation*}
f_i v^lv^j \kappa(Y_j,[Y_i,Y_l]) = 0. 
\end{equation*}
The constant functions $f_i = c_i$ are $k$ (independent) solutions of the momentum equation if and only if the kinetic energy is strong invariant on $S$, and hence the sections of the basis $\mathfrak{B}_{\g_S}$ are horizontal gauge symmetries. % inducing the horizontal gauge momenta given by $J_{\xi_i}$. 
%Therefore,  we \textcolor{red}{recover the result in \cite[Prop.4]{FGS2008}, that says that for any chosen global $G$-invariant basis $\{\xi_1,...,\xi_k\}$ of $\Gamma(\g_S)$, the corresponding functions $J_i$ are $G$-invariant horizontal gauge momenta. }\marginpar{\textcolor{red}{Does the Prop.4 on \cite{FGS2008} say something about some extra property of $\kappa$? here we need strong invariance!!!!}\textcolor{blue}{No condition on $\kappa$ there}}

As illustrative examples, see the vertical disk and the Chaplygin sleigh in \cite{FGS2008} and \cite{Bloch} respectively.

\noindent {\bf If $\textup{rank}(H)>1$.} \ In this case we cannot assert the existence of a global basis of $H$. However, in some examples the horizontal space $H$ may admit a global basis which we denoted by $\{X_1,...,X_n\}$ for $n=\textup{rank}(H)$. 
In this case, we observe that the second summand of the momentum equation \eqref{Eq:MomEq-Coord} gives the condition 
\[
 \kappa(X_\alpha, [Y_i,X_\beta]) -  \kappa(X_\beta, [Y_i,X_\alpha])=0  \quad \mbox{for all  } \alpha, \beta =1,..., n
\]
and the third summand gives a system of {\it partial} differential equations whose solutions induce the horizontal gauge momenta.  As an illustrative example, we can work out the Chaplygin ball \cite{chapsphere,duistermaat2004}: this example has a $G$-symmetry so that $\textup{rank}(S)=1$ and $\textup{rank}(H)=2$ with a global basis (see e.g. \cite{GN2008,balseiro2014}). However,  working with the momentum equation \eqref{Eq:MomEq}, it is possible to show that the system admits 1 horizontal gauge momentum, recovering the known result in  \cite{chapsphere,duistermaat2004, BorisovMamaev2001}.

\appendix

\section{Appendix: Almost Poisson brackets and gauge transformations} \label{A:Hamiltonization}

\noindent{\bf Almost Poisson brackets}.  
An {\it almost Poisson bracket} on a manifold $M$ is a bilinear bracket $\{\cdot, \cdot\}: C^\infty(M)\times C^\infty(M) \to C^\infty(M)$ that is skew-symmetric and satisfies Leibniz identity (but does not necessarily satisfy Jacobi identity).  Due to the bilinear property, an almost Poisson bracket induces a bivector field $\pi$ on $M$ defined, for each $f,g\in C^\infty(M)$ by 
$$
\pi(df, dg) = \{f,g\}.
$$
The vector field $X_f:= \{\cdot, f\}$ is the {\it hamiltonian vector field of $f$}.  Equivalently, $X_f = - \pi^\sharp(df)$, where $\pi^\sharp: T^*M\to TM$ is the map such that for $\alpha, \beta\in T^*M$,  $\beta(\pi^\sharp(\alpha)) = \pi(\alpha, \beta)$. 
The {\it characteristic distribution} of the bracket $\{\cdot,\cdot\}$ is the distribution on $M$ generated by the hamitonian vector fields.  

An almost Poisson bracket $\{\cdot, \cdot\}$ is {\it Poisson} when the Jacobi identity is satisfied, i.e.,
$$
\{f,\{g,h\}\} + \{g,\{h,f\}\} + \{h,\{f,g\}\} = 0, \qquad \mbox{for  } f,g,h \in C^\infty(M).
$$
Equivalently, a bivector field $\pi$ is Poisson if and only if $[\pi, \pi]=0$ where $[ \cdot, \cdot ]$ is the Schouten bracket, see e.g. \cite{MarsdenRatiuBook}.  The characteristic distribution of a Poisson bracket is integrable and foliated by symplectic leaves.

\begin{definition}\label{Def:Twisted}\cite{SeveraWeinstein}
 An almost Poisson bracket $\{\cdot, \cdot\}$ on $M$ is {\it twisted Poisson} if there exists a closed 3-form $\Phi$ on $M$ such that, for each $f,g,h\in C^\infty(M)$
 $$
 \{ f, \{g,h\}\} +  \{ g, \{h,f\}\} +  \{ h, \{f,g\}\} = \Phi (X_f, X_g,X_h), 
 $$
 where $X_f, X_g,X_h$ are the hamiltonian vector fields of $f,g,h,$ with respect to $\{\cdot, \cdot\}$.  In other words, a bivector field $\pi$ on $M$ is twisted Poisson if $[\pi,\pi] = \frac{1}{2} \pi^\sharp(\Phi)$.
 \end{definition}

 \begin{remark} \label{R:RegularTwisted}
The characteristic distribution of a twisted Poisson bracket is integrable and it is foliated by almost symplectic leaves. 
Conversely, it was shown in \cite{BN2011}, that any regular almost Poisson bracket with integrable characteristic distribution is a twisted Poisson bracket. 
\end{remark}

A regular almost Poisson bracket $\{\cdot, \cdot\}$ on $M$ is determined by a 2-form $\Omega$ and a distribution $F$ defined on $M$ so that $\Omega|_F$ is nondegenerate. In fact,  for $f\in C^\infty(M)$,
 \begin{equation}\label{Def:RegBracket}
 X_f = \{\cdot , f\} \quad \mbox{if and only if} \quad {\bf i}_{X_f}\Omega|_F = df |_F,
 \end{equation}
 (actually, the bracket is determined by the nondegenerate 2-section $\Omega|_F$ on $M$). The distribution $F$ is the characteristic distribution of the bracket.  If $F$ is integrable, then $\{\cdot, \cdot\}$ is a (regular) twisted Poisson bracket by the 3-form $\Phi=d\Omega$ ($\Omega$ is not necessarily closed). A Poisson bracket has $F$ integrable and $\Omega$ closed. 

\medskip

\noindent{\bf Gauge transformations of a (regular) bracket by a 2-form}.

\begin{definition}\label{Def:GaugeTransf}\cite{SeveraWeinstein}
Consider a (regular) bracket $\{\cdot, \cdot \}$ on the manifold $M$ as in \eqref{Def:RegBracket} and a 2-form $B$ satisfying that $(\Omega + B)|_F$ is nondegenerate.  A {\it gauge transformation of $\{\cdot, \cdot \}$ by the 2-form $B$} defines a bracket $\{\cdot,\cdot\}_\B$ on $M$ given, at each $f\in C^\infty(M)$, by
$$
{\bf i}_{X_f}(\Omega +B)|_F = df|_F \quad \mbox{if and only if} \quad X_f = \{\cdot , f\}_\B.
$$
In this case, we say that the brackets $\{\cdot,\cdot\}$ and $\{\cdot,\cdot\}_\B$ are {\it gauge related} .
\end{definition}

\begin{remark} \begin{enumerate} \item[$(i)$] The brackets $\{\cdot,\cdot\}$ and $\{\cdot,\cdot\}_\B$ have the same characteristic distribution $F$. Therefore, if an almost Poisson bracket has a nonintegrable characteristic distribution, all gauge related brackets will be almost Poisson with a  nonintegrable characteristic distribution.
 \item[$(ii)$]  If the bracket $\{\cdot , \cdot\}$ is twisted Poisson by a 3-form $\Phi$, then the gauge related bracket $\{\cdot , \cdot\}_\B$ is twisted Poisson by the 3-form $(\Phi + dB)$.  Moreover, they share the characteristic foliation but the 2-form on each leaf $F_\mu$ changes by the term $B_\mu = \iota_\mu B$ for $\iota_\mu : F_\mu \to M$ the inclusion.
 
 \item[$(iii$)] The original definition of  a {\it gauge transformation} in \cite{SeveraWeinstein} was given on Dirac structures and then the 2-form $B$ does not need to satisfy the nondegenerate condition $(\Omega + B)|_F$.
\end{enumerate}
 \end{remark}

\begin{definition}\label{Def:Semi-basic}
Let $\tau: M\to P$ be a vector bundle and let $\alpha$ be a $k$-form on the manifold $M$.  We say that $\alpha$ is {\it semi-basic} with respect to the bundle $M\to P$ if 
$$
{\bf i}_X \alpha = 0 \quad \mbox{for all $X \in TM$ such that $T\tau (X) = 0$} .
$$
The $k$-form $\alpha$ is {\it basic} if there exists a $k$-form $\bar\alpha$ on $P$ such that $\tau^*\bar\alpha = \alpha$.  
\end{definition}

\begin{remark}\label{R:Semi-basic}
 Consider the canonical symplectic 2-form $\Omega_\subQ$ on $T^*Q$. If $B$ is a semi-basic 2-form with respect to the bundle $T^*Q\to Q$, then $\Omega_\subQ+B$ is a nondegenerate 2-form on $T^*Q$.  
\end{remark}

\noindent{\bf Symmetries}. Let us consider an almost Poisson manifold $(M, \{\cdot, \cdot \})$ given as in \eqref{Def:RegBracket} and  a Lie group $G$ acting on $M$ and leaving  $\{\cdot, \cdot \}$ invariant. Then on the reduced manifold $M/G$ there is an almost Poisson bracket $\{\cdot, \cdot\}_\red$ defined, at each $f,g\in C^\infty(M/G)$ by 
$$ 
\{f,g\}_\red \circ \rho = \{\rho^*f, \rho^*g\},
$$
where $\rho: M\to M/G$ is the orbit projection. 

If a $G$-invariant 2-form $B$ satisfies that $(\Omega + B)|_F$ is nondegenerate, then the gauge related bracket  $\{\cdot, \cdot \}_\B$ is $G$-invariant as well. Both brackets $\{\cdot, \cdot \}$ and $\{\cdot, \cdot \}_\B$ can be reduced to obtain the corresponding reduced brackets  $\{\cdot, \cdot \}_\red$ and $\{\cdot, \cdot \}_\red^\B$ on the quotient manifold $M/G$ as the diagram shows:
\begin{equation}\label{Diag:Hamilt}
\xymatrix{  (M, \{\cdot, \cdot \})  \ar[d]_{\mbox{\tiny{reduction} }} \ar[rr]^{\mbox{\tiny{gauge transf. by $B$} } }
&&  (M, \{\cdot, \cdot \}_\B) \ar[d]\\
 (M/G, \{\cdot, \cdot \}_\red) && (M/G, \{\cdot, \cdot \}_\red^\B)   }
\end{equation}

As it was observe in \cite{GN2008,BN2011}, the brackets  $\{\cdot, \cdot \}_\red$ and $\{\cdot, \cdot \}_\red^\B$ can have different properties. More precisely, they are not necessarily gauge related and hence one can be Poisson while the other not.   In fact,  $\{\cdot, \cdot \}_\red$ and $\{\cdot, \cdot \}_\red^\B$ are gauge related if and only if the 2-form $B$ is basic with respect to the principal bundle $M\to M/G$.

\section{Appendix: Some facts on reconstruction theory} \label{app:rec}
The reconstruction of the dynamics from reduced equilibria 
and reduced periodic orbits has been well studied in \cite{field1990,krupa}, 
when the symmetry
group is compact and in \cite{AM1997} in the non--compact case. 
In this subsection we shortly review 
the basic results of reconstruction theory in the simplest framework, 
of free and proper group actions. We consider 
a Lie group $G$ that acts freely and properly on a manifold $M$. The freeness
and properness of the action guarantee 
that the quotient space $M/G$ has the structure of a manifold and $\tau: M\longrightarrow M/G$
is a principal bundle with structural group $G$. 
Let $X$ be a $G$--invariant vector field on $M$, 
then there exists a vector field $\bar X$ on $M/G$, which is $\tau$--related to $X$.
We recall that a $G$--orbit $\cO_{m_0}=G\cdot m_0$, with $m_0\in M$,  is a {\it relative equilibrium}
for $X$, if it is invariant  with respect to the flow of $X$ and its projection
to the reduced space $M/G$ is an equilibrium of
the reduced dynamics $\bar X$.
Moreover a $G$--invariant subset $\cP$ of $M$ is called a {\it relative periodic orbit}
for $X$, if it is invariant by the flow and its projection to the quotient manifold  $M/G$ is a periodic orbit 
of $\bar X$.

Let $\cP$ be a relative periodic orbit and
$\gamma$ a curve in $\cP$. By the periodicity of the 
reduced dynamics, the integral curves of the complete system, that pass 
through $\gamma(0)$, returns periodically, with period $T>0$, 
to the $G$--orbit through $\gamma(0)$.
The freeness of the action of $G$ on $M$ guarantees that
$\forall \gamma$ in $\cP$ there exists a unique $p(\hat{\gamma})$ 
in $G$ such that 
\[
  \phi^X_T(\gamma)=\psi_{p(\hat{\gamma})}(\gamma)\,,
\]
where $\phi^X_T$ is the flow of $X$ at time $T$, $\psi_g$ is the action of $G$ on $M$, 
$\hat \gamma$ is the projection of $\gamma$ on $M/G$ with respect to $\tau$, and
the map $p:\cP\rightarrow G$, $\gamma\mapsto p=p(\hat{\gamma})$ 
is the so--called \emph{phase} \cite{FG2007}.
The phase $p$ is a piecewise smooth map, constant along the orbits of $X$
(i.e. $p\circ \phi^X_t=p$, $\forall t$) and it is equivariant with respect to conjugation, that is
$\quad p(h\cdot \gamma)= h\,p(\hat{\gamma})h^{-1}, \quad \forall h\in G, \forall \gamma\in \cP$. 
Then the following Theorem holds.

\begin{proposition}\cite{field1990,krupa,AM1997}\label{periodic_loops}
  Let $\cP$ be a relative periodic orbit of $X$ on $M$. 
%  and let $\bar \cP$ the projection of $\cP$ on $M/G$ with respect to $\tau$. 
  Then
  \begin{itemize}
    \item[i)] if the group $G$ is compact, the flow of $X$ in $\cP$ is  quasi--periodic 
    with at most $rank\, G+1$ frequencies; 
    \item[ii)] if $G$ is non--compact, the flow of $X$ in $\cP$ is either 
    quasi--periodic, or escaping.
  \end{itemize}
\end{proposition}

The non--compact case is the most frequent and also the most interesting, for example one could say more on which of the two behaviours of the dynamics, namely quasi-periodicity or escaping, is ``generic'' by studying the group $G$ (but this goes beyond our scopes, for more details see \cite{AM1997}).    

%\begin{remark}
%  Also in the case of Proposition~\ref{periodic_loops}, if $\xi$ is the element 
%  of the Lie algebra that generates the 
%  phase $p$ associated to a relative periodic orbit $\cP$, we can introduce the 
%  semi--algebraic sets $\mathfrak{g}_{\mathfrak{T}}$ and $\mathfrak{g}_{\mathfrak{R}}$, 
%  whose codimension determine the character, \textit{generic} or \textit{special},
%  of the reconstructed trajectories. 
%\end{remark}  

\begin{remark}
  In \cite{field1990,krupa,AM1997} reconstructions results are given from the point of view
  of Lie Algebras, while \cite{FPZ2020} develops a theory in terms of groups. 
  Moreover \cite{FPZ2020} investigates the structure of the copies of $\bR{}$
  and shows that one can define an intrinsic notion of a certain number of frequencies that
  gives rise to the idea that, in this case, the reconstructed dynamics `spirals' toward a certain direction.
\end{remark}


{\small
\begin{thebibliography}{99}

%\bibitem{AM}
%\newblock R. Abraham and J.~E. Marsden, 
%\newblock Foundations of Mechanics.
%\newblock Benjamin, Reading 1978.

\bibitem{agostinelli1956}
     \newblock C. Agostinelli,
     \newblock \emph{Nuova forma sintetica delle equazioni del moto di un sistema anolonomo ed esistenza di un integrale lineare nelle velocit\`a.}
     \newblock Boll. Un. Mat. Ital., \textbf{11} (1956), 1--9.

%\bibitem{Arnold}
%\newblock V.I. Arnold,
%\newblock Mathematical Methods of Classical Mechanics. 
%\newblock Graduate Texts in Mathematics {\bf 60} Springer-Verlag, New York,(1989).

\bibitem{AM1997}
\newblock P. Ashwin and I. Melbourne,
\newblock {\it Noncompact drift for relative equilibria and relative periodic orbits.}
\newblock Nonlinearity, {\bf 10} (1997), 595--616.

\bibitem{balseiro2014} 
\newblock P. Balseiro, 
\newblock{\it The Jacobiator of Nonholonomic Systems and the Geometry of Reduced Nonholonomic Brackets.}
\newblock Arch. Ration. Mech. Anal. {\bf 214}, ( 2014), 453--501.

\bibitem{balseiro2017} 
\newblock P. Balseiro, 
\newblock {\it Hamiltonization of Solids of Revolution Through Reduction.}
\newblock  J. Nonlinear Sci. {\bf 27} (2017), 2001--2035.

\bibitem{BalFer2015} 
\newblock P. Balseiro and O.E. Fernandez,
\newblock{\it Reduction of Nonholonomic Systems in Two Stages.} 
\newblock Nonlinearity, Volume 28 (2015) 2873-2912.

\bibitem{BN2011} 
\newblock P. Balseiro and L. Garc\'{i}a-Naranjo, 
\newblock{\it Gauge Transformations, Twisted Poisson Brackets and Hamiltonization of Nonholonomic Systems.} 
\newblock Arch. Ration. Mech. Anal. {\bf 205} (2012), 267--310.

\bibitem{BS2016}
\newblock P. Balseiro and N. Sansonetto,
\newblock{\it A Geometric Characterization of Certain First Integrals for Nonholonomic Systems with Symmetries.}
\newblock SIGMA {\bf 12} (2016), 018, 14pages.

\bibitem{BalYapu19} 
\newblock P. Balseiro and L.P. Yapu. 
\newblock{\it Conserved quantities and hamiltonization of nonholonomic systems.} 

%\bibitem{BC}
%\newblock L.M. Bates and R.H. Cushman, 
%\newblock {\em What is a completely integrable nonholonomic dynamical system?}
%\newblock Rep. Math. Phys. {\bf44} (1999), 29--35.

%\bibitem{bluebook}
%\newblock L.M. Bates and R.H. Cushman, 
%\newblock Global aspects of classical integrable systems. Second edition.
%\newblock Birkh\"auser Verlag, Basel, 1997.

\bibitem{BGM96}
\newblock L.M.~Bates, H.~Graumann and C.~MacDonnell,
\newblock {\it Examples of gauge conservation laws in nonholonomic systems.}
\newblock \emph{Rep. Math. Phys.}, \textbf{37} (1996), 295--308.

\bibitem{BS93}
\newblock L.M. Bates and J. \'Sniatycki, 
\newblock {\em Nonholonomic reduction.}
\newblock Rep. Math. Phys. {\bf32} (1993), 99--115.

%\bibitem{BS2016} 
%L. Bates and J. Sniatycki, 
%{\it Not quite hamiltonian reduction, ....}    

%\bibitem{benenti2007} 
%\newblock S. Benenti,
%\newblock \emph{A 'User-Friendly' Approach to the Dynamical Equations of Non-Holonomic Systems.}
%\newblock SIGMA, \textbf{3} (2007), 33 pp.

%\bibitem{benenti2008} 
%\newblock S. Benenti,
%\newblock \emph{A general method for writing the dynamical equations of	nonholonomic systems with ideal constraints.}
%\newblock{Regul. Chaotic Dyn., \textbf{13} (2008), 283--315.}

\bibitem{Bloch}
\newblock A.M. Bloch, 
\newblock Nonholonomic Mechanics and Controls.
\newblock Interdisciplinary Applied Mathematics {\bf 24}, Systems and Control.
(Springer--Verlag, New York, 2003).

\bibitem{BKMM}
     \newblock A.M. Bloch, P.S. Krishnaprasad, J.E. Marsden, R.M. Murray,
     \newblock {\it Nonholonomic mechanical systems with symmetry.}
     \newblock Arch. Rational Mech. Anal. \textbf{136}  (1996), 21--99.

%\bibitem{BMZ}
%     \newblock A.M. Bloch, J.E. Marsden and D.V. Zenkov,
%     \newblock \emph{Quasivelocities and symmetries in non--holonomic systems.}
%     \newblock Dynamical Systems \textbf{24} No. 2 (2009), 187--222.

\bibitem{bogoyavlenskij} % (MR1643501) [10.1007/s002200050412]
\newblock O. I. Bogoyavlenskij,
\newblock {\it Extended integrability and bi-{H}amiltonian systems.}
\newblock Comm. Math. Phys., \textbf{196} (1998), 19--51.

\bibitem{BKK2015} 
\newblock{A.V. Bolsinov, A.A. Kilin, and A.O. Kazakov,}
\newblock{\it Monodromy as an obstruction to Hamiltonization: pro or contra?} 
\newblock{J. Geom. Phys. {\bf 87} (2015), 61--75.}

%\bibitem{BBM2011} 
%\newblock A.V. Bolsinov, A.V. Borisov, and I.S. Mamaev
%\newblock{\it Hamiltonization of Non-Holonomic Systems in the Neighborhood of Invariant Manifolds.}
%\newblock{Regul Chaotic Dyn., {\bf 16} (2011), 443--464.} 

%\bibitem{BM2002} 
%\newblock A.V. Borisov and I.S. Mamaev,
%\newblock {\it The rolling motion of a rigid body on a plane and a sphere, Hierarchy of dynamics.} 
%\newblock Regul. Chaotic Dyn., {\bf 7}  (2002), 201--219. 

\bibitem{BorisovMamaev2001} 
\newblock A.~V. Borisov and  I.~S. Mamaev,  
\newblock {\it Chaplygin's ball rolling problem is Hamiltonian.}
\newblock  Math. Notes, {\bf 70} (2001), 793--795.

\bibitem{BMK2002} 
\newblock A.V. Borisov,  I.S. Mamaev and A.A. Kilin,
\newblock {\it Rolling of a ball on a surface. New integrals and hierarchy of dynamics.} 
\newblock Regul. Chaotic Dyn., {\bf 7}  (2002), 177--200. 

%\bibitem{BorisovMamaev2008} 
%\newblock A.~V. Borisov and  I.~S. Mamaev,  
%\newblock {\it Conservation laws, Hierarchy of dynamics and explicit integration of nonholopnomic systems.}
%\newblock Regul. Chaotic Dyn., {\bf 13}  (2008), 443--489. 

%\bibitem{Bullo-Lewis}
%\newblock F. Bullo and A.D. Lewis, 
%\newblock{\em Geometric control of mechanical systems. Modeling, analysis, and design for simple mechanical control systems.} 
%\newblock Texts in Applied Mathematics {\bf 49} (Springer-Verlag, New York, 2005).

%\bibitem{byers1}
%N. Byers, {\em The life and time of Emmy Noether: Contributions of
%Emmy Noether to particle physics}. arXiv:hep-th/9411110v2 (1994)
%
%\bibitem{byers2}
%N. Byers, {\em E. Noether's discovery of the deep connection between
%symmetries and conservation laws.} In {\em The heritage of Emmy Noether
%(Ramat-Gan, 1996)}, Israel Math. Conf. Proc. {\bf 12}, 67--81, Bar-Ilan
%Univ., Ramat Gan, 1999.

\bibitem{CdLdDM}
\newblock F. Cantrijn, M. de Leon, M. de Diego and J. Marrero,
\newblock \emph{Reduction of nonholonomic mechanical systems with symmetries},
\newblock Rep. Math. Phys. \textbf{42} (1998), 25--45.

\bibitem{CCdLdD}
\newblock F. Cantrijn, J. Cort\'es, M. de Leon and M. de Diego,
\newblock {\it On the geometry of generalized Chaplygin systems},
\newblock Math. Proc. Cambridge Phil. Soc. {\bf 132} (2002), 323--351.

\bibitem{chapsphere} 
\newblock S.~A. Chaplygin,  
\newblock {\it On a ball's rolling on a horizontal plane. }
\newblock Regul.  Chaotic Dyn., {\bf 7} (2002), 131--148; original paper in Mathematical Collection of the Moscow
Mathematical Society, {\bf 24} (1903), 139--168.

\bibitem{cortes}
\newblock J. Cort\'es Monforte, 
\newblock Geometric, control and numerical aspects of nonholonomic systems.
\newblock Lecture Notes in Mathematics {\bf 1793} (Springer-Verlag, Berlin, 2002).

\bibitem{cushman1998} 
\newblock R.Cushman, 
\newblock {\it Routh's sphere.} 
\newblock Rep. on Math. Phys. {\bf 42} (1998), 42--70.

\bibitem{CKSB}
\newblock R. Cushman, D. Kemppeinen, J. \'Sniatycki and L. M.~Bates,
\newblock {\it Geometry of nonholonomic constraints.}
\newblock Rep. Math. Phys., \textbf{36} (1995), 275--286.

%\bibitem{CB}
%\newblock R.H. Cushman and L.M. Bates,
%\newblock Global Aspects of Classical Integrable Systems.
%\newblock Birkh\"auser/Springer, Basel, 2015.

%\bibitem{CD2001}
%\newblock R.H. Cushman and J.J. Duistermaat,
%\newblock {\it Non--Hamiltonian Monodromy}.
%\newblock J. Differential Equations {\bf 172} (2001), 42--58.

\bibitem{CDS}
\newblock R. Cushman, J.J. Duistermaat and J. \'Sniatycki,
\newblock Geometry of Nonholonomically Constrained Systems. 
\newblock Advanced Series in Nonlinear Dynamics {\bf 26}, Singapore: World Scientific, 2010.

%\bibitem{CS2007}
%\newblock R. Cushman and J. \'Sniatycki, 
%\newblock {\it Non--holonomic reduction of symmetries, constraints, and integrability.}
%\newblock Regul. Chaotic Dyn. {\bf 12} (2007), 615--621.

%\bibitem{dLCdDM} 
%\newblock M. de Leon, J. Cortes, M. de Diego and S. Mart\'{i}nez, 
%\newblock {\it Introduction to Mechanics and Symmetry.} 
%\newblock In: I. Bajo and E. Sanmartin (eds), {\it Recent Advances in Lie Theory}.
%Research and Exposition in Mathematics Series {\bf 25}, 305-332 (Heldermann Verlag, 2002),
%
%\bibitem{dLMMdD1997}
%\newblock M. de Le\'on, J.C. Marrero and D. Mart\'{\i}n de Diego,
%\newblock {\it Mechanical systems with nonlinear constraints.}
%\newblock Int. J. Th. Phys. \textbf{36} (1997), 979--995.

%\bibitem{DD1987}
%\newblock P. Dazord and T. Delzant,
%\newblock{\it Le probleme general des variables actions-angles.}
%\newblock J. Diff. Geom. {\bf 26} (1987), 223--251.

%\bibitem{duistermaat1980}
%\newblock J.J. Duistermaat, 
%\newblock {\it On global action-angle coordinates.}
%\newblock Comm. Pure Appl. Math. {\bf 33} (1980), 687-706.

\bibitem{duistermaat2004}
\newblock J.J. Duistermaat, 
\newblock {\it Chaplygin's sphere.} 
\newblock arXiv:math/0409019 (2004).


%\bibitem{fasso98}
%\newblock F. Fass\`o,
%\newblock {\it Quasi--periodicity of motions and complete integrability of Hamiltonian systems.}
%\newblock Ergodic Theory and Dynamical Systems {\bf 18} (1998), 1349--1362.

%\bibitem{fasso}
%\newblock F. Fass\`o,
%\newblock {\it Notes on Finite Dimensional Integrable Hamiltonian Systems.}
%\newblock Unpublished notes available at http://www.math.unipd.it/~fasso/research/papers/sc.pdf

\bibitem{FG2002}
\newblock F. Fass{\` o} and A. Giacobbe,
\newblock {\it Geometric structure of "broadly integrable" Hamiltonian systems.}
\newblock J. Geom. Phys. \textbf{44} (2002), 156--170.

\bibitem{FG2007}
\newblock F. Fass\`o and A. Giacobbe, 
\newblock{\it Geometry of Invariant Tori of Certain Integrable Systems with Symmetry and an Application
to a Nonholonomic System.}
\newblock SIGMA {\bf 3} (2008), 051.

\bibitem{FGS2005}
\newblock F. Fass{\` o}, A. Giacobbe and N. Sansonetto,
\newblock {\it Periodic flows, rank--two Poisson structures, and Nonholonomic systems.}
\newblock Reg. Ch. Dyn. \textbf{10} (2005), 267--284.

\bibitem{FGS2008}
\newblock F. Fass{\` o}, A. Giacobbe and N. Sansonetto,
\newblock \emph{Gauge conservation laws and the momentum equation in nonholonomic mechanics.}
\newblock Rep. Math. Phys. \textbf{62} No. 3 (2008), 345--367.

\bibitem{FGS2009}
\newblock F. Fass{\` o}, A. Giacobbe, N. Sansonetto,
\newblock \emph{On the number of weakly Noetherian constants of motion of nonholonomic systems.}
\newblock J. Geom. Mech. \textbf{1} (2009) 389--416.
     
\bibitem{FS2010}
\newblock F. Fass{\` o} and N. Sansonetto,
\newblock {\it An Elemental Overview of the Nonholonomic Noether Theorem.}
\newblock Int. J. Geom. Methods Mod. Phys. {\bf 6} (2010), 1343--1355.

\bibitem{FGS2012}
\newblock F. Fass{\` o}, A. Giacobbe, N. Sansonetto,
\newblock {\it Linear weakly Noetherian constants of motion are horizontal gauge momenta.}
\newblock J. Geom. Mech. \textbf{4} (2012) 129--136.

\bibitem{FRS2007}
\newblock F. Fass{\` o}, A. Ramos, N. Sansonetto,
\newblock {\it The reaction-annihilator distribution and the nonholonomic Noether theorem for lifted actions.}
\newblock Reg. Ch. Dyn. \textbf{12} (2007), 579--588.

\bibitem{FPZ2020}
\newblock F. Fass\`o. S. Passarella, and M. Zoppello, 
\newblock{\it Control of locomotion systems and dynamics in relative periodic orbits.}
\newblock To appear in J. Geom. Mech. doi:10.3934/jgm.2020022.

\bibitem{fedorov} %(MR1720911)
\newblock Y. N. Fedorov,
\newblock {Systems with an invariant measure on Lie groups},
\newblock In \emph{Hamiltonian systems with three or more degrees of freedom (S'Agar\`o, 1995)}, 350--356, NATO Adv. Sci. Inst. Ser. C Math. Phys. Sci., 533 Kluwer, Dordrecht, 1999.

\bibitem{field1990} 
M.J. Field, 
{\it Equivariant dynamical systems.} 
Trans. Am. Math. Soc. {\bf 259} (1990), 185--205.  

%\bibitem{field1991}
%\newblock M. Field,
%\newblock {\it Local structure of equivariant dynamics, in Singularity Theory and Its Applications, Part II .}
%\newblock Lecture Notes in Math., {\bf 1463} Springer--Berlin (1991), 142--166.

%\bibitem{Gallier} 
%\newblock J.~Gallier,
%\newblock Notes on Differential Geometry and Lie Groups. 
%\newblock Springer, Geometry and Computing Series, 2020.

\bibitem{GN2008}  
\newblock L. Garcia-Naranjo
\newblock \emph{Reduction of  almost Poisson brackets and hamiltonization of the Chaplygin sphere.}
\newblock  Disc. and Cont. Dyn. Syst. Series S, {\bf 3}, (2010), %no. 1, 
37--60.

\bibitem{GNMontaldi}  
\newblock L.C. Garc\'ia-Naranjo and J. Montaldi
\newblock \emph{Gauge momenta as Casimir functions of nonholonomic systems.}
\newblock  Arch Rational Mech Anal (2018), 228 (2), pp 563-602

\bibitem{hermans}
\newblock J. Hermans,
\newblock {\it A symmetric sphere rolling on a surface.}
\newblock Nonlinearity, {\bf 8} (1995), 493--515.

\bibitem{IdLMM} 
\newblock A. Ibort, M. de Le\'on, J. C. Marrero, D. Mart\'\i n de Diego
\newblock {\it Dirac brackets in constrained dynamics.}
\newblock Fortschritte der Physik, Vol.47 (1999), 459--492.


\bibitem{iliev1}
\newblock Il. Iliev and Khr. Semerdzhiev,
\newblock \emph{Relations between the first integrals of a nonholonomic mechanical system and of the corresponding system freed of constraints.}
\newblock J. Appl. Math. Mech. \textbf{36}  (1972), 381--388.

\bibitem{iliev2}
\newblock Il. Iliev,
\newblock \emph{On first integrals of a nonholonomic mechanical system.}
\newblock J. Appl. Math. Mech. \textbf{39}  (1975), 147--150.

%\bibitem{jovanovic}
%\newblock B. Jovanovic, 
%\newblock {\it Symmetries and Integrability.} 
%\newblock Publications de l'Institut Math\'ematique (Beograd), {\bf 84} (2008), 1--36.

%\bibitem{koiller}
%\newblock J. Koiller,
%\newblock \emph{Reduction of some classical nonholonomic systems with symmetry},
%\newblock Arch. Rat. Mech. An. {\bf 118} (1992), 113--138.

%\bibitem{kosmann1985}
%\newblock Y. Kosmann-Schwarzbach, 
%\newblock {\em Sur les th\'eor\`emes de Noether.} 
%\newblock G\'eom\'etrie et physique (Marseille-Luminy, 1985). Travaux en Cours {\bf 21}, 147--160 (Hermann, Paris, 1987).
%
%\bibitem{kosmann2006}
%\newblock Y. Kosmann-Schwarzbach, 
%\newblock {\em Les th\'eor\`emes de Noether. Invariance et lois de conservation au XX$\sp{\rm e}$ si\`ecle.}
%\newblock {\'Editions de l'\'Ecole Polytechnique} (Palaiseau, 2006).

%\bibitem{KK1978}
%\newblock V.V. Kozlov and N.N. Kolesnikov, 
%\newblock {\it On theorems of dynamics.}  
%\newblock J. Appl. Math. Mech. {\bf 42 } (1978), 28--33. 

\bibitem{krupa}
M. Krupa,
{\it Bifurcations of relative equilibria.} 
SIAM J. Math. Anal. {\bf 21} (1990), 1453--86.

%\bibitem{LM94}
%\newblock A. Lewis and R.M. Murray,
%\newblock {\it Variational principles in constrained systems: theory and experiments.}
%\newblock  Intern. J. Nonlinear Mech {\bf 30} (1994), 793--815.


\bibitem{marle95}
\newblock C.-M. Marle,
\newblock {\it Reduction of constrained mechanical systems and stability of relative equilibria.}
\newblock Comm. Math. Phys. {\bf 174} (1995), 295--318.

\bibitem{Marle98}
\newblock C.-M. Marle,
\newblock {\it Various approaches to conservative and nonconservative nonholonomic systems.}
\newblock Rep. Math. Phys. {\bf 42} (1998), 211--229.

\bibitem{marle2001}
\newblock C.-M. Marle,
\newblock \emph{On symmetries and constants of motion in Hamiltonian systems with nonholonomic constraints},
\newblock  In ``Classical and quantum integrability'' (Warsaw, 2001), 223--242, Banach Center Publ.~{\bf 59} 2001, Polish Acad. Sci. Warsaw (2003),  223--242

\bibitem{MarsdenRatiuBook}
\newblock J.E. Marsden and T.S. Ratiu,
\newblock Introduction to Mechanics and Symmetry, 2nd ed.
\newblock Texts in Appl. Math. New York, Springer, {\bf 17}, 1999.

\bibitem{milnor1976}
\newblock J. Milnor,
\newblock \emph{Curvatures of Left Invariant Metrics on Lie Groups}.
\newblock  Advances in Mathematics \textbf{21} (1976), 293--329.

%\bibitem{MF1978}
%\newblock A.S. Mischenko and A. T. Fomenko,
%\newblock {\it Generalized Liouville method of integration of Hamiltonian systems.}
%\newblock Funct. Anal. Appl. \textbf{12} (1978), 113--121.

\bibitem{Naber}
\newblock G.L. Naber,
\newblock Topology, Geometry and Gauge Fields. Applied Mathematical Sciences, {\bf 141}.
\newblock  Springer, Heidelberg, 1991.

\bibitem{Nash}
\newblock Ch. Nash,
\newblock Differential Topology and Quantum Field Theory.
\newblock Academic Press, New York, 1991.

\bibitem{NF}
\newblock Ju.I. Neimark and N.A. Fufaev,
\newblock Dynamics of Nonholonomic Systems. 
\newblock Translations of Mathematical Monographs {\bf 33} (AMS, Providence, 1972).

%\bibitem{nekhoroshev}
%\newblock N. N. Nekhoroshev,
%\newblock {\it Action--angle variables and their generalizations.}
%\newblock Trans. Moskow Math. Soc. \textbf{26} (1972), 181--198.

%\bibitem{noether}
%\newblock E. Noether, 
%\newblock {\em  Invarianten beliebiger Differentialausdr\"ucke.}
%\newblock Nachr. Akad. Wiss. G\"ottingen Math.-Phys. Kl. II 1918, 235--257. English translation: {\em Invariant variation problems}, Transport Theory Statist. Phys. {\bf 1}  (1971), 186--207.

%\bibitem{OR}
%\newblock J.P. Ortega and T.S. Ratiu,
%\newblock Momentum maps and Hamiltonian reduction.
%\newblock Birkh\"auser, Boston 2004.

\bibitem{OLMB}
\newblock J. Ostrowski, A. Lewis, R. Murray and J. Burdick
\newblock {\it Nonholonomic mechanics and locomotion: the snakeboard example.}
\newblock Proceedings of the 1994 IEEE International Conference on Robotics and Automation.


%\bibitem{parasiuk1}
%\newblock I.O. Parasiuk, 
%\newblock {\it Co--isotropic invariant tori of Hamiltonian systems in the quasiclassival theory of the motion of a conduction electron.} 
%\newblock Ukrainian Math J. {\bf 42} (1990), 308--312.
%
%\bibitem{parasiuk2}
%\newblock I.O. Parasiuk, 
%\newblock {\it Variables of action--angle type on symplectic manifolds stratifies by coisotropic tori.} 
%\newblock Ukrainian Math J. {\bf 45} (1993), 85--93.
%
%\bibitem{parasiuk3}
%\newblock I.O. Parasiuk, 
%\newblock {\it Reduction and coisotroopic invariant tori of Hamiltonian systems with non--Poisson commutative symmetries. II.} 
%\newblock Ukrainian Math J. {\bf 46} (1994), 991--1002.

%\bibitem{Pars}
%\newblock L. Pars, 
%\newblock A Treatise on Analytical Dynamics. 
%\newblock Heinemann, New York, 1965.

\bibitem{Routh}
\newblock E.J. Routh, 
\newblock Treatise on the Dynamics of a System of Rigid Bodies (Advanced Part). 
\newblock Dover, New York, 1955.

\bibitem{RS}
\newblock G. Rudolf and M. Schmdt, 
\newblock {Differential Geometry and Mathematical Physics. Part 2.} 
\newblock Theoretical and Mathematical Physics Series, 2017.

\bibitem{sansonetto}
\newblock N. Sansonetto, 
\newblock {First integrals in nonholonomic systems.} 
\newblock Ph.D. thesis, Universit\`a degli Studi di Padova.

\bibitem{SVN}
\newblock D. Sepe and S. Vu Ngoc, 
\newblock {Integrable systems, symmetries, and quantization.} 
\newblock Lett. Math. Phys. {\bf 108} (2018), 499--571.

\bibitem{SeveraWeinstein} 
\newblock P. \v{S}evera, A. Weinstein, 
\newblock {\it Poisson Geometry with a 3-form Background.}
\newblock Progress of Theoretical Physics, Vol. 144 (2001), 145-154.

\bibitem{Sniatycki98}
\newblock J. Sniatycki,
\newblock {\it Nonholonomic Noether theorem and reduction of symmetries.}
\newblock Rep. Math. Phys. \textbf{42} (1998), 5-23.

\bibitem{ShaftMashke}
\newblock A. Van der Shaft and B.M. Mashke, 
\newblock{\it On the Hamiltonian formulation of nonholonomic mechanical systems.}
\newblock Rep. Math. Phys. {\bf 34} (1994), 225-233. 

%\bibitem{winternitz2004}
%\newblock P. Winternitz, 
%\newblock{\it Superintegrable Systems in Classical and Quantum Mechanics.}
%\newblock{In In: Shabat A.B., Gonz\'alez-L\'opez A., Ma\~nas M., Mart\'{i}nez Alonso L., Rodr\'{i}guez M.A. 
%(eds) New Trends in Integrability and Partial Solvability. NATO Science Series 
%(Series II: Mathematics, Physics and Chemistry), vol 132. Springer, Dordrecht.}

\bibitem{zenkov1995}
\newblock D.V. Zenkov, 
\newblock{\it The geometry of the Routh problem.}
\newblock J. Nonlinear Sci. {\bf 5} (1995), 503-519. 

\bibitem{zenkov2003}
\newblock D.V. Zenkov,
\newblock \emph{Linear conservation laws of nonholonomic systems with symmetry.}
\newblock In ``Dynamical systems and differential equations'' (Wilmington, NC, 2002),  
Discrete Contin. Dyn. Syst. suppl. (2003), 967--976.

%\bibitem{zung2006}
%\newblock N,T. Zung,
%\newblock {\it Torus action and integrable systems, in Topological Methods in the Theory of Integrable Systems.}
%\newblock  Editors A.V. Bolsinov, A.T. Fomenko and A.A. Oshemkov, Cambridge Scientific Publications  (2006), 289-328, math.DS/0407455.

\bibitem{zung2016}
\newblock N.T. Zung,
\newblock {\it Geometry of Integrable non--Hamiltonian Systems, Geometry and Dynamics of Integrable Systems. Advanced Courses in Mathematics, CRM Barcelona.}
\newblock Editors V. Matveev and E. Miranda, Birkh\"auser (2016), 85-135.

%\bibitem{zung2017}
%\newblock N,T. Zung,
%\newblock{\it A conceptual approach to the problem of action-angle variables.}
%\newblock arXiv:1706.08859

%\bibitem{zung-minh2013}
%\newblock N.T. Zung, N.V. Minh, 
%\newblock{\it Geometry of integrable dynamical systems on 2-dimensional surfaces.} 
%\newblock Acta Math. Vietnam. {\bf 38} (2013), 79-106.

\end{thebibliography}
}
\end{document}